%% file: main.tex
\documentclass[a4paper,11pt,reqno]{amsart}

\input{preamble}

\title[Fundamental Gaps for the Dirichlet \(p\)-Laplacian]
{Fundamental Gaps for the Dirichlet \(p\)-Laplacian on Convex Domains}

\author{Rui Chen}
\address{
School of Mathematical Sciences, Fudan University,
Shanghai 200433, China
}
\address{
Brandenburg University of Technology Cottbus--Senftenberg,
03046 Cottbus, Germany
}
\email{chenrui23@m.fudan.edu.cn}

\author{Daniel Hauer}
\address{
Brandenburg University of Technology Cottbus--Senftenberg,
Platz der Deutschen Einheit 1, 03046 Cottbus, Germany
}
\address{
School of Mathematics and Statistics,
The University of Sydney,
NSW 2006, Australia
}
\email{daniel.hauer@b-tu.de}



\begin{document}

\begin{abstract}
We study the fundamental gap of the Dirichlet \(p\)-Laplacian on
bounded convex domains. We first prove the sharp one-dimensional
estimate
\[
\Gamma_p(I_D,V)
\geq
(p-1)(2^p-1)
\left(\frac{\pi_p}{D}\right)^p,
\qquad I_D:=(-\frac{D}{2},\frac{D}{2})
\]
for all $1<p<\infty$ and convex potentials $V$, with equality if and only if \(V\) is constant.
An analysis of collapsing convex domains then shows that the case
\(p=2\) is exceptional: for \(p\neq2\), the higher-dimensional gap is
not governed by the sharp one-dimensional constant.

For \(V\equiv0\), we introduce the optimal diameter-normalized gap
\[
\mathcal G_{p,N}
:=
\inf_{\Omega\subset\mathbb R^N}
\operatorname{diam}(\Omega)^p
\Gamma_p(\Omega,0),
\]
where the infimum is taken over bounded convex domains. We prove that
\[
\mathcal G_{2,N}=3\pi^2,
\qquad
\mathcal G_{p,N}=0
\quad\text{for }1<p<2,\ N\geq2,
\]
whereas for every \(p>2\) the constant \(\mathcal G_{p,N}\) is
positive and the infimum is attained. We obtain explicit
dimension-dependent bounds and determine its principal asymptotic
regimes:
\[
\mathcal G_{p,N}\asymp_p N^{p-2}
\quad\text{as }N\to\infty,
\]
and, for every fixed \(N\),
\[
\lim_{p\downarrow2}\mathcal G_{p,N}=3\pi^2,
\qquad
\lim_{p\to\infty}\mathcal G_{p,N}^{1/p}=4.
\]
We also determine the sharp collapse at \(p=2\): for \(N\geq2\), exact
minimizing domains degenerate as \(p\downarrow2\), and every limiting
normalized ground-state measure is supported on a unit segment and,
up to rigid motions, is exactly
\[
2\cos^2(\pi t)\,dt
\quad\text{on}\quad
\left(-\frac12,\frac12\right).
\]

The proofs develop several tools for the nonlinear problem. We establish
log-concavity of the first \(p\)-eigenfunction with convex potential,
a degenerate weighted Poincar\'e inequality adapted to the ground-state
measure, and a new lower bound for the fundamental gap in terms of the first eigenvalue and the associated weighted Poincar'e constant. These estimates also provide the
compactness and equality analysis needed for the characterization of
the collapsing minimizers.
\end{abstract}

\maketitle

%\tableofcontents

\section{Introduction and Main Results}

For \(p=2\), the fundamental gap is the difference between the first two
eigenvalues of the Dirichlet Schr\"odinger operator
\[
-\Delta+V.
\]
It represents the separation between the ground-state and first excited-state
energies and also determines the leading rate of convergence to equilibrium;
see van den Berg \cite{vanDenBerg1983} and Smits \cite{Smits1996}. Let
\(\Omega\subset\mathbb R^N\) be a bounded convex domain of diameter \(\operatorname{diam}(\Omega):=D\),
and let \(V\in C(\overline\Omega)\) be convex. The fundamental gap
conjecture asserted that
\begin{equation}\label{eq:sharp gap p=2}
    \lambda_{2,2}(\Omega,V)-\lambda_{1,2}(\Omega,V)
\geq
\frac{3\pi^2}{D^2}.
\end{equation}
Its study is connected with the log-concavity result of
Brascamp and Lieb \cite{BrascampLieb1976} and the sharp 
Poincar\'e inequality of Payne and Weinberger
\cite{PayneWeinberger1960}. The conjecture was preceded by the
dimension-independent estimates of Singer, Wong, Yau, and Yau
\cite{SingerWongYauYau1985}, Yu and Zhong \cite{YuZhong1986}, and Ling
\cite{Ling1993}, while sharp one-dimensional results were obtained by
Ashbaugh--Benguria \cite{AshbaughBenguria1989} and Lavine
\cite{Lavine1994}. Andrews and Clutterbuck resolved the conjecture by
proving the sharp bound in \cite{AndrewsClutterbuck2011}, and an
alternative elliptic proof was given by Ni \cite{Ni2013}. The fundamental gap problem has been studied on curved spaces.
Sharp comparison results are known for convex spherical domains and
for several positively curved settings
\cite{OdenSungWang1999,SetoWangWei2019,HeWeiZhang2020,
DaiSetoWei2021,AndrewsNi2012,KhanNguyenTuerkoenWei2025}.
In negative curvature, however, log-concavity may fail and convex
domains of fixed diameter may have arbitrarily small fundamental gap
\cite{Shih1989,BourniEtAl2021,BourniEtAl2022,KhanNguyen2024};
see also \cite{NguyenStancuWei2022,ClutterbuckJackelNguyen2025,
KhanTuerkoen2026} for related results in hyperbolic and conformally
flat settings.

\smallskip

In this paper we study the corresponding problem for the Dirichlet
\(p\)-Laplacian.  Let \(1<p<\infty\), and consider
\begin{equation}\label{eq:introduction-eigenvalue-problem}
\begin{cases}
-\Delta_pu+V(x)|u|^{p-2}u
=
\lambda |u|^{p-2}u,
&x\in\Omega,\\[1mm]
u=0,
&x\in\partial\Omega,
\end{cases}
\end{equation}
where
\[
\Delta_pu
:=
\operatorname{div}\bigl(|\nabla u|^{p-2}\nabla u\bigr).
\]
We denote by \(\lambda_{k,p}(\Omega,V)\) the variational eigenvalues
defined through the Krasnosel'ski\u{\i} genus.  More precisely, if
\[
\mathcal K
:=
\left\{
u\in W_0^{1,p}(\Omega):
\|u\|_{L^p(\Omega)}=1
\right\},
\qquad
\mathcal E_V(u)
:=
\int_\Omega
\bigl(|\nabla u|^p+V|u|^p\bigr)\,dx,
\]
then
\begin{equation}\label{eq:eigenvalue}
\lambda_{k,p}(\Omega,V)
:=
\inf_{\substack{
A\subset\mathcal K\ {\rm compact,\ symmetric}\\
\gamma(A)\geq k}}
\ \sup_{u\in A}\mathcal E_V(u),
\qquad k\geq1,
\end{equation}
where \(\gamma\) denotes the Krasnosel'ski\u{\i} genus; see
\cite{Amann1972,GarciaAzoreroPeral1987,Szulkin1988,
CuestaDeFigueiredoGossez1999,Le2006,CuestaRamosQuoirin2009,
AudouxBobkovParini2018,FuscoMukherjeeZhang2019}.
The fundamental gap is
\[
\Gamma_p(\Omega,V)
:=
\lambda_{2,p}(\Omega,V)-\lambda_{1,p}(\Omega,V).
\]

The case \(p=2\) is special in two respects. First, \(V\) is the
potential in the Schr\"odinger operator and therefore has a direct
physical interpretation. Second, the sharp estimate in \eqref{eq:sharp gap p=2}
holds uniformly for all convex potentials \(V\). No corresponding
potential-independent sharp estimate is known for \(p\neq2\). More
importantly, the results below show that the fundamental gap already
behaves very differently from the linear case when \(V\equiv0\). We
therefore first study the unperturbed Dirichlet \(p\)-Laplacian and the
dependence of its fundamental gap on the geometry of the domain.

\subsection{The one-dimensional gap and collapsing convex domains}

We begin with the one-dimensional problem.  Fundamental gaps for one-dimensional \(p\)-Laplacian problems have been
studied for various classes of potentials and boundary conditions. Cheng, Lian,
and Wang considered gap minimization under single-well assumptions
\cite{ChengLianWang2014}, while Andrews, Clutterbuck, and Hauer obtained
further estimates for nonlinear one-dimensional eigenvalue problems
\cite{AndrewsClutterbuckHauer2021}.  Related gap and eigenvalue-ratio
questions were considered in \cite{ChenCheng2020,WenZhou2026}.
For \(p=2\), Lavine proved the sharp Dirichlet gap for arbitrary convex
potentials \cite{Lavine1994}.  For \(p\neq2\), the corresponding sharp result remained open. To the best of
our knowledge, the theorem below gives the first sharp result for arbitrary
convex potentials in the full range \(1<p<\infty\).

We use the normalization
\begin{equation}\label{eq:pip}
\pi_p
:=
2\int_0^1\frac{ds}{(1-s^p)^{1/p}}
=
\frac{2\pi}{p\sin(\pi/p)}.
\end{equation}
If $I_D:=(-D/2,D/2),$
then the Dirichlet eigenvalues of the one-dimensional \(p\)-Laplacian
with zero potential are
\begin{equation}\label{eq:first eigen}
\lambda_{k,p}(I_D,0)
=
(p-1)\left(\frac{k\pi_p}{D}\right)^p,
\qquad k\geq1;
\end{equation}
see, for instance, \cite[Section~2]{EdmundsGurkaLang2012} and
\cite[Chapter~2]{LangEdmunds2011}.  Hence
\[
\Gamma_p(I_D,0)
=
(p-1)(2^p-1)
\left(\frac{\pi_p}{D}\right)^p.
\]
Our first result shows that the same quantity is the sharp lower bound
for every convex potential.

\begin{theorem}
\label{thm:sharp-one-dimensional-fundamental-gap}
Let \(1<p<\infty\), \(D>0\), and let
\(V\in C(\overline{I_D})\) be convex.  Then
\[
\Gamma_p(I_D,V)
\geq
(p-1)(2^p-1)
\left(\frac{\pi_p}{D}\right)^p.
\]
Equality holds if and only if \(V\) is constant.
\end{theorem}

Thus, in one dimension, allowing an arbitrary convex potential does not
change the sharp gap constant, and equality is attained precisely for
constant potentials.

For \(p=2\), Lavine's proof uses an integral identity specific to the
linear equation, see \cite[Lemma~3.3]{Lavine1994}.  Its nonlinear analogue contains an additional term
involving $u|u'|^{p-2}u',$
which does not reduce to a boundary term when \(p\neq2\).  Our proof
therefore uses a different argument based on the nodal decomposition of
the second eigenfunction.  For an affine potential \(V(x)=ax\), after translating and scaling the two
nodal intervals, the fundamental-gap comparison reduces to the study of
\[
\widehat\mu(\beta)
:=
\lambda_{1,p}\bigl((0,1),\beta x\bigr)-\frac{\beta}{2}.
\]
By reflection symmetry and the variational characterization of the first
eigenvalue, \(\widehat\mu\) is even and concave; see
Lemma~\ref{lem:mu-hat-concavity-symmetry}. Hence
\[
\widehat\mu(\beta)
\leq
\widehat\mu(0)
\qquad
\text{for every }\beta\in\mathbb R.
\] 
This inequality is the key
one-dimensional comparison in the proof of
Theorem~\ref{thm:sharp-one-dimensional-fundamental-gap}.

\smallskip

 The one-dimensional theorem might suggest that the zero-potential one-dimensional constant continues to control the gap in higher dimensions. This is what happens for \(p=2\) in \eqref{eq:sharp gap p=2}; however, for \(p\neq2\), dimensional collapse shows that this picture breaks down already when \(V\equiv0\).

Very little is known about higher-dimensional fundamental gaps for the
\(p\)-Laplacian, largely because the two key ingredients in the linear
argument of Andrews and Clutterbuck \cite{AndrewsClutterbuck2011} do
not persist for \(p\neq2\). The
quotient \(u_2/u_1\) no longer satisfies a closed linear equation, and
the sharp one-dimensional two-point modulus underlying the linear
theory fails in dimensions \(N\geq2\); see \cite{chen-two-point}.
A recent estimate of Yessirkegenov and Zhangirbayev
\cite{YessirkegenovZhangirbayev2026} gives, for \(p\geq2\),
\[
\Gamma_p(\Omega,0)
\geq
2^{2-p}
\left(\frac{\pi_p}{D}\right)^p
\inf_{c\in\mathbb R}
\|u_2-cu_1\|_{L^p(\Omega)}^p.
\]
The right-hand side still depends on the first two eigenfunctions and
therefore does not yield a universal diameter-only lower bound.

To the best of our knowledge, this is the first systematic study of the
higher-dimensional fundamental gap for the Dirichlet \(p\)-Laplacian on
convex domains over the full range \(1<p<\infty\). Our
higher-dimensional results concern the unperturbed problem \(V\equiv0\).
We begin by analyzing the gap along collapsing convex domains.
Motivated by previous work on fundamental gaps under geometric collapse
\cite{LuRowlett2013,LuRowlett2014} and by spectral asymptotics on thin
domains
\cite{BorisovFreitas2010,FriedlanderSolomyak2009,BrianiButtazzoPrinari2022},
we consider a family of smooth convex domains collapsing to a segment.

To isolate the geometric part of the nonlinear problem, from this point
on we first take \(V\equiv0\).  The following family already separates
the three ranges
\[
1<p<2,\qquad p=2,\qquad p>2.
\]
Let \(N\geq2\), \(D>0\), and choose \(m_\varepsilon\in\mathbb N\) such
that $m_\varepsilon\to\infty,\,
m_\varepsilon\varepsilon^2\to\infty$ as $\varepsilon\downarrow0.$ Set
\begin{equation}\label{eq:intro-collapsing-domains}
\Omega_\varepsilon
:=
\left\{
(x,y)\in\mathbb R\times\mathbb R^{N-1}:
\left(\frac{2x}{D}\right)^{2m_\varepsilon}
+
\left(\frac{|y|}{\varepsilon}\right)^{2m_\varepsilon}
<1
\right\}.
\end{equation}
These are smooth bounded convex domains satisfying
\[
\overline{\Omega_\varepsilon}
\longrightarrow
\left[-\frac D2,\frac D2\right]\times\{0\}
\quad\text{in the Hausdorff sense},
\qquad
\operatorname{diam}(\Omega_\varepsilon)\longrightarrow D.
\]

\begin{theorem}
\label{thm:smooth-quadratic-threshold-collapse}
Let \(\Omega_\varepsilon\) be given by
\eqref{eq:intro-collapsing-domains}.  Then as $\varepsilon\downarrow0$, we have
\[
\Gamma_p(\Omega_\varepsilon,0)
\longrightarrow
\begin{cases}
0,
&1<p<2,\\[1mm]
\dfrac{3\pi^2}{D^2},
&p=2,\\[3mm]
+\infty,
&p>2.
\end{cases}
\]
More precisely, for \(p>2\), $\Gamma_p(\Omega_\varepsilon,0)
\gtrsim
\varepsilon^{\,2-p}.$
\end{theorem}

The first two cases follow by comparing
\(\Omega_\varepsilon\) with thin cylinders.  For \(1<p<2\), the first
two eigenvalues have the same leading transverse contribution, while
the longitudinal localization cost is of order $\varepsilon^{2-p}=o(1),$
and hence the gap tends to zero.  At \(p=2\), separation of variables
gives the exact longitudinal contribution $\frac{3\pi^2}{D^2}.$
For \(p>2\), the same factor \(\varepsilon^{2-p}\) diverges. The lower
bound in this range is obtained from the higher-dimensional gap estimate
in Corollary \ref{cor:enhanced-fundamental-gap}.

\smallskip

Theorem~\ref{thm:smooth-quadratic-threshold-collapse} already shows a
sharp distinction between the three ranges of \(p\). For every
\(N\geq2\), the collapsing family constructed above gives
\[
\inf_{\substack{
\Omega\subset\mathbb R^N\\
\Omega\ {\rm bounded,\ convex}}}
\operatorname{diam}(\Omega)^p\Gamma_p(\Omega,0)
=
0,
\qquad
1<p<2,
\]
whereas for \(p=2\) the sharp value is \(3\pi^2\). For \(p>2\), along
the same family the gap instead diverges at least at the rate
\(\varepsilon^{2-p}\). These contrasting behaviors lead to the
diameter-normalized optimization problem considered next.

\subsection{The optimal diameter-normalized gap}
\label{the optimal gap}

Let \(\mathcal C_N\) denote the class of bounded convex domains in
\(\mathbb R^N\). For \(N\geq1\) and \(1<p<\infty\), define
\[
\mathcal G_{p,N}
:=
\inf_{\Omega\in\mathcal C_N}
\operatorname{diam}(\Omega)^p
\Gamma_p(\Omega,0).
\]
The quantity \(\mathcal G_{p,N}\) is invariant under translations,
orthogonal transformations, and dilations. Our first result determines
\(\mathcal G_{p,N}\) completely for \(N=1\) and for \(1<p\leq2\),
and proves existence of a minimizing domain for every \(N\geq2\) and
\(p>2\). It also describes what happens to these minimizers as
\(p\downarrow2\).

\begin{theorem}
\label{thm:existence-gap-minimizer}
We have
\[
\mathcal G_{p,N}
=
\begin{cases}
(p-1)(2^p-1)\pi_p^p,
& N=1,\quad 1<p<\infty,\\[1mm]
0,
& N\geq2,\quad 1<p<2,\\[1mm]
3\pi^2,
& N\geq2,\quad p=2.
\end{cases}
\]
For \(N\geq2\) and \(p>2\), the infimum is attained and, in particular, $\mathcal G_{p,N}>0.$
More precisely, there exists \(\Omega_p^*\in\mathcal C_N\) with $\operatorname{diam}(\Omega_p^*)=1$
such that $\mathcal G_{p,N}
=
\Gamma_p(\Omega_p^*,0).$
Moreover, for every choice of diameter-normalized minimizers
\(\Omega_p^*\), $r_{\Omega_p^*}
\rightarrow0$ as $p\downarrow2,$
where
\[
r_\Omega
:=
\sup\left\{
r>0:
B_r(x)\subset\Omega
\text{ for some }x\in\Omega
\right\}
\]
denotes the inradius of \(\Omega\).
\end{theorem}

The first line follows from the sharp one-dimensional result above,
while the cases \(1<p<2\) and \(p=2\) follow respectively from
Theorem~\ref{thm:smooth-quadratic-threshold-collapse} and the classical
fundamental-gap theorem \cite{AndrewsClutterbuck2011}. For fixed
\(p>2\), the main input for existence is the enhanced gap estimate
\eqref{eq:enhanced-fundamental-gap}. It implies that $\Gamma_p(\Omega_j,0)\rightarrow+\infty$
whenever
\[
\operatorname{diam}(\Omega_j)=1,
\qquad
r_{\Omega_j}\longrightarrow0.
\]
A minimizing sequence has bounded fundamental gap and therefore cannot
collapse. Compactness of convex bodies then yields a minimizer.

The situation changes at \(p=2\). The quantitative refinement of the
fundamental-gap inequality obtained by Amato, Bucur, and Fragal\`a
\cite[Theorem~1]{AmatoBucurFragala2024} gives, for \(N\geq2\),
\[
\Gamma_2(\Omega,0)
\geq
\frac{3\pi^2}{D^2}
+
c_N\frac{w_\Omega^6}{D^8}
>
\frac{3\pi^2}{D^2},
\qquad
D:=\operatorname{diam}(\Omega),
\]
where \(w_\Omega\) denotes the width of \(\Omega\). Hence
\(\mathcal G_{2,N}=3\pi^2\) is not attained by any
\(\Omega\in\mathcal C_N\), although it is approached by convex domains
collapsing to a segment. Thus \(p=2\) marks the transition between nonattainment through
collapsing sequences for \(1<p\leq2\) and attainment for every fixed
\(p>2\). Theorem~\ref{thm:existence-gap-minimizer} shows, however, that this
nondegeneracy is not uniform as \(p\downarrow2\).

\smallskip

It is therefore natural to ask how the minimizing domains degenerate
as \(p\downarrow2\), and whether their limiting geometry and the
distribution of their first eigenfunctions can be described. Directly
identifying the Hausdorff limits is delicate: in dimensions \(N\geq3\),
a degenerate convex limit may still have affine dimension larger than
one. We therefore also keep track of the ground-state measures. If
\(u_p>0\) is the normalized first eigenfunction on a minimizer
\(\Omega_p^*\), we set
\[
d\mu_p
:=
u_p^p\mathbf 1_{\Omega_p^*}\,dx,
\qquad
\int_{\Omega_p^*}u_p^p\,dx=1.
\]

The appropriate quantity for describing the spectral information
carried by a limiting measure is its Poincar\'e constant. Since the
weak limit may be supported on a proper affine subspace, we use the
following formulation. Let \(\nu\) be a log-concave probability measure
on \(\mathbb R^N\), and set
\[
S:=\operatorname{aff}(\operatorname{supp}\nu),
\]
where \(\operatorname{aff}\) denotes the affine hull. For a smooth
function \(f\) on \(S\), let \(\nabla_S f\) denote the Euclidean
gradient along \(S\). We denote by \(H^1(\nu)\) the completion of smooth
functions on \(S\) with respect to
\[
\|f\|_{H^1(\nu)}^2
:=
\int_S |f|^2\,d\nu
+
\int_S |\nabla_S f|^2\,d\nu,
\]
and set
\[
\mathcal E_\nu(f)
:=
\int_S|\nabla_S f|^2\,d\nu,
\qquad
\operatorname{Var}_\nu(f)
:=
\int_S
\left|f-\int_S f\,d\nu\right|^2\,d\nu.
\]
If \(\nu\) is a point mass, set \(P(\nu):=0\). Otherwise, define
\[
P(\nu)
:=
\sup_{\substack{
f\in H^1(\nu)\\
f\not\equiv\mathrm{const}\ \nu\text{-a.e.}}}
\frac{\operatorname{Var}_\nu(f)}
{\mathcal E_\nu(f)}.
\]

The connection between \(P(\nu)\) and the fundamental gap is exact in
the linear case. If \(p=2\) and \(d\nu=u_1^2\,dx\), where \(u_1\) is the
normalized first Dirichlet eigenfunction of \(-\Delta\) on \(\Omega\),
then the ground-state transform gives
\[
P(\nu)
=
\frac{1}{
\lambda_{2,2}(\Omega,0)-\lambda_{1,2}(\Omega,0)}
=
\frac{1}{\Gamma_2(\Omega,0)}.
\]
This follows from the ground-state identity used in the proof of
Theorem~\ref{thm:ground-state-weighted-spectral-gap}(ii). Hence the
classical fundamental-gap theorem
\cite{AndrewsClutterbuck2011} yields
\[
P(\nu)
\leq
\frac{\operatorname{diam}(\Omega)^2}{3\pi^2}.
\]
This is stronger than what follows from log-concavity alone: an
arbitrary log-concave probability measure satisfies
\[
P(\nu)
\leq
\frac{\operatorname{diam}(\operatorname{supp}\nu)^2}{\pi^2};
\]
see Lemma~\ref{lemma:ordinary3}. Thus \(P(\nu)\) retains the sharp
fundamental-gap information needed in the analysis of the measures
\(\mu_p\), beyond what is implied by log-concavity alone.

\smallskip

We now describe the degeneration of exact minimizers as \(p\downarrow2\).
Let \(p_j>2\) with \(p_j\downarrow2\), and let
\(\Omega_j\in\mathcal C_N\) be diameter-normalized minimizers:
\[
\operatorname{diam}(\Omega_j)=1,
\qquad
\Gamma_{p_j}(\Omega_j,0)
=
\mathcal G_{p_j,N}.
\]
Let \(u_j>0\) be the first Dirichlet \(p_j\)-eigenfunction normalized by
\[
\int_{\Omega_j}u_j^{p_j}\,dx=1,
\qquad
d\mu_j
:=
u_j^{p_j}\mathbf 1_{\Omega_j}\,dx.
\]
\begin{theorem}
\label{thm:limits-gap-minimizers}
Let \(N\geq2\), and let \(p_j,\Omega_j,\mu_j\) be as above.
After translations and passage to a subsequence, there exist a compact
set \(K\subset\mathbb R^N\) and a probability measure \(\mu\) such that
\[
\overline{\Omega_j}
\longrightarrow
K
\quad\text{in the Hausdorff distance},
\qquad
\mu_j\rightharpoonup\mu
\quad\text{weakly as probability measures}.
\]
The limits \(K\) and \(\mu\) satisfy the following properties.

\begin{enumerate}
\item[\rm(i)]
The set \(K\) is convex and
\[
\operatorname{diam}(K)=1,
\qquad
\operatorname{int}_{\mathbb R^N}K=\varnothing,
\qquad
\operatorname{supp}\mu\subset K.
\]
If \(N=2\), then \(K\) is a line segment of length \(1\).

\item[\rm(ii)]
The limiting measure satisfies
\[
P(\mu)=\frac1{3\pi^2},
\qquad
\operatorname{diam}(\operatorname{supp}\mu)=1,
\qquad
\dim\operatorname{aff}(\operatorname{supp}\mu)=1.
\]
Hence there exist \(a\in\mathbb R^N\) and
\(e\in\mathbb S^{N-1}\) such that
\[
\operatorname{supp}\mu
=
\left\{
a+te:
|t|\leq\frac12
\right\},
\]
and
\[
\int_{\mathbb R^N}f\,d\mu
=
2\int_{-1/2}^{1/2}
f(a+te)\cos^2(\pi t)\,dt,
\qquad
f\in C_b(\mathbb R^N).
\]
In particular, if \(N=2\), then $K=\operatorname{supp}\mu.$
\end{enumerate}
\end{theorem}
Thus, the limiting ground-state measure is completely
characterized. A key point is that the limit retains the ground-state
form of the linear theory. Writing $d\mu=\rho\,dx_S,$
one first proves that there exist a convex potential \(V\) on the
relative interior of \(\operatorname{supp}\mu\), bounded convex sets
\(U_n\) exhausting this relative interior, and normalized positive first
Dirichlet eigenfunctions \(\varphi_n\) of \(-\Delta_S+V\) on \(U_n\)
such that $\varphi_n^2\,dx_S
\rightarrow
\mu$
in total variation. Equivalently,
\[
\varphi_n
\longrightarrow
\sqrt{\rho}
\qquad
\text{strongly in }L^2(S).
\]
The sharp linear fundamental-gap theorem therefore passes to the limit
and yields
\[
P(\mu)
\leq
\frac{\operatorname{diam}(\operatorname{supp}\mu)^2}{3\pi^2}.
\]

For limits of exact minimizers, the nonlinear gap estimate gives the
reverse bound, and hence
\[
P(\mu)
=
\frac1{3\pi^2},
\qquad
\operatorname{diam}(\operatorname{supp}\mu)=1.
\]
The equality analysis in Proposition~\ref{prop:one-dimensional-limit-measure}
then implies $\dim\operatorname{aff}(\operatorname{supp}\mu)=1.$
Consequently, after an isometry identifying
\(\operatorname{supp}\mu\) with \([-1/2,1/2]\),
\[
d\mu(t)
=
2\cos^2(\pi t)\,dt,
\qquad
-\frac12<t<\frac12.
\]
When \(N=2\), convexity and $\operatorname{int}_{\mathbb R^2}K=\varnothing$
force \(K\) itself to be the same unit line segment.

The collapse is also reflected in the first two eigenvalues, see Corollary \ref{cor:minimizer-spectral-collapse}. Although
\[
\lambda_{k,p_j}(\Omega_j,0)
\longrightarrow
+\infty,
\qquad
k=1,2,
\]
their difference remains finite and satisfies
\[
\Gamma_{p_j}(\Omega_j,0)
=
\lambda_{2,p_j}(\Omega_j,0)
-
\lambda_{1,p_j}(\Omega_j,0)
\longrightarrow
3\pi^2,
\]
while
\[
\frac{\lambda_{2,p_j}(\Omega_j,0)}
{\lambda_{1,p_j}(\Omega_j,0)}
\longrightarrow
1.
\]

Theorem~\ref{thm:existence-gap-minimizer} gives the exact value of
\(\mathcal G_{p,N}\) for \(N=1\) and for \(1<p\leq2\), but no explicit
formula is available in the genuinely higher-dimensional nonlinear
regime \(N\geq2\), \(p>2\). The preceding analysis of minimizers shows
that this regime is also the one in which compactness is restored and
the minimizing domains depend nontrivially on \(p\). We therefore turn
to the dependence of the optimal constant itself on \(p\) and \(N\).
The next theorem gives its behavior as \(p\downarrow2\), as
\(p\to\infty\), and as \(N\to\infty\) for fixed \(p>2\).

\begin{theorem}
\label{thm:optimal-gap-asymptotics}
The constants \(\mathcal G_{p,N}\) satisfy the following.

\begin{enumerate}
\item[\rm(i)]
For every \(N\geq1\),
\[
\lim_{p\downarrow2}\mathcal G_{p,N}
=
\mathcal G_{2,N}
=
3\pi^2.
\]

\item[\rm(ii)]
For every \(p>2\), there exist constants \(c_p,C_p>0\), depending only
on \(p\), such that
\[
c_pN^{p-2}
\leq
\mathcal G_{p,N}
\leq
C_pN^{p-2},
\qquad
N\geq2.
\]

\item[\rm(iii)]
For every \(N\geq1\),
\[
\lim_{p\to\infty}
\mathcal G_{p,N}^{1/p}
=
4.
\]
\end{enumerate}
\end{theorem}

Part~\textnormal{\rm (ii)} has an immediate consequence. For every fixed
\(p>2\), there exists \(N_p\geq2\) such that
\[
\mathcal G_{p,N}>
\mathcal G_{p,1}
=
(p-1)(2^p-1)\pi_p^p,
\qquad
N\geq N_p.
\]
Thus, unlike the linear case \(p=2\), the optimal diameter-normalized
gap for \(p>2\) depends genuinely on the dimension. 

Part~\textnormal{\rm (i)} is precisely
Corollary~\ref{cor:continuity-weighted-poincare}. For
part~\textnormal{\rm (ii)}, the lower bound follows from the enhanced
fundamental-gap estimate \eqref{eq:enhanced-fundamental-gap} together with the Faber--Krahn,
isodiametric, and Cheeger inequalities, which yield the factor
\(N^{p-2}\). For the upper bound, we use the product domain
\(B_1^{N-1}\times(-1,1)\), test the second eigenvalue with an odd
separated function, and apply the genus characterization of
\(\lambda_{2,p}\), together with $\lambda_{1,p}(B_1^{N-1})\lesssim N^p.$
Finally, part~\textnormal{\rm (iii)} follows from the \(p\to\infty\)
limits of the first two Dirichlet \(p\)-eigenvalues in terms of the
inradius and the maximal radius of two disjoint inscribed balls. The
unit ball gives the upper bound \(4\), while the lower bound follows
from a compactness argument and the geometric estimate $R_2(\Omega)\leq\frac14$
for convex domain $\Omega$ of diameter \(1\).

The preceding results show that, when \(p\neq2\), the higher-dimensional
fundamental gap already differs substantially from the linear case even
for \(V\equiv0\). In particular, dimensional collapse, attainment of the
optimal gap, and the dependence of \(\mathcal G_{p,N}\) on the dimension
all change across \(p=2\). Thus, unlike in the linear fundamental-gap
theorem, there is no reason to expect a convex potential to interact
with the geometry in the same way. We therefore formulated the preceding
optimization problem for \(V\equiv0\), in order to isolate the dependence
on the domain. We now turn to properties of the first eigenfunction and to estimates that remain valid in the presence of convex potentials.

\subsection{Log-concavity of the first eigenfunction with convex potentials}

We first establish the log-concavity of the first eigenfunction with convex potentials, which
will be used in the gap estimates below. For \(V\equiv0\), this was
proved by Sakaguchi \cite{sakaguchi-concavity}. By adapting his
regularization argument, we extend the result to convex potentials.

\begin{theorem}
\label{thm:first-eigenfunction-log-concavity}
Let \(1<p<\infty\), let \(\Omega\subset\mathbb R^N\) be a bounded convex
domain, and let \(V\in C(\overline\Omega)\) be convex. Then the positive
first Dirichlet eigenfunction associated with
\eqref{eq:introduction-eigenvalue-problem} is log-concave in \(\Omega\).
\end{theorem}

The proof is based on uniformly elliptic regularizations of the
\(p\)-Laplacian, boundary concavity estimates, and a two-point maximum
principle. The potential produces additional terms in the regularized
equation, while its convexity gives the sign required in the two-point
argument. We give the complete regularization and limiting argument in
Section~\ref{sec:logconca}.

\subsection{A degenerate weighted Poincar\'e inequality and gap estimates}

Although log-concavity remains available for \(p\neq2\), it is not by
itself sufficient to produce the higher-dimensional gap estimates needed
below. We therefore proceed differently. The starting point is
a nonlinear ground-state representation for the energy deficit above the
first eigenvalue. Two quantitative convexity remainders then lead to,
respectively, a diameter-only lower bound and a quadratic weighted form
with respect to the ground-state measure. The latter gives rise to the
degenerate weighted Poincar'e inequality used in the optimization and
collapse arguments.

From this point on, we assume \(p\geq2\). To implement this idea, we
quantify the convexity remainder of the function
\(\xi\mapsto|\xi|^p\). For \(\xi,\eta\in\mathbb R^N\), set
\[
\mathcal C_p(\xi,\eta)
:=
|\xi|^p-|\xi-\eta|^p
-p|\xi-\eta|^{p-2}(\xi-\eta)\cdot\eta,
\]
and
\[
\mathfrak c_p
:=
\inf_{\substack{\xi,\eta\in\mathbb R^N\\ \eta\neq0}}
\frac{\mathcal C_p(\xi,\eta)}{|\eta|^p}.
\]
Let \(u_1>0\) be the first eigenfunction associated with
\(\lambda_{1,p}(\Omega,V)\), then, by the same argument as in
\cite[Corollary~3.5]{ApseitYessirkegenovZhangirbayev2025}, one obtains
the ground-state identity
\[
\mathcal E_V(w)
-
\lambda_{1,p}(\Omega,V)
\int_\Omega |w|^p\,dx
=
\int_\Omega
\mathcal C_p
\left(
z\nabla u_1+u_1\nabla z,
u_1\nabla z
\right)\,dx.
\]
The convexity remainder satisfies
\begin{equation}\label{eq:cp-remainder-estimate}
\mathcal C_p(\xi,\eta)
\geq
\mathfrak c_p|\eta|^p,
\qquad
\mathfrak c_p\geq2^{2-p},
\end{equation}
and, for \(p\geq2\),
\begin{equation}\label{eq:cpaboptimal}
\mathcal C_p(a+b,b)
\geq
\frac{p}{2}|a|^{p-2}|b|^2.
\end{equation}
The constant \(p/2\) in \eqref{eq:cpaboptimal} is optimal. Moreover,
\(\mathfrak c_2=1\), while for \(p>2\) the constant \(\mathfrak c_p\)
is given explicitly in \cite[Lemma~2.7 and Remark~2.9]{YessirkegenovZhangirbayev2026}.

The two estimates above give different quantitative remainders. The
\(p\)-power remainder \eqref{eq:cp-remainder-estimate}, combined with
the Poincar\'e inequality for log-concave measures, gives a
diameter-only bound for convex potentials. The quadratic remainder
\eqref{eq:cpaboptimal} instead produces a weighted quadratic form
determined by the first eigenfunction. Using the mountain-pass
characterization
\[
\lambda_{2,p}(\Omega,V)
=
\inf_{\gamma\in\Gamma(u_1,-u_1)}
\max_{s\in[0,1]}
\mathcal E_V(\gamma(s)),
\]
together with the nonlinear analogue of the orthogonal complement of
the ground state, we obtain the following two estimates. For the \(p\)-Laplacian with a potential, this characterization follows
from \cite[Proposition~15]{CuestaRamosQuoirin2009}. In the case
\(V\equiv0\), see also \cite{CuestaDeFigueiredoGossez1999} and
\cite[Theorem~2.6]{BrascoFranzina2013}.

\begin{theorem}
\label{thm:dimension-free-higher-dimensional-gap}
Let \(p\geq2\), let
\(\Omega\subset\mathbb R^N\) be a bounded convex domain, and set
\(D:=\operatorname{diam}(\Omega)\).

\begin{enumerate}
\item[\rm(i)]
If \(V\in C(\overline\Omega)\) is convex, then
\[
\Gamma_p(\Omega,V)
\geq
\mathfrak c_p(p-1)
\left(\frac{\pi_p}{D}\right)^p.
\]

\item[\rm(ii)]
If \(V\in C(\overline\Omega)\), then
\[
\Gamma_p(\Omega,V)
\geq
\frac{2}{p}\Lambda_p(\Omega,V),
\]
where
\[
\Lambda_p(\Omega,V)
:=
\inf_{\substack{
g\in C^{0,1}(\overline\Omega)\setminus\{0\}\\
\int_\Omega g\,d\mu=0}}
\frac{
\displaystyle
\int_\Omega
|\nabla\log u_1|^{p-2}|\nabla g|^2\,d\mu
}{
\displaystyle
\int_\Omega g^2\,d\mu
},
\qquad
d\mu:=u_1^p\,dx.
\]
\end{enumerate}
\end{theorem}

For \(V\equiv0\), the estimate in part~\textnormal{\rm (i)} with
\(2^{2-p}\) in place of \(\mathfrak c_p\) follows from
\cite[Theorem~3.3]{YessirkegenovZhangirbayev2026}. Thus
part~\textnormal{\rm (i)} improves the constant to the optimal
convexity remainder \(\mathfrak c_p\) and extends the estimate to convex
potentials. Part~\textnormal{\rm (ii)} is of a different form: the
nonlinear energy deficit controls the quadratic Dirichlet form
\[
\int_\Omega
|\nabla\log u_1|^{p-2}|\nabla g|^2\,d\mu.
\]
For \(p=2\), the weight disappears and $\Gamma_2(\Omega,V)=\Lambda_2(\Omega,V).$
Thus the fundamental gap is exactly the first nonzero eigenvalue of the
ground-state weighted Neumann form. Equivalently, the sharp
fundamental-gap constant is the reciprocal of the sharp Poincar\'e
constant for the ground-state measure \(d\mu=u_1^2\,dx\). For \(p>2\), however, the weight $|\nabla\log u_1|^{p-2}$
vanishes on the critical set of \(u_1\), and a genuinely degenerate
weighted Poincar\'e inequality is required.

\smallskip

We establish this inequality for the zero-potential problem considered
in the spectral optimization above. Let \(u_1\) be the positive first
Dirichlet \(p\)-eigenfunction on \(\Omega\in\mathcal C_N\), normalized by $\int_\Omega u_1^p\,dx=1,$
and set
\[
d\mu:=u_1^p\,dx,
\qquad
r:=|\nabla\log u_1|,
\qquad
D:=\operatorname{diam}(\Omega).
\]
Define
\[
\Lambda_p(\Omega)
:=
\Lambda_p(\Omega,0)
=
\inf_{\substack{
g\in C^{0,1}(\overline\Omega)\setminus\{0\}\\
\int_\Omega g\,d\mu=0}}
\frac{
\displaystyle\int_\Omega r^{p-2}|\nabla g|^2\,d\mu
}{
\displaystyle\int_\Omega g^2\,d\mu
},
\]
and
\begin{equation}\label{eq:lamopti}
C_{p,N}^{\mathrm{opt}}
:=
\sup_{\Omega\in\mathcal C_N}
\frac{
\lambda_{1,p}(\Omega,0)^{\frac{p-2}{p}}
}{
D^2\Lambda_p(\Omega)
}.
\end{equation}
Equivalently, \(C_{p,N}^{\mathrm{opt}}\) is the optimal scale-invariant
constant in the degenerate weighted Poincar\'e inequality below.

\begin{theorem}
\label{thm:ground-state-weighted-spectral-gap}
For every \(p\geq2\) and \(N\geq1\), \(C_{p,N}^{\mathrm{opt}}<\infty\).
Equivalently, for every \(\Omega\in\mathcal C_N\) and every
\(g\in C^{0,1}(\overline\Omega)\),
\[
\operatorname{Var}_{\mu}(g)
\leq
C_{p,N}^{\mathrm{opt}}D^2
\lambda_{1,p}(\Omega,0)^{-\frac{p-2}{p}}
\int_\Omega r^{p-2}|\nabla g|^2\,d\mu.
\]

The following hold:

\begin{enumerate}

\item[\rm(i)]
For every \(p>2\) and \(N\geq2\),
\[
\frac{1}{3\pi^2}
<
C_{p,N}^{\mathrm{opt}}
\leq
5(p-1)^2.
\]

\item[\rm(ii)]
If \(p=2\), then
\[
C_{2,N}^{\mathrm{opt}}
=
\frac{1}{3\pi^2}
\qquad
\text{for every }N\geq1.
\]

\item[\rm(iii)]
If \(N=1\) and \(\Omega:=I_D=(-D/2,D/2)\), then
\[
\Lambda_p(I_D)
=
(p+1)\left(\frac{\pi_p}{D}\right)^p,
\]
and
\[
\Lambda_p(I_D)<\Gamma_p(I_D,0),
\qquad
C_{2,1}^{\mathrm{opt}}
<
C_{p,1}^{\mathrm{opt}}
\qquad
p>2.
\]
\end{enumerate}
\end{theorem}

The factor
\[
D^2\lambda_{1,p}(\Omega,0)^{-\frac{p-2}{p}}
\]
is forced by dilation invariance. The main difficulty is the vanishing
of \(r^{p-2}\) on the critical set of \(u_1\): estimates involving the
inverse weight introduce the singular factor \(r^{2-p}\).
Our proof uses the logarithmic ground-state equation
\eqref{eq:log-ground-state-equation}. More precisely,
Lemma~\ref{lem:log-ground-state-equations} provides a weighted
divergence identity valid across the critical set, together with a
convexity inequality in the sense of Radon measures. These imply the
ground-state Hardy estimate in
Proposition~\ref{prop:ground-state-hardy}, which controls the singular
weight \(r^{2-p}\). Combining this estimate with the weighted
\(L^1\)-diameter inequality of
Lemma~\ref{lem:weighted-L1-diameter} yields the degenerate weighted
Poincar\'e inequality above. To the best of our knowledge, such an
inequality with the correct spectral scaling is new.

\smallskip

Combining Theorems~\ref{thm:dimension-free-higher-dimensional-gap}
and~\ref{thm:ground-state-weighted-spectral-gap} gives the estimate used
throughout the optimization and collapse arguments.

\begin{corollary}
\label{cor:enhanced-fundamental-gap}
Let \(p\geq2\), let \(\Omega\in\mathcal C_N\), and set
\(D:=\operatorname{diam}(\Omega)\). Then
\begin{equation}\label{eq:enhanced-fundamental-gap}
\Gamma_p(\Omega,0)
\geq
\frac{2}{p\,C_{p,N}^{\mathrm{opt}}}
\lambda_{1,p}(\Omega,0)^{\frac{p-2}{p}}
D^{-2}.
\end{equation}
Consequently,
\[
\Gamma_p(\Omega,0)
\geq
\frac{2}{p\,C_{p,N}^{\mathrm{opt}}}
(p-1)^{\frac{p-2}{p}}
\pi_p^{p-2}D^{-p}
\]
and
\[
\mathcal G_{p,N}
\geq
\frac{2}{p\,C_{p,N}^{\mathrm{opt}}}
(p-1)^{\frac{p-2}{p}}
\pi_p^{p-2}.
\]

If \(p>2\) and \(\Omega_\varepsilon\) is the collapsing family of
Theorem~\ref{thm:smooth-quadratic-threshold-collapse}, then
\[
\Gamma_p(\Omega_\varepsilon,0)
\gtrsim
\varepsilon^{2-p}
\longrightarrow
+\infty
\qquad
\text{as }\varepsilon\downarrow0.
\]
\end{corollary}

The estimate \eqref{eq:enhanced-fundamental-gap} explains the
noncollapse of minimizing sequences for every fixed \(p>2\): if
\(\operatorname{diam}(\Omega_j)=1\) and the inradius tends to zero, then
\(\lambda_{1,p}(\Omega_j,0)\to+\infty\), and hence $\Gamma_p(\Omega_j,0)\to+\infty.$
It also explains why the zero-potential problem was singled out in the
shape optimization above. A bound of the form
\eqref{eq:enhanced-fundamental-gap} cannot hold uniformly for arbitrary
nontrivial potentials with a constant depending only on \(p\), since
the spectral gap is invariant under additive shifts of \(V\), whereas
the first eigenvalue is not.

\subsection{Structure of the proof of Theorem~\ref{thm:limits-gap-minimizers}}

We summarize the proof of Theorem~\ref{thm:limits-gap-minimizers} in
four steps.

\smallskip
\noindent
\textbf{Step 1: compactness and degeneration.}
After translations and passage to a subsequence, Blaschke's selection
theorem and weak compactness of probability measures give
\[
\overline{\Omega_j}\longrightarrow K,
\qquad
\mu_j\rightharpoonup\mu,
\]
where \(K\) is compact and convex,
\[
\operatorname{diam}(K)=1,
\qquad
\operatorname{supp}\mu\subset K.
\]
By Theorem~\ref{thm:existence-gap-minimizer},
\(r_{\Omega_j}\to0\). Hausdorff convergence of convex sets then implies $\operatorname{int}_{\mathbb R^N}K=\varnothing.$
If \(N=2\), \(K\) is therefore a line segment of length \(1\).

\smallskip
\noindent
\textbf{Step 2: the limiting Poincar\'e constant.}
Set
\[
\lambda_j:=\lambda_{1,p_j}(\Omega_j,0),
\qquad
\alpha_j:=\frac{p_j-2}{p_j},
\qquad
Q_j:=
\frac{\lambda_j^{\alpha_j}}
{\Lambda_{p_j}(\Omega_j)}.
\]
By Lemma~\ref{lem:constant-limit}, $Q_j\rightarrow P(\mu).$
Since \(\Omega_j\) is an exact minimizer,
\[
\Gamma_{p_j}(\Omega_j,0)
=
\mathcal G_{p_j,N}
\longrightarrow3\pi^2.
\]
Moreover,
Theorem~\ref{thm:dimension-free-higher-dimensional-gap}\textnormal{(ii)}
gives
\[
Q_j
\geq
\frac{2}{p_j}
\frac{\lambda_j^{\alpha_j}}
{\Gamma_{p_j}(\Omega_j,0)}.
\]
Together with
Lemma~\ref{lem:first-eigenvalue-diameter-lower-bound}, $\liminf_{j\to\infty}\lambda_j^{\alpha_j}\geq1,$
and hence $P(\mu)\geq\frac1{3\pi^2}.$
On the other hand, Theorem~\ref{thm:limit-measures} yields
\[
P(\mu)
\leq
\frac{\operatorname{diam}(\operatorname{supp}\mu)^2}{3\pi^2}
\leq
\frac{\operatorname{diam}(K)^2}{3\pi^2}
=
\frac1{3\pi^2}.
\]
Therefore
\[
P(\mu)=\frac1{3\pi^2},
\qquad
\operatorname{diam}(\operatorname{supp}\mu)=1.
\]

\smallskip
\noindent
\textbf{Step 3: one-dimensional concentration.}
By Step~2, \(\mu\) satisfies the hypotheses of
Proposition~\ref{prop:one-dimensional-limit-measure}. Its proof writes
\[
d\mu=w^2\,dx_S,
\qquad
S:=\operatorname{aff}(\operatorname{supp}\mu),
\]
as a Schr\"odinger ground-state measure and passes the sharp
Andrews--Clutterbuck two-point estimate to \(w\).

If \(\dim S\geq2\), Klartag localization
\cite[Theorem~1.5]{Klartag2017} reduces the equality case
\(P(\mu)^{-1}=3\pi^2\) to one-dimensional transport rays. The
quantitative one-dimensional gap estimate forces almost every
effective ray to have length \(1\) and weight $C_\xi\cos^2(\pi t).$
The sharp two-point estimate then implies $(\log h_\xi)''=0$
for the corresponding conditional Lebesgue density. By
Lemma~\ref{lem:full-length-ray-cluster}, such rays have zero
\(\dim S\)-dimensional measure, contradicting the positive
\(g^2\,d\mu\)-mass of the effective rays. Hence
\[
\dim\operatorname{aff}(\operatorname{supp}\mu)=1.
\]

\smallskip
\noindent
\textbf{Step 4: identification of the limit measure.}
Since \(\operatorname{supp}\mu\) is convex, one-dimensional, and has
diameter \(1\), there exist \(a\in\mathbb R^N\) and
\(e\in\mathbb S^{N-1}\) such that
\[
\operatorname{supp}\mu
=
\left\{
a+te:
|t|\leq\frac12
\right\}.
\]
The equality case in
Lemma~\ref{lem:one-dimensional-quantitative-gap} gives, under the
identification \(a+te\leftrightarrow t\),
\[
d\mu(t)=2\cos^2(\pi t)\,dt.
\]
Equivalently,
\[
\int_{\mathbb R^N}f\,d\mu
=
2\int_{-1/2}^{1/2}
f(a+te)\cos^2(\pi t)\,dt,
\qquad
f\in C_b(\mathbb R^N).
\]
If \(N=2\), then \(K\) and \(\operatorname{supp}\mu\) are line segments
of diameter \(1\) with
\(\operatorname{supp}\mu\subset K\), and therefore
\[
K=\operatorname{supp}\mu.
\]

The paper is organized as follows. Section~2 establishes the sharp one-dimensional gap and analyzes collapse to a segment. Section~3 studies the optimal diameter-normalized gap, including existence of minimizing domains for \(p>2\) and its asymptotic dependence on \(p\) and the dimension. Section~4 analyzes the degeneration of these minimizers as \(p\downarrow2\) and identifies the limiting ground-state measure. Section~5 proves log-concavity of the first \(p\)-eigenfunction on convex domains. Section~6 develops the ground-state stability estimate and the degenerate weighted Poincar'e inequality, and derives the enhanced fundamental-gap bound used in the preceding sections. The regularity and approximation results needed for the log-concavity argument are collected in the appendix.

\section{One-dimensional Gaps and Collapsing Convex Domains}

This section contains the two basic results that motivate the
higher-dimensional optimization problem considered later. We first prove
the sharp one-dimensional fundamental-gap estimate for arbitrary convex
potentials and characterize the equality case. We then study a family of
smooth convex domains collapsing to a segment and show that the
fundamental gap exhibits three different behaviors according as
\(1<p<2\), \(p=2\), or \(p>2\). In particular, the linear case appears
as the transition point between vanishing and diverging gaps under
dimensional collapse.

\subsection{The sharp one-dimensional gap}

In this subsection, we prove the sharp one-dimensional fundamental-gap
estimate for every \(p>1\) and every convex potential. We first recall the
following one-dimensional comparison principle. If \(V\) is convex on
\(I_D\), then there exists an affine function \(\ell(x)=ax+b\) such that
\begin{equation}\label{eq:convex-potential-linear-comparison}
\Gamma_p(I_D,V)
\geq
\Gamma_p(I_D,\ell)
=
\Gamma_p(I_D,ax),
\end{equation}
with strict inequality unless \(V\) is affine. This is the homogeneous
Dirichlet analogue of \cite[Theorem~1.2]{AndrewsClutterbuckHauer2021},
whose comparison argument carries over unchanged. It therefore remains to
minimize the gap among linear potentials.

For \(\ell>0\) and \(a,\beta\in\mathbb R\), define
\[
\nu(\ell,a)
:=
\lambda_{1,p}\bigl((0,\ell),ax\bigr),
\quad
\mu(\beta)
:=
\lambda_{1,p}\bigl((0,1),\beta x\bigr),\quad \widehat\mu(\beta):=\mu(\beta)-\frac{\beta}{2}.
\]

We first record the exact scaling relation that reduces the first
eigenvalue on an interval of arbitrary length to the corresponding
problem on the unit interval.

\begin{lemma}\label{lem:nu-scaling-global}
For any \(\ell>0\) and \(a\in\mathbb R\), we have
\[
\nu(\ell,a)
=
\ell^{-p}\mu\bigl(a\ell^{p+1}\bigr).
\]
\end{lemma}

\begin{proof}
For any \(u\in W_0^{1,p}(0,\ell)\), let
\(v(t):=u(\ell t)\in W_0^{1,p}(0,1)\). The change of variables
\(y=\ell t\) gives
\[
\int_0^\ell |u'|^p\,dy
=
\ell^{1-p}\int_0^1 |v'|^p\,dt
\]
and
\[
\int_0^\ell y^k|u|^p\,dy
=
\ell^{k+1}\int_0^1 t^k|v|^p\,dt,
\qquad
k=0,1.
\]
Substituting these into the Rayleigh quotient immediately yields the
result.
\end{proof}

We next establish the concavity and evenness of the centered first
eigenvalue, using its variational characterization and the reflection
symmetry of the unit interval.

\begin{lemma}\label{lem:mu-hat-concavity-symmetry}
The function \(\widehat\mu(\beta)\) is even and concave on
\(\mathbb R\). Consequently,
\[
\widehat\mu(\beta)
\leq
\widehat\mu(0)
=
\mu(0)
\qquad
\text{for every }\beta\in\mathbb R.
\]
\end{lemma}

\begin{proof}
The variational characterization gives
\[
\mu(\beta)
=
\inf_{\substack{
w\in W_0^{1,p}(0,1)\\
\|w\|_{L^p(0,1)}=1}}
\left\{
\int_0^1|w'|^p\,dt
+
\beta\int_0^1t|w|^p\,dt
\right\}.
\]
Thus \(\mu\), being the infimum of affine functions of \(\beta\), is
concave. For \(\widetilde w(t):=w(1-t)\), one has
\[
\int_0^1t|\widetilde w(t)|^p\,dt
=
\int_0^1(1-t)|w(t)|^p\,dt.
\]
Taking infima in the corresponding Rayleigh quotients yields $\mu(-\beta)=\mu(\beta)-\beta.$
Hence
\[
\widehat\mu(-\beta)
=
\mu(-\beta)+\frac{\beta}{2}
=
\widehat\mu(\beta).
\]
Therefore \(\widehat\mu\) is even and concave, and consequently
\[
\widehat\mu(0)
\geq
\frac{\widehat\mu(\beta)+\widehat\mu(-\beta)}{2}
=
\widehat\mu(\beta),
\]
which completes the proof.
\end{proof}

We now combine the nodal decomposition of a second eigenfunction with
the scaling relation in Lemma \ref{lem:nu-scaling-global} and the concavity of $\widehat\mu(\beta)$
to prove the sharp one-dimensional gap estimate and characterize the
equality case.

\begin{proof}[\textbf{Proof of
Theorem~\ref{thm:sharp-one-dimensional-fundamental-gap}.}]
We first treat linear potentials. Define $F(\beta)
:=
\widehat\mu(0)-\widehat\mu(\beta).$
By Lemma~\ref{lem:mu-hat-concavity-symmetry}, \(F\) is nonnegative,
even, and convex, with \(F(0)=0\). Hence
\begin{equation}\label{eq:D-convex-scaling}
F(\theta\beta)
\leq
\theta F(\beta)
\qquad
\text{for }\beta\in\mathbb R,\quad 0\leq\theta\leq1.
\end{equation}

For \(\ell>0\), set $\eta(\ell,a)
:=
\nu(\ell,a)-\frac{a\ell}{2}.$
After translating \((0,\ell)\) to
\((-\ell/2,\ell/2)\), \(\eta(\ell,a)\) is the first eigenvalue for the
centered linear potential \(ay\). By
Lemma~\ref{lem:nu-scaling-global},
\begin{equation}\label{eq:eta-scaling-global}
\eta(\ell,a)
=
\ell^{-p}\widehat\mu\bigl(a\ell^{p+1}\bigr).
\end{equation}

Let \(u_2(\cdot;a)\) be an eigenfunction corresponding to
\(\lambda_{2,p}(I_D,ax)\). By the one-dimensional nodal theorem,
\(u_2(\cdot;a)\) has exactly two nodal intervals, and hence exactly one
zero \(z(a)\in I_D\); see
\cite[Lemma~2.2]{AndrewsClutterbuckHauer2021} and \cite[Theorem~5 and the subsequent corollary]
{walter-sturm-liouville-p}. Set
\[
r:=\frac{z(a)+D/2}{D}\in(0,1).
\]
The two nodal intervals have lengths \(rD\) and \((1-r)D\).
The restriction of \(u_2(\cdot;a)\) to either nodal interval has a
fixed sign and is therefore its first Dirichlet eigenfunction.

The centers of the two nodal intervals are $-\frac{(1-r)D}{2}$ and $\frac{rD}{2}.$
Translating them to centered intervals gives
\begin{equation}\label{eq:nodal-left-centered}
\lambda_{2,p}(I_D,ax)
=
-\frac{a(1-r)D}{2}
+
\eta(rD,a)
\end{equation}
and
\begin{equation}\label{eq:nodal-right-centered}
\lambda_{2,p}(I_D,ax)
=
\frac{arD}{2}
+
\eta((1-r)D,a).
\end{equation}
Multiplying the first identity by \(r\), the second by \(1-r\), and thus,
\begin{equation}\label{eq:lambda1-weighted-eta}
\lambda_{2,p}(I_D,ax)
=
r\eta(rD,a)
+
(1-r)\eta((1-r)D,a).
\end{equation}
Since \(I_D\) is centered, $\lambda_{1,p}(I_D,ax)=\eta(D,a).$ Set \(A:=aD^{p+1}\). From
\eqref{eq:eta-scaling-global}--\eqref{eq:lambda1-weighted-eta},
\[
\begin{aligned}
D^p\Gamma_p(I_D,ax)
={}&
r^{1-p}\widehat\mu(Ar^{p+1})
+
(1-r)^{1-p}\widehat\mu(A(1-r)^{p+1})
-
\widehat\mu(A)\\
={}&
\mu(0)
\bigl(r^{1-p}+(1-r)^{1-p}-1\bigr)\\
&-
r^{1-p}F(Ar^{p+1})
-
(1-r)^{1-p}F(A(1-r)^{p+1})
+
F(A).
\end{aligned}
\]
By \eqref{eq:D-convex-scaling},
\[
F(Ar^{p+1})
\leq
r^{p+1}F(A),
\qquad
F(A(1-r)^{p+1})
\leq
(1-r)^{p+1}F(A).
\]
Consequently,
\[
\begin{aligned}
&r^{1-p}F(Ar^{p+1})
+
(1-r)^{1-p}F(A(1-r)^{p+1})\leq
\bigl(r^2+(1-r)^2\bigr)F(A)
\leq
F(A),
\end{aligned}
\]
and therefore
\[
D^p\Gamma_p(I_D,ax)
\geq
\mu(0)
\bigl(r^{1-p}+(1-r)^{1-p}-1\bigr).
\]
Since the function $r\mapsto r^{1-p}+(1-r)^{1-p}$
is strictly convex and symmetric about \(r=1/2\), $r^{1-p}+(1-r)^{1-p}\geq2^p.$
Moreover, $\mu(0)=(p-1)\pi_p^p.$
Thus
\[
\Gamma_p(I_D,ax)
\geq
(p-1)(2^p-1)
\left(\frac{\pi_p}{D}\right)^p.
\]

For \(a=0\), by \eqref{eq:first eigen}, we know that equality holds. Conversely, if equality holds, then necessarily
\(r=1/2\). Equations \eqref{eq:nodal-left-centered} and
\eqref{eq:nodal-right-centered} then give \(a=0\).

\smallskip

We now return to an arbitrary convex potential. By the linear comparison
\eqref{eq:convex-potential-linear-comparison}, there exists an affine
function \(\ell(x)=ax+b\) such that
\[
\Gamma_p(I_D,V)
\geq
\Gamma_p(I_D,\ell)
=
\Gamma_p(I_D,ax).
\]
The linear-potential estimate proved above then yields
\[
\Gamma_p(I_D,V)
\geq
\Gamma_p(I_D,ax)
\geq
\Gamma_p(I_D,0)
=
(p-1)(2^p-1)
\left(\frac{\pi_p}{D}\right)^p.
\]
If equality holds and \(V\) is not affine, then the first comparison is
strict, a contradiction. Thus \(V(x)=ax+b\). The equality
characterization for linear potentials forces \(a=0\), so \(V\) is
constant. Conversely, constant potentials shift all eigenvalues equally and hence attain equality.
\end{proof}

\subsection{Collapse to a segment}

We now establish the asymptotic behavior of the fundamental gap on the
collapsing smooth convex domains introduced in
Theorem~\ref{thm:smooth-quadratic-threshold-collapse}.

\begin{proof}[\textbf{Proof of
Theorem~\ref{thm:smooth-quadratic-threshold-collapse}.}]
Set
\[
I_D:=\left(-\frac D2,\frac D2\right),
\qquad
C_\varepsilon:=I_D\times B_\varepsilon^{N-1},
\qquad
s_\varepsilon:=2^{-1/(2m_\varepsilon)}.
\]
Since
\[
\left(\frac{2x}{D}\right)^{2m_\varepsilon}
+
\left(\frac{|y|}{\varepsilon}\right)^{2m_\varepsilon}
<1
\]
defines a bounded convex sublevel set of a smooth convex function whose
gradient does not vanish on the boundary, each
\(\Omega_\varepsilon\) is bounded and convex with \(C^\infty\) boundary. Moreover, $s_\varepsilon C_\varepsilon
\subset
\Omega_\varepsilon
\subset
C_\varepsilon.$
Therefore
\[
s_\varepsilon\sqrt{D^2+4\varepsilon^2}
\leq
\operatorname{diam}(\Omega_\varepsilon)
\leq
\sqrt{D^2+4\varepsilon^2},
\]
and hence $\operatorname{diam}(\Omega_\varepsilon)\rightarrow D.$ Since $s_\varepsilon^{-p}
=
2^{p/(2m_\varepsilon)},$
we have $s_\varepsilon^{-p}-1
=
O(m_\varepsilon^{-1})
=
o(\varepsilon^2),$
where the last relation follows from
\(m_\varepsilon\varepsilon^2\to\infty\).
By domain monotonicity and the scaling law
\(\lambda_{k,p}(t\Omega,0)=t^{-p}\lambda_{k,p}(\Omega,0)\),
\begin{equation}
\label{eq:smooth-domain-eigenvalue-sandwich}
\lambda_{k,p}(C_\varepsilon,0)
\leq
\lambda_{k,p}(\Omega_\varepsilon,0)
\leq
s_\varepsilon^{-p}\lambda_{k,p}(C_\varepsilon,0),
\qquad k=1,2.
\end{equation}

We now distinguish the three regimes.

\smallskip
\noindent\textbf{(i) Case \(1<p<2\).}
Let $\mu_{p,N}:=\lambda_{1,p}(B_1^{N-1},0).$
For every \(u\in W_0^{1,p}(C_\varepsilon)\), applying the variational characterization of the first Dirichlet $p$-Laplacian eigenvalue on transverse section
\(B_\varepsilon^{N-1}\) gives
\[
\int_{C_\varepsilon}|\nabla u|^p
\geq
\int_{C_\varepsilon}|\nabla_yu|^p
\geq
\mu_{p,N}\varepsilon^{-p}
\int_{C_\varepsilon}|u|^p.
\]
Thus $\lambda_{1,p}(C_\varepsilon,0)
\geq
\mu_{p,N}\varepsilon^{-p}.$ Let \(\psi_p\) be the positive first Dirichlet eigenfunction of
\(B_1^{N-1}\), normalized by
\[
\int_{B_1^{N-1}}\psi_p^p\,dz=1,
\qquad
\int_{B_1^{N-1}}|\nabla\psi_p|^p\,dz=\mu_{p,N},
\]
and choose nonnegative functions
\(h_1,h_2\in C_c^\infty(I_D)\) with disjoint supports. Define
\[
v_{i,\varepsilon}(x,y)
:=
h_i(x)\psi_p\!\left(\frac{y}{\varepsilon}\right),
\qquad i=1,2.
\]
After the change of variables \(y=\varepsilon z\),
\[
\int_{C_\varepsilon}|v_{i,\varepsilon}|^p
=
\varepsilon^{N-1}\int_{I_D}|h_i|^p\,dx,\quad
\int_{C_\varepsilon}|\nabla v_{i,\varepsilon}|^p
=
\varepsilon^{N-1-p}
\int_{I_D\times B_1^{N-1}}
\left(A_i^2+\varepsilon^2B_i^2\right)^{p/2}\,dx\,dz,
\]
where $A_i:=|h_i||\nabla\psi_p|$ and $B_i:=|h_i'|\psi_p.$
Hence
\[
\frac{\displaystyle
\int_{C_\varepsilon}|\nabla v_{i,\varepsilon}|^p}
{\displaystyle
\int_{C_\varepsilon}|v_{i,\varepsilon}|^p}
=
\mu_{p,N}\varepsilon^{-p}
+
\frac{\varepsilon^{-p}R_{i,\varepsilon}}
{\displaystyle\int_{I_D}|h_i|^p\,dx},\quad R_{i,\varepsilon}
:=
\int_{I_D\times B_1^{N-1}}
\left[
\left(A_i^2+\varepsilon^2B_i^2\right)^{p/2}
-A_i^p
\right]dxdz.
\]

We claim that $\varepsilon^{-p}R_{i,\varepsilon}\rightarrow0.$
Fix \(\delta>0\). Since \(1<p<2\), the map \(t\mapsto t^{p/2}\) is concave on
\([0,\infty)\). Hence, on \(\{A_i<\delta\}\),
\[
0\le
\left(A_i^2+\varepsilon^2B_i^2\right)^{p/2}-A_i^p
\le
\varepsilon^pB_i^p.
\]
On the other hand, on \(\{A_i\ge\delta\}\), the mean value theorem gives
\[
0\le
\left(A_i^2+\varepsilon^2B_i^2\right)^{p/2}-A_i^p
\le
\frac p2 A_i^{p-2}\varepsilon^2B_i^2
\le
\frac p2\delta^{p-2}\varepsilon^2B_i^2.
\]
Therefore
\[
\varepsilon^{-p}R_{i,\varepsilon}
\le
\int_{\{A_i<\delta\}}B_i^p
+
\frac p2\delta^{p-2}\varepsilon^{2-p}\int_{\{A_i\ge \delta\}} B_i^2.
\]
Since \(2-p>0\), letting \(\varepsilon\downarrow0\) yields
\[
\limsup_{\varepsilon\downarrow0}
\varepsilon^{-p}R_{i,\varepsilon}
\le
\int_{\{A_i<\delta\}}B_i^p.
\]

It remains to let \(\delta\downarrow0\). Since $A_i=|h_i||\nabla\psi_p|,\,
B_i=|h_i'|\psi_p,$
we have \(B_i=0\) a.e. on \(\{A_i=0\}\): indeed,
\(h_i'=0\) on \(\{h_i=0\}\), while
\(\{\nabla\psi_p=0\}\) has measure zero because the positive first
Dirichlet \(p\)-eigenfunction of the ball is radial and strictly
decreasing in the radial variable. Hence, using \(B_i^p\in L^1\), by dominated convergence theorem, we prove the claim.
Consequently,
\[
\frac{\displaystyle
\int_{C_\varepsilon}|\nabla v_{i,\varepsilon}|^p}
{\displaystyle
\int_{C_\varepsilon}|v_{i,\varepsilon}|^p}
=
\mu_{p,N}\varepsilon^{-p}+o(1),
\qquad i=1,2.
\]

Let $w_{i,\varepsilon}
:=
\frac{v_{i,\varepsilon}}
{\|v_{i,\varepsilon}\|_{L^p(C_\varepsilon)}},$
and set
\[
\mathcal A_\varepsilon
:=
\left\{
aw_{1,\varepsilon}+bw_{2,\varepsilon}:
|a|^p+|b|^p=1
\right\}.
\]
Since the supports are disjoint,
\(\mathcal A_\varepsilon\) is a compact symmetric subset of the
\(L^p\)-unit sphere with Krasnosel'ski\u{\i} genus two, and
\[
\sup_{u\in\mathcal A_\varepsilon}
\int_{C_\varepsilon}|\nabla u|^p
\leq
\max_{i=1,2}
\frac{\displaystyle
\int_{C_\varepsilon}|\nabla v_{i,\varepsilon}|^p}
{\displaystyle
\int_{C_\varepsilon}|v_{i,\varepsilon}|^p}.
\]
Thus $\lambda_{2,p}(C_\varepsilon,0)
\leq
\mu_{p,N}\varepsilon^{-p}+o(1).$
Together with $\lambda_{1,p}(C_\varepsilon,0)
\geq
\mu_{p,N}\varepsilon^{-p},$
this gives
\[\lambda_{2,p}(C_\varepsilon,0)
-
\lambda_{1,p}(C_\varepsilon,0)
\longrightarrow0.\]
Finally, by \eqref{eq:smooth-domain-eigenvalue-sandwich},
\[
\begin{aligned}
0
&\leq
\lambda_{2,p}(\Omega_\varepsilon,0)
-
\lambda_{1,p}(\Omega_\varepsilon,0)\leq
\lambda_{2,p}(C_\varepsilon,0)
-
\lambda_{1,p}(C_\varepsilon,0)
+
\left(s_\varepsilon^{-p}-1\right)
\lambda_{2,p}(C_\varepsilon,0).
\end{aligned}
\]
Since $s_\varepsilon^{-p}-1=O(m_\varepsilon^{-1}),$
we have $\left(s_\varepsilon^{-p}-1\right)
\lambda_{2,p}(C_\varepsilon,0)\rightarrow0.$
Hence
\[
\lambda_{2,p}(\Omega_\varepsilon,0)
-
\lambda_{1,p}(\Omega_\varepsilon,0)
\longrightarrow0.
\]

\smallskip
\noindent\textbf{(ii) Case \(p=2\).}
Let $0<\mu_{1,N}<\mu_{2,N}\leq\mu_{3,N}\leq\cdots$
denote the Dirichlet eigenvalues of \(B_1^{N-1}\). By separation of
variables, the Dirichlet spectrum of $C_\varepsilon=I_D\times B_\varepsilon^{N-1}$
is
\[
\left\{
\frac{\mu_{j,N}}{\varepsilon^2}
+
\left(\frac{k\pi}{D}\right)^2
:
j,k\in\mathbb N
\right\}.
\]
Hence $\lambda_{1,2}(C_\varepsilon,0)
=
\frac{\mu_{1,N}}{\varepsilon^2}
+
\frac{\pi^2}{D^2}.$
Moreover,
\[
\lambda_{2,2}(C_\varepsilon,0)
=
\min\left\{
\frac{\mu_{1,N}}{\varepsilon^2}
+
\frac{4\pi^2}{D^2},
\,
\frac{\mu_{2,N}}{\varepsilon^2}
+
\frac{\pi^2}{D^2}
\right\}.
\]
Then, for all sufficiently small \(\varepsilon\), $\lambda_{2,2}(C_\varepsilon,0)
=
\frac{\mu_{1,N}}{\varepsilon^2}
+
\frac{4\pi^2}{D^2}.$
Therefore
\[\lambda_{2,2}(C_\varepsilon,0)
-
\lambda_{1,2}(C_\varepsilon,0)
=
\frac{3\pi^2}{D^2}.\]

By \eqref{eq:smooth-domain-eigenvalue-sandwich}, we obtain
\[
0
\leq
\lambda_{k,2}(\Omega_\varepsilon,0)
-
\lambda_{k,2}(C_\varepsilon,0)
\leq
\left(s_\varepsilon^{-2}-1\right)
\lambda_{k,2}(C_\varepsilon,0),
\qquad k=1,2.
\]
Since $s_\varepsilon^{-2}-1
=
O(m_\varepsilon^{-1}),\,\lambda_{k,2}(C_\varepsilon,0)
=
O(\varepsilon^{-2}),$
we obtain
\[
\lambda_{k,2}(\Omega_\varepsilon,0)
-
\lambda_{k,2}(C_\varepsilon,0)
=
O\!\left(\frac{1}{m_\varepsilon\varepsilon^2}\right)
=
o(1),
\qquad k=1,2.
\]
Consequently, we obtain
\[
\lambda_{2,2}(\Omega_\varepsilon,0)
-
\lambda_{1,2}(\Omega_\varepsilon,0)
=
\frac{3\pi^2}{D^2}+o(1),
\]

\noindent\textbf{(iii) Case \(p>2\).}
By Corollary~\ref{cor:enhanced-fundamental-gap}, the enhanced
fundamental-gap estimate gives
\[
\Gamma_p(\Omega_\varepsilon,0)
\gtrsim
\lambda_{1,p}(\Omega_\varepsilon,0)^{\frac{p-2}{p}}
\operatorname{diam}(\Omega_\varepsilon)^{-2}.
\]
Since $\Omega_\varepsilon
\subset
C_\varepsilon,$
domain monotonicity and the transverse Poincar\'e inequality give
\[
\lambda_{1,p}(\Omega_\varepsilon,0)
\geq
\lambda_{1,p}(C_\varepsilon,0)
\geq
\mu_{p,N}\varepsilon^{-p}.
\]
Together with
\(\operatorname{diam}(\Omega_\varepsilon)\to D\), this yields $\Gamma_p(\Omega_\varepsilon,0)
\gtrsim
\varepsilon^{2-p}.$
Since \(p>2\), the right-hand side tends to \(+\infty\), we complete the proof.
\end{proof}

\section{The Optimal Diameter-normalized Gap}
\label{sec:optimal-diameter-normalized-gap}

This section proves the existence statement and the asymptotic properties of the optimal diameter-normalized gap. We first treat attainment for fixed \(p>2\), and then study the dependence of \(\mathcal G_{p,N}\) on \(p\) and \(N\).

\subsection{Existence of minimizing domains}

We first establish the existence theorem for the optimal diameter-normalized gap. The cases \(N=1\), \(1<p<2\), and \(p=2\) follow from the preceding one-dimensional and collapsing-domain results, so the main point is the attainment of the infimum for \(N\geq2\) and \(p>2\). This follows from the enhanced fundamental-gap estimate in Corollary \ref{cor:enhanced-fundamental-gap}, which excludes collapse of diameter-normalized minimizing sequences, together with compactness of convex bodies.

\begin{proof}[\textbf{Proof of Theorem~\ref{thm:existence-gap-minimizer}.}]
	The cases \(N=1\), \(1<p<2\), and \(p=2\) follow respectively from
	Theorem~\ref{thm:sharp-one-dimensional-fundamental-gap},
	Theorem~\ref{thm:smooth-quadratic-threshold-collapse}, and the
	classical sharp fundamental gap result in
	\cite{AndrewsClutterbuck2011}. It therefore remains only to consider the
	case \(N\geq2\) and \(p>2\).
	
	By the scaling relations $\lambda_{k,p}(t\Omega,0)
	=
	t^{-p}\lambda_{k,p}(\Omega,0),\,k=1,2$,
	it is enough to minimize \(\Gamma_p(\Omega,0)\) under the normalization
	\(\operatorname{diam}(\Omega)=1\). The positivity of
	\(\mathcal G_{p,N}\) follows from
	Corollary~\ref{cor:enhanced-fundamental-gap}:
	\begin{equation}
		\label{eq:enhanced-gap-arbitrary-convex}
		\Gamma_p(\Omega,0)
		\geq
		c_{p,N}\lambda_{1,p}(\Omega,0)^{\frac{p-2}{p}}
		\operatorname{diam}(\Omega)^{-2},
		\qquad
		c_{p,N}:=\frac{2}{pC_{p,N}^{\mathrm{opt}}}>0,
	\end{equation}
	for every \(\Omega\in\mathcal C_N\). Let \(\{\Omega_j\}_{j\geq1}\subset\mathcal C_N\) be a minimizing sequence,
	normalized so that
	\[
	\operatorname{diam}(\Omega_j)=1,
	\qquad
	\Gamma_p(\Omega_j,0)\longrightarrow\mathcal G_{p,N}.
	\]
	In particular, for some \(M>0\), $\Gamma_p(\Omega_j,0)\leq M.$
	Since \(p>2\), \eqref{eq:enhanced-gap-arbitrary-convex} yields
	\[
	\lambda_{1,p}(\Omega_j,0)
	\leq
	\left(\frac{M}{c_{p,N}}\right)^{\frac{p}{p-2}}
	=: \Lambda.
	\]
	Let \(r_j\) denote the inradius of \(\Omega_j\). By the
	\(p\)-Laplacian Hersch--Protter inequality
	\cite[Theorem~1.2]{Brasco2018Inradius}, $\lambda_{1,p}(\Omega_j,0)
	\geq
	c_p\,r_j^{-p},$
	where \(c_p>0\) depends only on \(p\). Consequently,
	$r_j\geq r_0>0.$
	After translating an incenter of each \(\Omega_j\) to the origin, we have
	\[
	B_{r_0}(0)\subset\Omega_j\subset B_1(0).
	\]
	By the Blaschke selection theorem
	\cite[Theorem~1.8.5]{Schneider2014}, after passing to a subsequence,
	$\overline{\Omega_j}\longrightarrow K$
	in the Hausdorff metric for some compact convex set \(K\). Moreover,
	$B_{r_0}(0)\subset K$ and $\operatorname{diam}(K)=1.$
	Hence \(K\) has nonempty interior. Setting $\Omega_*:=\operatorname{int}K,$
	we obtain \(\Omega_*\in\mathcal C_N\) and $\operatorname{diam}(\Omega_*)=1.$
	
	Set $\delta_j
	:=
	d_H(\overline{\Omega_j},K)
	\rightarrow0.$
	Since $B_{r_0}(0)\subset K$, the support function characterization of
	Hausdorff distance and inclusion
	\cite[Section 1.7.1]{Schneider2014} gives, with
	$\eta_j:=\delta_j/r_0\rightarrow0,$
	\[
	(1-\eta_j)K
	\subset
	\overline{\Omega_j}
	\subset
	(1+\eta_j)K.
	\]
	Since \(\Omega_*=\operatorname{int}K\), it follows that
	$(1-\eta_j)\Omega_*
	\subset
	\Omega_j
	\subset
	(1+\eta_j)\Omega_*.$
	
	By domain monotonicity and the scaling relation for the variational
	\(p\)-eigenvalues, for \(k=1,2\),
	\[
	(1+\eta_j)^{-p}\lambda_{k,p}(\Omega_*,0)
	\leq
	\lambda_{k,p}(\Omega_j,0)
	\leq
	(1-\eta_j)^{-p}\lambda_{k,p}(\Omega_*,0).
	\]
	Since \(\eta_j\to0\), letting \(j\to\infty\) yields
	$\lambda_{k,p}(\Omega_j,0)
	\rightarrow
	\lambda_{k,p}(\Omega_*,0),\,k=1,2.$
	Consequently, $\Gamma_p(\Omega_j,0)
	\rightarrow
	\Gamma_p(\Omega_*,0),$ and hence $\Gamma_p(\Omega_*,0)
	=
	\mathcal G_{p,N}.$
	
	\smallskip
	
	We finally prove that every family of normalized minimizers satisfies
	$r_{\Omega_p^*}\rightarrow0$ as \(p\downarrow2\).
	Suppose otherwise. Then there exist \(p_j\downarrow2\), \(r_0>0\), and
	minimizers
	\[
	\Omega_j:=\Omega_{p_j}^*,
	\qquad
	\operatorname{diam}(\Omega_j)=1,
	\qquad
	r_{\Omega_j}\geq r_0.
	\]
	After translating an incenter of each \(\Omega_j\) to the origin, we have
	$B_{r_0}(0)\subset\Omega_j\subset B_1(0).$
	Repeating the preceding compactness argument, after passing to a
	subsequence there exist a nondegenerate bounded convex domain
	\(\Omega_\infty\), with \(\operatorname{diam}(\Omega_\infty)=1\), and
	\(\eta_j\downarrow0\) such that
	\[
	(1-\eta_j)\Omega_\infty
	\subset
	\Omega_j
	\subset
	(1+\eta_j)\Omega_\infty.
	\]
	Hence, by domain monotonicity and scaling,
	\[
	(1+\eta_j)^{-p_j}
	\lambda_{k,p_j}(\Omega_\infty,0)
	\leq
	\lambda_{k,p_j}(\Omega_j,0)
	\leq
	(1-\eta_j)^{-p_j}
	\lambda_{k,p_j}(\Omega_\infty,0),
	\qquad k=1,2.
	\]
	Since \(\Omega_\infty\) is fixed, the continuity of the
	variational eigenvalues with respect to \(p\)
	\cite{DegiovanniMarzocchi2015} gives
	$\lambda_{k,p_j}(\Omega_\infty,0)
	\rightarrow
	\lambda_{k,2}(\Omega_\infty,0),\,k=1,2.$
	Consequently,
	$\mathcal G_{p_j,N}
	=
	\Gamma_{p_j}(\Omega_j,0)
	\rightarrow
	\Gamma_2(\Omega_\infty,0).$
	On the other hand, for every \(\delta>0\),
	Theorem~\ref{thm:smooth-quadratic-threshold-collapse} provides a fixed
	convex domain \(\Omega^\delta\) of diameter one such that $\Gamma_2(\Omega^\delta,0)<3\pi^2+\delta.$ Thus, we obtain
	\[
	\limsup_{p\downarrow2}\mathcal G_{p,N}
	\leq
	\Gamma_2(\Omega^\delta,0)
	<
	3\pi^2+\delta.
	\]
	Letting \(\delta\downarrow0\), we obtain
	\(\limsup_{p\downarrow2}\mathcal G_{p,N}\leq3\pi^2.\)
	Since \(\Omega_\infty\) is a nondegenerate
	\(N\)-dimensional convex domain,
	\cite[Theorem~1]{AmatoBucurFragala2024} yields
	$\Gamma_2(\Omega_\infty,0)>3\pi^2,$
	a contradiction. Thus $r_{\Omega_p^*}\rightarrow0$ as \(p\downarrow2\),
	and the proof is complete.
\end{proof}

\subsection{Asymptotics in \texorpdfstring{$p$}{p} and the dimension}

We next prove the three asymptotic statements in Theorem~\ref{thm:optimal-gap-asymptotics}. Part~\textnormal{\rm (i)} follows from the continuity of the optimal degenerate weighted Poincar\'e constant $C_{p,N}^{\mathrm{opt}}$ at \(p=2\); part~\textnormal{\rm (ii)} combines the enhanced gap estimate \eqref{eq:enhanced-fundamental-gap} with a product domain test; and part~\textnormal{\rm (iii)} uses the \(p\to\infty\) limits of the first two Dirichlet eigenvalues.

\begin{proof}[\textbf{Proof of Theorem~\ref{thm:optimal-gap-asymptotics}.}]
\textnormal{\rm (i)} follows from Corollary \ref{cor:continuity-weighted-poincare} and Theorem~\ref{thm:existence-gap-minimizer}.
	
    We next prove \textnormal{\rm (ii)}. Let
	\(\Omega\in\mathcal C_N\) and
	\(D=\operatorname{diam}(\Omega)\). By the enhanced gap estimate \eqref{eq:enhanced-fundamental-gap},
	\[
	\Gamma_p(\Omega,0)
	\geq
	\frac{2}{p\,C_{p,N}^{\mathrm{opt}}} D^{-2}
	\lambda_{1,p}(\Omega,0)^{\frac{p-2}{p}}.
	\]
	By the Faber--Krahn inequality
	\cite[Theorem~1]{KawohlFridman2003} and the isodiametric inequality
	\cite[Theorem~8.8]{Gruber2007}, the ball \(\Omega^\ast=B_r^N\)
	with \(|\Omega^\ast|=|\Omega|\) satisfies \(r\leq D/2\), and hence
	\[
	\lambda_{1,p}(\Omega,0)
	\geq
	\lambda_{1,p}(\Omega^\ast,0)
	\geq
	\lambda_{1,p}(B_{D/2}^N,0).
	\]
	Moreover, by the \(p\)-Cheeger inequality
	\cite[Theorem~3 and Remark~5]{KawohlFridman2003}, $\lambda_{1,p}(B_{D/2}^N,0)
	\geq
	\left(\frac{2N}{pD}\right)^p.$
	Hence
	\[
	D^p\Gamma_p(\Omega,0)
	\geq
	\frac{2}{p\,C_{p,N}^{\mathrm{opt}}}\left(\frac2p\right)^{p-2}N^{p-2}.
	\]
	Using
	Theorem~\ref{thm:ground-state-weighted-spectral-gap}\textnormal{(i)}, we obtain $\mathcal G_{p,N}\geq c_pN^{p-2}$
	for some \(c_p>0\) depending only on \(p\).
	
	\smallskip
	
	For the reverse inequality, consider $\Omega_N:=B_1^{N-1}\times(-1,1).$
	Then $\operatorname{diam}(\Omega_N)=2\sqrt2.$
	Let \(\phi_{N-1}\) be a corresponding positive first eigenfunction normalized by $\int_{B_1^{N-1}}\phi_{N-1}^p\,dx=1.$
	Testing the Rayleigh quotient with \(v(x)=1-|x|\), we have $\int_{B_1^{N-1}}|\nabla v|^p\,dx
	=
	|B_1^{N-1}|,$ while by polar coordinates,
	$$
	\int_{B_1^{N-1}}(1-|x|)^p\,dx
	=
	|S^{N-2}|
	\int_0^1(1-r)^p r^{N-2}\,dr
	=
	|S^{N-2}|
	\frac{\Gamma(N-1)\Gamma(p+1)}
	{\Gamma(N+p)}.
	$$
	Since $|B_1^{N-1}|
	=
	\frac{|S^{N-2}|}{N-1},$
	the Rayleigh characterization yields
	\begin{equation}
		\label{eq:lambda-m-growth}
		\lambda_{1,p}(B_1^{N-1})
		\leq
		\frac{\Gamma(N+p)}
		{\Gamma(N)\Gamma(p+1)}
		\leq
		C_p N^p.
	\end{equation}

	Fix a nontrivial odd function $\psi\in W_0^{1,p}(-1,1),$
	and define
	\[
	w(x,t):=\phi_{N-1}(x)\psi(t),
	\qquad
	(x,t)\in\Omega_N.
	\]
	Since \(\phi_{N-1}>0\) in \(B_1^{N-1}\) and \(\psi\) is odd, we have
	\[
	w(x,-t)=-w(x,t),
	\quad
	w_+(x,-t)=w_-(x,t),
	\]
	where $w_\pm := \max\{\pm w, 0\}$ denote the positive and negative parts of $w$, respectively. Hence,
	\[
	\int_{\Omega_N}|w_+|^p\,dx\,dt
	=
	\int_{\Omega_N}|w_-|^p\,dx\,dt,
	\qquad
	\int_{\Omega_N}|\nabla w_+|^p\,dx\,dt
	=
	\int_{\Omega_N}|\nabla w_-|^p\,dx\,dt.
	\]
	After \(L^p\)-normalization, set
	\[
	e_+:=\frac{w_+}{\|w_+\|_{L^p(\Omega_N)}},
	\qquad
	e_-:=\frac{w_-}{\|w_-\|_{L^p(\Omega_N)}}.
	\]
	Since \(w_+w_-=0\) a.e., the set $\mathcal A
	:=
	\left\{
	ae_+-be_-:\ |a|^p+|b|^p=1
	\right\}$
	is contained in the \(L^p\)-unit sphere of
	\(W_0^{1,p}(\Omega_N)\). It is compact and symmetric, and is odd
	homeomorphic to the \(l^p\)-unit sphere in \(\mathbb R^2\).
	Consequently, $\gamma(\mathcal A)=2.$ 
	Therefore, by the genus characterization \eqref{eq:eigenvalue}, we have
	\[
	\lambda_{2,p}(\Omega_N)
	\leq
	\frac{\displaystyle\int_{\Omega_N}|\nabla w|^p\,dx\,dt}
	{\displaystyle\int_{\Omega_N}|w|^p\,dx\,dt}.
	\]
	For \(p>2\) and \(a,b\geq0\), we use the elementary inequality
	\[
	(a^2+b^2)^{p/2}
	\leq
	a^p+C_p\bigl(a^{p-2}b^2+b^p\bigr).
	\]
	Applying this with $a=|\psi|\,|\nabla\phi_{N-1}|,
	\,
	b=\phi_{N-1}|\psi'|,$
	we obtain
	\[
	|\nabla w|^p
	\leq
	|\psi|^p|\nabla\phi_{N-1}|^p
	+
	C_p|\psi|^{p-2}|\psi'|^2
	\phi_{N-1}^2|\nabla\phi_{N-1}|^{p-2}
	+
	C_p\phi_{N-1}^p|\psi'|^p.
	\]
	Moreover, since $\int_{B_1^{N-1}}|\nabla\phi_{N-1}|^p\,dx
	=
	\lambda_{1,p}(B_1^{N-1}),$
	H\"older's inequality yields
	\[
	\int_{B_1^{N-1}}
	\phi_{N-1}^2|\nabla\phi_{N-1}|^{p-2}\,dx
	\leq
	\lambda_{1,p}(B_1^{N-1})^{\frac{p-2}{p}}.
	\]
	Integrating in \(x\) and \(t\), we obtain
	\[
	\frac{\displaystyle\int_{\Omega_N}|\nabla w|^p\,dx\,dt}
	{\displaystyle\int_{\Omega_N}|w|^p\,dx\,dt}
	\leq
	\lambda_{1,p}(B_1^{N-1})
	+
	C_p\lambda_{1,p}(B_1^{N-1})^{\frac{p-2}{p}}
	+
	C_p.
	\]
	On the other hand, for every \(u\in W_0^{1,p}(\Omega_N)\),
	\[
	\int_{B_1^{N-1}}|\nabla_x u(x,t)|^p\,dx
	\geq
	\lambda_{1,p}(B_1^{N-1})
	\int_{B_1^{N-1}}|u(x,t)|^p\,dx.
	\]
	Integrating in \(t\) and using \(|\nabla u|\geq|\nabla_x u|\), we obtain $\lambda_{1,p}(\Omega_N)
	\geq
	\lambda_{1,p}(B_1^{N-1}).$
	Thus
	\[
	\begin{aligned}
		\Gamma_p(\Omega_N,0)
		&=
		\lambda_{2,p}(\Omega_N)-\lambda_{1,p}(\Omega_N)\leq
		C_p\left(
		\lambda_{1,p}(B_1^{N-1})^{\frac{p-2}{p}}+1
		\right)
		\leq
		C_p N^{p-2},
	\end{aligned}
	\]
	where we used \eqref{eq:lambda-m-growth}. Since $\operatorname{diam}(\Omega_N)=2\sqrt2$, and hence $\mathcal G_{p,N}\le C_pN^{p-2}$. This proves \textnormal{\rm (ii)}.
	
	\smallskip
	
	We finally prove \textnormal{\rm (iii)}. $N=1$ follows from Theorem~\ref{thm:existence-gap-minimizer}. Next, we focus on $N\ge 2.$ Since
	\(\operatorname{diam}(B_1^N)=2\), $\mathcal G_{p,N}^{1/p}
	\leq
	2\,\Gamma_p(B_1^N,0)^{1/p}.$
	For a bounded domain \(U\), let \(R(U)\) denote its inradius and let
	\(R_2(U)\) be the largest \(r>0\) such that \(U\) contains two
	disjoint balls of radius \(r\). By
	\cite[Lemma~1.5]{JuutinenLindqvistManfredi1999} and
	\cite[Theorem~4.1]{JuutinenLindqvist2005}, as \(p\to\infty\) we have
	\[
	\lambda_{1,p}(U)^{1/p}
	\longrightarrow
	\frac1{R(U)},
	\qquad
	\lambda_{2,p}(U)^{1/p}
	\longrightarrow
	\frac1{R_2(U)}.
	\]
	Since $R(B_1^N)=1,\,
	R_2(B_1^N)=\frac12,$
	we obtain $\lambda_{1,p}(B_1^N)^{1/p}\to1,\,
	\lambda_{2,p}(B_1^N)^{1/p}\to2,$
	and therefore
	\[
	\Gamma_p(B_1^N,0)^{1/p}
	=
	\lambda_{2,p}(B_1^N)^{1/p}
	\left(
	1-
	\frac{\lambda_{1,p}(B_1^N)}
	{\lambda_{2,p}(B_1^N)}
	\right)^{1/p}
	\longrightarrow2.
	\]
	Thus $\limsup_{p\to\infty}
	\mathcal G_{p,N}^{1/p}
	\leq4.$
	Suppose now, by contradiction, that $\liminf_{p\to\infty}
	\mathcal G_{p,N}^{1/p}
	<4.$
	Then there exist \(\eta>0\), a sequence \(p_j\to\infty\), and
	\(\Omega_j\in\mathcal C_N\) such that, after scaling,
	\[
	\operatorname{diam}(\Omega_j)=1,
	\qquad
	\Gamma_{p_j}(\Omega_j,0)^{1/p_j}
	\leq4-\eta.
	\]
	Set $\lambda_{k,j}
	:=
	\lambda_{k,p_j}(\Omega_j,0),\,k=1,2.$ By the enhanced gap estimate \eqref{eq:enhanced-fundamental-gap}, we have $$\Gamma_{p_j}(\Omega_j, 0)^{1/p_j} \geq \left( \frac{2}{p_j C_{p_j, N}^{\mathrm{opt}}} \right)^{1/p_j} \left( \lambda_{1, j}^{1/p_j} \right)^{1 - \frac{2}{p_j}},$$
	Theorem~\ref{thm:ground-state-weighted-spectral-gap}\textnormal{(i)} implies $\left(C_{p,N}^{\mathrm{opt}}\right)^{1/p}\rightarrow1$ as $p\to\infty$, and hence
	\begin{equation}
		\label{eq:lambda1-limsup-large-p}
		\limsup_{j\to\infty}
		\lambda_{1,j}^{1/p_j}
		\leq4-\eta.
	\end{equation}

	Let \(B_{r_j}\) satisfy $|B_{r_j}|=|\Omega_j|.$
	By the Faber--Krahn inequality \cite[Theorem~1]{KawohlFridman2003},
	\[
	\lambda_{1,j}
	\geq
	\lambda_{1,p_j}(B_{r_j})
	=
	r_j^{-p_j}\lambda_{1,p_j}(B_1).
	\]
	Since $\lambda_{1,p}(B_1)^{1/p}\to1,$
	\eqref{eq:lambda1-limsup-large-p} yields $r_j\geq c>0$
	for all sufficiently large \(j\). Hence $|\Omega_j|\geq c_N>0.$
	
	\smallskip
	
	After translations and passing to a subsequence, the Blaschke selection theorem
	\cite[Theorem~1.8.5]{Schneider2014} yields $\overline{\Omega_j}\to K$
	in the Hausdorff metric for some compact convex set \(K\) with nonempty interior. Setting $\Omega_\infty:=\operatorname{int}K,$
	we have \(\operatorname{diam}(\Omega_\infty)=1\).
	
	For \(\varepsilon>0\), set
	\[
	U_\varepsilon
	:=
	\{x\in\Omega_\infty:
	\operatorname{dist}(x,\partial\Omega_\infty)>\varepsilon\},
	\qquad
	V_\varepsilon
	:=
	\Omega_\infty+B_\varepsilon .
	\]
	Thus, for every sufficiently small \(\varepsilon>0\) and sufficiently large \(j\), $U_\varepsilon\subset\Omega_j\subset V_\varepsilon$. Hence,
	\[
	\lambda_{k,p_j}(V_\varepsilon)
	\leq
	\lambda_{k,j}
	\leq
	\lambda_{k,p_j}(U_\varepsilon),
	\qquad k=1,2.
	\]
	Therefore, using the fixed-domain \(p\to\infty\) limits and letting
	\(\varepsilon\downarrow0\), we obtain
	\begin{equation}
		\label{eq:varying-domain-p-infty}
		\lambda_{1,j}^{1/p_j}
		\longrightarrow
		\frac1{R(\Omega_\infty)},
		\qquad
		\lambda_{2,j}^{1/p_j}
		\longrightarrow
		\frac1{R_2(\Omega_\infty)}.
	\end{equation}
	
	We next note that $R_2(\Omega_\infty)\leq \frac14.$
	Indeed, if \(B_r(x_1),B_r(x_2)\subset\Omega_\infty\) are disjoint, then $|x_1-x_2|\geq2r.$
	Writing $e:=\frac{x_2-x_1}{|x_2-x_1|},$
	for every \(0<\delta<r\),
	\[
	x_1-(r-\delta)e,\,
	x_2+(r-\delta)e\in\Omega_\infty,
	\]
	and hence $\operatorname{diam}(\Omega_\infty)
	\geq
	|x_1-x_2|+2(r-\delta).$
	Letting \(\delta\downarrow0\) gives
	\[
	1=\operatorname{diam}(\Omega_\infty)
	\geq |x_1-x_2|+2r
	\geq4r.
	\]
	Taking the supremum over such \(r\) yields $R_2(\Omega_\infty)\leq\frac14.$ On the other hand, \eqref{eq:lambda1-limsup-large-p} and
	\eqref{eq:varying-domain-p-infty} imply
	\[
	\frac1{R(\Omega_\infty)}
	\leq4-\eta
	<
	4
	\leq
	\frac1{R_2(\Omega_\infty)}.
	\]
	Therefore $q:=
	\frac{R_2(\Omega_\infty)}{R(\Omega_\infty)}
	<1,$
	and \eqref{eq:varying-domain-p-infty} gives $\left(
	\frac{\lambda_{1,j}}{\lambda_{2,j}}
	\right)^{1/p_j}
	\rightarrow q.$
	Choose \(q<\theta<1\). For all sufficiently large \(j\),
	\[
	0\leq
	\frac{\lambda_{1,j}}{\lambda_{2,j}}
	\leq
	\theta^{p_j}
	\longrightarrow0.
	\]
	Consequently, $\left(
	1-\frac{\lambda_{1,j}}{\lambda_{2,j}}
	\right)^{1/p_j}
	\rightarrow1.$
	Thus, by \eqref{eq:varying-domain-p-infty},
	\[
	\begin{aligned}
		\Gamma_{p_j}(\Omega_j,0)^{1/p_j}
		&=
		\lambda_{2,j}^{1/p_j}
		\left(
		1-\frac{\lambda_{1,j}}{\lambda_{2,j}}
		\right)^{1/p_j}\longrightarrow
		\frac1{R_2(\Omega_\infty)}
		\geq4,
	\end{aligned}
	\]
	contradicting $\Gamma_{p_j}(\Omega_j,0)^{1/p_j}
	\leq4-\eta.$
	Hence, $\lim_{p\to\infty}\mathcal G_{p,N}^{1/p}=4.$
\end{proof}

\section{Degeneration of Minimizing Domains and Ground-state Measures}
\label{sec:gap-minimizer-degeneration}

We now prove Theorem~\ref{thm:limits-gap-minimizers}. The main issue is to identify weak limits of the ground-state measures when the domains collapse as \(p\downarrow2\). We first establish compactness for the associated weighted quadratic forms, then recover a convex linear Schr\"odinger potential on the limiting affine space, and finally analyze the equality case of the sharp linear fundamental gap bound.

\subsection{Poincar\'e constants for varying ground-state measures}\label{natation}

For convenience, we first fix the notation used below. Let \(p>2\) and \(\Omega\in\mathcal C_N\). We write \(\lambda:=\lambda_{1,p}(\Omega,0)\), \(D:=\operatorname{diam}(\Omega)\), and let \(u\) be the positive first Dirichlet \(p\)-eigenfunction, normalized by \(\int_\Omega u^p\,dx=1\). Thus \(d\mu:=u^p\,dx\) is a probability measure. We further set
\begin{equation}\label{eq:defs-weight}
	r:=|\nabla\log u|, \qquad a:=\lambda^{-\frac{p-2}{p}}r^{p-2}, \qquad \mathcal Q_{p,\Omega}(g) := \int_\Omega a|\nabla g|^2\,d\mu.
\end{equation}
We recall that
\[
\Lambda_p(\Omega)
:=
\inf_{\substack{
		g\in C^{0,1}(\overline\Omega)\setminus\{0\}\\
		\int_\Omega g\,d\mu=0}}
\frac{\int_\Omega r^{p-2}|\nabla g|^2\,d\mu}
{\int_\Omega g^2\,d\mu}.
\]
Since replacing \(g\) by \(g-\int_\Omega g\,d\mu\) leaves its gradient
unchanged and satisfies
\[
\int_\Omega\left|g-\int_\Omega g\,d\mu\right|^2\,d\mu
=
\operatorname{Var}_\mu(g),
\]
we obtain
\[
\frac{1}{\Lambda_p(\Omega)}
=
\sup_{\substack{
		g\in C^{0,1}(\overline\Omega)\\
		g\not\equiv\mathrm{const}}}
\frac{\operatorname{Var}_\mu(g)}
{\displaystyle\int_\Omega r^{p-2}|\nabla g|^2\,d\mu}.
\]
Using the definition of \(\mathcal Q_{p,\Omega}\), we obtain
\begin{equation}\label{eq:C1}
	\frac{\lambda^{\frac{p-2}{p}}}
	{\Lambda_p(\Omega)}
	=
	\sup_{\substack{
			g\in C^{0,1}(\overline\Omega)\\
			g\not\equiv\mathrm{const}}}
	\frac{\operatorname{Var}_\mu(g)}
	{\mathcal Q_{p,\Omega}(g)}.
\end{equation}
Hence
\begin{equation}\label{eq:C}
	C_{p,N}^{\mathrm{opt}}
	=
	\sup_{\Omega\in\mathcal C_N}
	\frac{\lambda^{\frac{p-2}{p}}}
	{D^2\Lambda_p(\Omega)}
	=
	\sup_{\substack{
			\Omega\in\mathcal C_N,\;
			g\in C^{0,1}(\overline\Omega)\\
			g\not\equiv\mathrm{const}}}
	\frac{\operatorname{Var}_\mu(g)}
	{D^2\mathcal Q_{p,\Omega}(g)}.
\end{equation}
Let \(\nu\) be a log-concave probability measure on \(\mathbb R^N\).
Its support, denoted by \(\operatorname{supp}\nu\), is the smallest
closed set \(F\subset\mathbb R^N\) such that \(\nu(F)=1\). We set
\[
S:=\operatorname{aff}(\operatorname{supp}\nu),
\]
where \(\operatorname{aff}(E)\) is the smallest affine subspace of
\(\mathbb R^N\) containing \(E\). Thus \(\nu\) is supported on \(S\) and
is not supported on any proper affine subspace of \(S\),
\(\dim S\) may be strictly smaller than \(N\).

\smallskip

Fix \(x_0\in S\) and write
\[
S=x_0+E,
\qquad
E:=S-S,
\qquad
d:=\dim S,
\]
so that \(E\) is the linear subspace parallel to \(S\). We denote by
\(dx_S\) the \(d\)-dimensional Lebesgue measure induced by the
Euclidean structure on \(S\), and by \(C^\infty(S)\) the corresponding
space of smooth functions. For a differentiable function \(f\) on \(S\), we denote by
\(\nabla_S f\) its Euclidean gradient on \(S\). Equivalently, if \(F\)
is any smooth extension of \(f\) to a neighborhood of \(S\) in
\(\mathbb R^N\), then
\[
\nabla_S f
=
P_E(\nabla F)\big|_S,
\]
where \(P_E\) denotes the orthogonal projection onto \(E\).

We shall use the following classical facts about log-concave probability
measures. Set $A:=\operatorname{ri}(\operatorname{supp}\nu),$
where \(\operatorname{ri}(E)\) denotes the interior of \(E\) relative to
its affine hull. Thus,
\[
x\in\operatorname{ri}(E)
\quad\Longleftrightarrow\quad
x\in E
\ \text{and}\ 
B_r(x)\cap\operatorname{aff}(E)\subset E
\]
for some \(r>0\). If \(\nu\) is not a point mass, then, after identifying \(S\)
isometrically with \(\mathbb R^{\dim S}\),
\cite[Theorem~2.7]{SaumardWellner2014} yields
\[
d\nu=e^{-2v}\,dx_S
\qquad\text{on }A,
\]
for some finite convex function \(v:A\to\mathbb R\).
In particular, \(v\) is locally Lipschitz, and hence continuous, on \(A\). Finally, after the same identification of \(S\) with
\(\mathbb R^{\dim S}\),
\cite[Proposition~27]{BartheKlartag2020} implies that
\[
C_c^\infty(S)
\quad\text{is dense in}\quad
H^1(\nu).
\]

Throughout, weak convergence of finite Borel measures is understood in
the narrow sense, namely,
\[
\nu_j\rightharpoonup\nu
\quad\Longleftrightarrow\quad
\int_{\mathbb R^N}\phi\,d\nu_j
\longrightarrow
\int_{\mathbb R^N}\phi\,d\nu
\qquad
\text{for every }\phi\in C_b(\mathbb R^N),
\]
where \(C_b(\mathbb R^N)\) denotes the space of bounded continuous
functions on \(\mathbb R^N\). We shall use repeatedly the fact that weak limits of finite
log-concave measures are again log-concave; see
\cite[Proposition~3.6]{SaumardWellner2014}.

\smallskip

We first record a compactness result for Sobolev functions with respect
to varying log-concave measures. The argument does not require the affine dimensions of
\(\operatorname{supp}\nu_j\) to remain fixed, and therefore also covers
lower-dimensional limits, including a point mass.

\begin{lemma}\label{lemma:ordinary1}
	Let \(\nu_j\) be nonzero finite log-concave measures on a fixed
	Euclidean space, with supports contained in a common compact set, and
	assume $\nu_j\rightharpoonup\nu,\,\nu\neq0.$
	If
	\[
	\sup_j\left(
	\int |f_j|^2\,d\nu_j+\mathcal E_{\nu_j}(f_j)
	\right)<\infty,
	\]
	then, up to a subsequence, there exists \(f\in H^1(\nu)\) such that $f_j\nu_j\rightharpoonup f\nu,\,
	f_j^2\nu_j\rightharpoonup f^2\nu,$
	and
	\[
	\int |f_j|^2\,d\nu_j
	\longrightarrow
	\int |f|^2\,d\nu,
	\qquad
	\mathcal E_\nu(f)
	\leq
	\liminf_{j\to\infty}\mathcal E_{\nu_j}(f_j).
	\]
\end{lemma}

\begin{proof}
	We use the compactness criterion of
	\cite[Theorem~6.3(ii)]{GigliMondinoSavare2015}. Set
	\[
	X_j:=\operatorname{supp}\nu_j,
	\qquad
	X:=\operatorname{supp}\nu,
	\]
	with the Euclidean distance inherited from the common ambient space.
	For each \(j\), let
	\(S_j:=\operatorname{aff}(X_j)\), and set
	\(S:=\operatorname{aff}(X)\). After identifying \(S_j\) and \(S\)
	with their Euclidean models, the measures \(\nu_j\) and \(\nu\) have
	log-concave densities on their respective convex supports. Hence the
	metric-measure spaces $(X_j,|\cdot|,\nu_j),\,
	(X,|\cdot|,\nu)$
	satisfy \(\mathrm{CD}(0,\infty)\).
	
	Moreover, on every compact subset of the relative interior of the support,
the log-concave density is bounded above and away from zero. Hence the
metric Sobolev space agrees locally with the usual weighted Euclidean
Sobolev space on the corresponding affine hull. Using the density of smooth
functions in \(H^1(\nu)\), see \cite[Proposition~27]{BartheKlartag2020},
we obtain
\[
\operatorname{Ch}_{\nu_j}(f)
=
\frac12\mathcal E_{\nu_j}(f),
\qquad
\operatorname{Ch}_{\nu}(f)
=
\frac12\mathcal E_{\nu}(f).
\]
For a point mass, both sides vanish identically.

	Fix \(x\in\operatorname{supp}\nu\). Since
	\(\nu_j\rightharpoonup\nu\), by the  Portmanteau theorem
	\cite[Theorem~2.1]{Billingsley1999}, for every \(r>0\) and all sufficiently
	large \(j\),
	\[
	\operatorname{supp}\nu_j\cap B_r(x)\neq\varnothing.
	\]
	Hence we may choose \(x_j\in\operatorname{supp}\nu_j\) such that
	\(x_j\to x\). As all spaces are already realized isometrically in the
	common ambient space \(\mathbb R^N\), the convergence
	\(\nu_j\rightharpoonup\nu\), together with \(x_j\to x\), yields
	\[
	(X_j,|\cdot|,\nu_j,x_j)
	\longrightarrow
	(X,|\cdot|,\nu,x)
	\]
	in the pointed measured Gromov sense. By \cite[Definition~6.1 and Section~6.1, after (6.6)]
	{GigliMondinoSavare2015}, the uniform bound $\sup_j\int |f_j|^2\,d\nu_j<\infty$
	allows us, after passing to a subsequence, to find \(f\in L^2(\nu)\)
	such that \(f_j\) converges to \(f\) weakly in the varying \(L^2\)
	sense.
	By \cite[Theorem~6.3(ii)]{GigliMondinoSavare2015},
\[
\operatorname{Ch}_{\nu}(f)
\leq
\liminf_{j\to\infty}\operatorname{Ch}_{\nu_j}(f_j).
\]
Using $\operatorname{Ch}_{\nu_j}
=
\frac12\mathcal E_{\nu_j},\,
\operatorname{Ch}_{\nu}
=
\frac12\mathcal E_{\nu},$
we obtain \(f\in H^1(\nu)\) and
\[
\mathcal E_\nu(f)
\leq
\liminf_{j\to\infty}\mathcal E_{\nu_j}(f_j).
\]
	
	The additional condition \cite[(6.11)]{GigliMondinoSavare2015} is
	automatic here. Indeed, since all supports are contained in a common
	compact set, for all sufficiently large \(R\),
	\[
	\int_{\mathbb R^N\setminus B_R(x_j)}
	|f_j|^2\,d\nu_j=0.
	\]
	Hence \cite[Theorem~6.3(ii)]{GigliMondinoSavare2015} upgrades the weak
	convergence to strong convergence in the varying \(L^2\) spaces. In
	particular,
	\[
	\int |f_j|^2\,d\nu_j
	\longrightarrow
	\int |f|^2\,d\nu.
	\]
	Moreover, for every \(\phi\in C_b(\mathbb R^N)\), after multiplying
	\(\phi\) by a fixed compactly supported cutoff if necessary,
	\cite[(6.6)]{GigliMondinoSavare2015} applied with
	\(\zeta(x,r):=\phi(x)r^2\) gives
	\[
	\int \phi f_j^2\,d\nu_j
	\longrightarrow
	\int \phi f^2\,d\nu.
	\]
	Therefore, $f_j^2\nu_j\rightharpoonup f^2\nu,$ we complete the proof.
\end{proof}

We shall also need the following second-order estimate for the
self-adjoint operator associated with a log-concave measure. It gives
\(H^1\)-regularity of the gradient and controls the Hessian directly by
the \(L^2(\nu)\)-norm of the generator, with constant \(1\). 

\begin{lemma}\label{lemma:ordinary2}
	Let \(U\subset\mathbb R^d\) be a bounded convex open set and let
	\(d\nu=\rho\,dx\), where \(\rho>0\) in \(U\) and \(\rho\) is
	log-concave. Consider the closed quadratic form
	\[
	\mathcal E_\nu(h,\phi)
	:=
	\int_U \nabla h\cdot\nabla\phi\,d\nu,
	\qquad
	h,\phi\in H^1(\nu),
	\]
	and let \(L_\nu\) be the associated nonnegative self-adjoint operator in
	\(L^2(\nu)\). Its domain is
	\[
	D(L_\nu)
	:=
	\left\{
	h\in H^1(\nu):
	\exists\, f\in L^2(\nu)\ \text{such that}\
	\mathcal E_\nu(h,\phi)
	=
	\int_U f\phi\,d\nu
	\ \text{for every }\phi\in H^1(\nu)
	\right\},
	\]
	and for \(h\in D(L_\nu)\) we define \(L_\nu h:=f\). Equivalently,
	\(h\) satisfies
	\[
	-\operatorname{div}(\rho\nabla h)
	=
	\rho f
	\]
	in the weak sense. Then, for every \(h\in D(L_\nu)\),
	\[\nabla h\in H^1(\nu;\mathbb R^d),
	\qquad
	\int_U |D^2h|^2\,d\nu
	\leq
	\int_U |f|^2\,d\nu.\]
\end{lemma}

\begin{proof}
	By the definition of \(D(L_\nu)\), we have \(f\in L^2(\nu)\) and
	\begin{equation}\label{eq:weak-Lnu}
		\int_U \nabla h\cdot\nabla\phi\,d\nu
		=
		\int_U f\phi\,d\nu
		\qquad
		\text{for every }\phi\in H^1(\nu).
	\end{equation}
	Taking \(\phi=1\) gives $\int_U f\,d\nu=0.$
	Since \(d\nu=\rho\,dx\), \eqref{eq:weak-Lnu} is the weak formulation
	\[
	\int_U \rho\,\nabla h\cdot\nabla\phi\,dx
	=
	\int_U \rho f\phi\,dx,
	\]
	to which \cite[Theorem~1.1]{AntoniniCiraoloPagliarin2025} applies.
	
	Assume first that \(d\geq2\). We apply that theorem with exponent
	\(p=2\), domain \(U\), weight \(\rho\), and datum \(f\). Its hypotheses
	are satisfied: \(U\) is bounded and convex, \(\rho\) is positive and
	log-concave in \(U\), and $f\in L^2(U;\rho),\,\int_U f\rho\,dx=0.$
	\cite[Theorem~1.1]{AntoniniCiraoloPagliarin2025} therefore gives $\nabla h\in H^1(\nu;\mathbb R^d).$
	Moreover, the constant \(\overline c(p)\) in
	\cite[(5.35)]{AntoniniCiraoloPagliarin2025} satisfies
	\(\overline c(2)=1\), and \cite[(5.38)]{AntoniniCiraoloPagliarin2025}
	gives
	\[
	\int_U |D^2h|^2\,d\nu
	\leq
	\int_U |L_\nu h|^2\,d\nu.
	\]
	
	It remains to treat \(d=1\), since
	\cite[Theorem~1.1]{AntoniniCiraoloPagliarin2025} is stated for
	dimensions \(d\geq2\). Identify \(U\) with an interval and define
	\[
	\widetilde U:=U\times(0,1),\qquad
	\widetilde\rho(x,t):=\rho(x),\qquad
	\widetilde h(x,t):=h(x),\qquad
	\widetilde f(x,t):=f(x).
	\]
	Then \(\widetilde U\subset\mathbb R^2\) is bounded and convex,
	\(\widetilde\rho\) is positive and log-concave, and $\int_{\widetilde U}\widetilde f\,\widetilde\rho\,dx\,dt=0.$
	For \(\Phi\in H^1(\widetilde U;\widetilde\rho)\), set $\overline\Phi(x):=\int_0^1\Phi(x,t)\,dt.$
	By Fubini's theorem and Jensen's inequality,
	\(\overline\Phi\in H^1(U;\nu)\) and $\overline\Phi'(x)=\int_0^1\partial_x\Phi(x,t)\,dt$
	in the weak sense. Hence \eqref{eq:weak-Lnu} gives
	\begin{align*}
		\int_{\widetilde U}
		\widetilde\rho\,
		\nabla\widetilde h\cdot\nabla\Phi\,dx\,dt
		&=
		\int_U\rho h'\,\overline\Phi'\,dx=
		\int_U\rho f\,\overline\Phi\,dx=
		\int_{\widetilde U}
		\widetilde\rho\,\widetilde f\,\Phi\,dx\,dt.
	\end{align*}
	Thus \(\widetilde h\) satisfies the two-dimensional weak problem, and
	the preceding argument yields
	\[
	\nabla\widetilde h\in
	H^1(\widetilde U;\widetilde\rho;\mathbb R^2),
	\qquad
	\int_{\widetilde U}|D^2\widetilde h|^2\widetilde\rho\,dx\,dt
	\leq
	\int_{\widetilde U}\widetilde f^2\widetilde\rho\,dx\,dt.
	\]
	Since
	\[
	D^2\widetilde h
	=
	\begin{pmatrix}
		h''&0\\
		0&0
	\end{pmatrix},
	\]
	this reduces to
	\[
	h'\in H^1(\nu),
	\qquad
	\int_U|h''|^2\,d\nu
	\leq
	\int_U|f|^2\,d\nu
	=
	\int_U|L_\nu h|^2\,d\nu.
	\]
	Hence we complete the proof.
\end{proof}

The next estimate gives a dimension-free Poincar\'e bound for an arbitrary log-concave measure. It will provide the basic coercivity needed when the limiting measure is supported on a proper affine subspace.

\begin{lemma}\label{lemma:ordinary3}
	Let \(\nu\) be a nonzero finite log-concave measure with $D_\nu:=\operatorname{diam}(\operatorname{supp}\nu)<\infty.$ Then
	\begin{equation}\label{eq:rough-P}
		\operatorname{Var}_\nu(f) \leq \frac{D_\nu^2}{\pi^2}\, \mathcal E_\nu(f) \qquad \text{for every }f\in H^1(\nu).
	\end{equation}
\end{lemma}

\begin{proof}
	The result follows from \cite[Theorem~1.1]{ferone-nitsch-trombetti}, applied with exponent \(2\) on the affine hull of \(\operatorname{supp}\nu\), together with the density of smooth functions in \(H^1(\nu)\). If \(\nu\) is a point mass, then \(\operatorname{Var}_\nu(f)=0\), so the conclusion is immediate.
\end{proof}

\subsection{Compactness of the degenerate weighted forms}

We next develop the compactness estimates needed to pass from the degenerate weighted forms to the ordinary Dirichlet form of a weak limit measure. These estimates will later identify the Poincar\'e constant carried by a collapsing sequence of ground states.

We next establish an inverse-weight estimate for the normalized
ground-state weight \(a\), together with uniform moment bounds
and its quantitative convergence to \(1\) as \(p\downarrow2\).

\begin{proposition}\label{prop:inverse}
	Let \(2<p<5/2\), let \(\mu\) be a probability measure, and set
	\[
	\varepsilon:=p-2,
	\qquad
	Z:=r\lambda^{-1/p},
	\qquad
	a:=Z^\varepsilon,
	\]
	and define $k_p:=7^{2\varepsilon},\,
	d_{p,\Omega}:=\frac{4\varepsilon}{7}\lambda^{-2/p}.$
	Then, for every \(f\in H^1(\mu)\),
	\begin{equation}\label{eq:inverse}
		\int a^{-1}f^2\,d\mu
		\leq
		k_p
		\left(
		\int f^2\,d\mu
		+
		d_{p,\Omega}\int |\nabla f|^2\,d\mu
		\right).
	\end{equation}
	Moreover,
	\begin{equation}\label{eq:moments}
		\int Z^p\,d\mu=1,
		\qquad
		\int Z^{-1/2}\,d\mu\leq7,
		\qquad
		\int a^2\,d\mu\leq1.
	\end{equation}
	Finally, for every \(1\leq q<\infty\), there exist
	\(\varepsilon_q,K_q>0\), depending only on \(q\), such that whenever
	\(0<p-2\leq\varepsilon_q\),
	\begin{equation}\label{eq:Lq}
		\|a-1\|_{L^q(\mu)}
		+
		\|a^{-1}-1\|_{L^q(\mu)}
		\leq
		K_q(p-2).
	\end{equation}
	The constants \(\varepsilon_q\) and \(K_q\) are independent of
	\(N\) and \(\Omega\).
\end{proposition}

\begin{proof}
	We first assume that \(\Omega\) is smooth and convex. We first prove the estimates for \(\lambda=1\). Set
	\[
	v:=-\log u,
	\qquad
	r:=|\nabla v|,
	\qquad
	X:=r^{p-2}\nabla v,
	\]
	and define $\operatorname{div}_{\mu}Y
	:=
	u^{-p}\operatorname{div}(u^pY).$ By Lemma \ref{lem:log-ground-state-equations},
	we have
	\[\operatorname{div}X
	=
	1+(p-1)r^p,
	\qquad
	\operatorname{div}_{\mu}X
	=
	1-r^p.\]
	
	By Theorem~\ref{thm:first-eigenfunction-log-concavity}, \(v\) is
	convex. On \(\{r>0\}\), the equation is uniformly elliptic locally,
	and hence \(v\) is smooth there. Set
	\[
	e:=\frac{\nabla v}{r},
	\qquad
	H:=D^2v[e,e].
	\]
	Convexity gives $0\leq H\leq\Delta v.$
	Moreover, $\operatorname{div}X
	=
	r^{p-2}\bigl(\Delta v+(p-2)H\bigr),$
	and therefore
	\begin{equation}\label{eq:Hbound}
		(p-1)r^{p-2}H
		\leq
		1+(p-1)r^p.
	\end{equation}
	
	We also claim that $|\{r=0\}|=0.$
	Indeed, since \(v\) is convex,
	\(\{r=0\}\) is precisely the set of minimizers of \(v\), hence is
	convex. If it had positive \(N\)-dimensional measure, it would have
	nonempty interior. Thus \(v\), and hence \(u\), would be constant on
	a nonempty open subset of \(\Omega\), which is incompatible with
	\(-\Delta_pu=u^{p-1}>0\).
	
	\smallskip
	
	Define $F:=(r^{-1/2}-1)_+X.$
	On \(\{0<r<1\}\),
	\[
	|F|
	=
	(r^{-1/2}-1)r^{p-1}
	\leq
	r^{p-3/2},
	\]
	which tends to zero as \(r\downarrow0\). Thus \(F\) extends
	continuously by \(F=0\) on \(\{r=0\}\). Since $\nabla r\cdot\nabla v=rH,$
	we obtain on \(\{0<r<1\}\)
	\begin{align}
		\operatorname{div}_{\mu}F
		&=
		(r^{-1/2}-1)(1-r^p)
		-\frac12 r^{p-5/2}H \notag\\
		&\geq
		\left(1-\frac{1}{2(p-1)}\right)r^{-1/2}
		-1-\frac32r^{p-1/2}+r^p \geq
		\frac12r^{-1/2}-\frac52.
		\label{eq:Fdiv}
	\end{align}
	Here we used \eqref{eq:Hbound} in the first inequality and
	\(0<r<1\), \(p>2\), in the second inequality. By the Hopf boundary lemma and the boundary \(C^1\)-regularity of
	\(u\) in Proposition \ref{prop:first-eigenfunction-boundary-regularity},
	\[
	\frac{|\nabla u|}{u}\longrightarrow\infty
	\qquad\text{as }x\to\partial\Omega.
	\]
	Hence \(r>1\) in a boundary collar of \(\partial\Omega\), and therefore
	\(F=0\) there.
	
	\smallskip
	
	Let \(f\) be Lipschitz and let \(\delta\in(0,1)\) be a regular value
of \(r\), so that \(\nabla r\neq0\) on \(\{r=\delta\}\). Set $E_\delta:=\{r>\delta\}.$
	Since \(F=0\) on \(\{r\geq1\}\) and in a boundary collar of
	\(\partial\Omega\), integration by parts on \(E_\delta\) gives
	\[
	\int_{E_\delta}
	f^2\operatorname{div}_{\mu}F\,d\mu
	=
	\int_{\{r=\delta\}}
	f^2F\cdot\nu\,u^p\,d\mathcal H^{N-1}
	-
	2\int_{E_\delta}fF\cdot\nabla f\,d\mu,
	\]
	where $\nu=-\frac{\nabla r}{|\nabla r|}$
	is the outward unit normal to \(E_\delta\) along \(\{r=\delta\}\).
	Since
	\[
	F\cdot\nu
	=
	-(\delta^{-1/2}-1)\delta^{p-2}
	\frac{\nabla v\cdot\nabla r}{|\nabla r|}
	=
	-(\delta^{-1/2}-1)\delta^{p-1}
	\frac{H}{|\nabla r|}
	\leq0,
	\]
	and \(F=0\) on \(\{r\geq1\}\), \eqref{eq:Fdiv} yields
	\[
	\frac12
	\int_{\{\delta<r<1\}}
	r^{-1/2}f^2\,d\mu
	\leq
	\frac52
	\int_{\{\delta<r<1\}}
	f^2\,d\mu
	+
	2\int_{\{\delta<r<1\}}
	|f||\nabla f|\,d\mu.
	\]
	By Sard's theorem, we may choose regular values
\(\delta_k\downarrow0\). Letting \(k\to\infty\) and using
\(\mu(\{r=0\})=0\), we obtain
	\[
	\frac12
	\int_{\{r<1\}}
	r^{-1/2}f^2\,d\mu
	\leq
	\frac52
	\int_{\{r<1\}}
	f^2\,d\mu
	+
	2\int_{\{r<1\}}
	|f||\nabla f|\,d\mu.
	\]
	Since \(r^{-1/2}\leq1\) on \(\{r\geq1\}\),
	\[
	\int r^{-1/2}f^2\,d\mu
	\leq
	5\int f^2\,d\mu
	+
	4\int |f||\nabla f|\,d\mu.
	\]
	Using $4|f||\nabla f|
	\leq
	2f^2+2|\nabla f|^2,$
	we conclude that
	\begin{equation}\label{eq:half-negative}
		\int r^{-1/2}f^2\,d\mu
		\leq
		7\int f^2\,d\mu
		+
		2\int|\nabla f|^2\,d\mu.
	\end{equation}
	Let $M:=\int f^2\,d\mu,\,
	E:=\int|\nabla f|^2\,d\mu.$
	Since \(0<2\varepsilon<1\), H\"older's inequality gives, for \(M>0\),
	\begin{align*}
		\int r^{-\varepsilon}f^2\,d\mu
		&\leq
		\left(
		\int r^{-1/2}f^2\,d\mu
		\right)^{2\varepsilon}
		M^{1-2\varepsilon}\leq
		7^{2\varepsilon}M
		\left(
		1+\frac{2E}{7M}
		\right)^{2\varepsilon}\leq
		7^{2\varepsilon}
		\left(
		M+\frac{4\varepsilon}{7}E
		\right),
	\end{align*}
	where in the last step we used $(1+t)^\theta\leq1+\theta t,\,
	t\geq0,\,0<\theta<1.$
	The case \(M=0\) is immediate.
	
	\smallskip
	
	We now restore the general value of \(\lambda\). Let
	\[
	s:=\lambda^{1/p},
	\qquad
	\widetilde\Omega:=s\Omega,
	\qquad
	\widetilde u(y):=s^{-N/p}u(y/s).
	\]
	Then \(\widetilde u\) is \(L^p\)-normalized and has first eigenvalue
	one on \(\widetilde\Omega\). Moreover, $|\nabla_y\log\widetilde u(y)|
	=
	s^{-1}r(y/s),$
	so the corresponding quantity \(\widetilde r\) equals \(Z\), while
	for \(\widetilde f(y):=f(y/s)\),
	\[
	\int_{\widetilde\Omega}|\nabla\widetilde f|^2\,d\widetilde\mu
	=
	\lambda^{-2/p}
	\int_\Omega|\nabla f|^2\,d\mu.
	\]
	Consequently \eqref{eq:half-negative} becomes
	\begin{equation}\label{eq:half-negative-scaled}
		\int Z^{-1/2}f^2\,d\mu
		\leq
		7\int f^2\,d\mu
		+
		2\lambda^{-2/p}\int|\nabla f|^2\,d\mu,
	\end{equation}
	and the preceding interpolation estimate gives
	\[
	\int Z^{-\varepsilon}f^2\,d\mu
	\leq
	7^{2\varepsilon}
	\left(
	\int f^2\,d\mu
	+
	\frac{4\varepsilon}{7}\lambda^{-2/p}
	\int|\nabla f|^2\,d\mu
	\right).
	\]
	Since \(a^{-1}=Z^{-\varepsilon}\), this is
	\eqref{eq:inverse} for smooth convex domains and Lipschitz \(f\). Taking \(f=1\) in \eqref{eq:half-negative-scaled} gives $\int Z^{-1/2}\,d\mu\leq7.$
	Since \(\int_\Omega u^p\,dx=1\),
	\[
	\int Z^p\,d\mu
	=
	\lambda^{-1}
	\int r^pu^p\,dx
	=
	\lambda^{-1}
	\int|\nabla u|^p\,dx
	=
	1.
	\]
	Moreover, \(2\varepsilon<p\), and hence, since \(\mu\) is a
	probability measure,
	\[
	\int a^2\,d\mu
	=
	\int Z^{2\varepsilon}\,d\mu
	\leq
	\left(
	\int Z^p\,d\mu
	\right)^{2\varepsilon/p}
	=
	1.
	\]
	This proves \eqref{eq:moments} in the smooth case.
	
	\smallskip
	
	We next remove the boundary regularity assumption. Let $\Omega_m\Subset\Omega,\,
	\Omega_m\uparrow\Omega,$
	be a smooth convex exhaustion, and let \(u_m\) be the positive
	\(L^p\)-normalized first eigenfunction on \(\Omega_m\), extended by
	zero outside \(\Omega_m\). Write $\lambda_m:=\lambda_{1,p}(\Omega_m).$
	By domain monotonicity and approximation of \(W^{1,p}_0(\Omega)\)
	by compactly supported functions, $\lambda_m\rightarrow\lambda.$
	The sequence \(u_m\) is bounded in \(W^{1,p}_0(\Omega)\). By the compact embedding \(W^{1,p}_0(\Omega)\hookrightarrow L^p(\Omega)\),
every weakly convergent subsequence converges strongly in \(L^p(\Omega)\).
Hence every weak limit \(w\) satisfies
\[
\|w\|_{L^p(\Omega)}=1,
\qquad
\int_\Omega|\nabla w|^p\,dx
\leq
\liminf_{m\to\infty}\lambda_m
=
\lambda,
\]
and therefore \(w\) is a normalized first eigenfunction on \(\Omega\).
	By simplicity, \(w=u\). Moreover,
	\[
	\|\nabla u_m\|_{L^p}^p
	=
	\lambda_m
	\longrightarrow
	\lambda
	=
	\|\nabla u\|_{L^p}^p.
	\]
	The uniform convexity of \(L^p\) therefore yields
	\[
	u_m\longrightarrow u
	\qquad\text{strongly in }W^{1,p}(\mathbb R^N).
	\]
	After passing to a subsequence, $u_m\to u,\,
	\nabla u_m\to\nabla u$
	almost everywhere.
	
	\smallskip
	
	Let $Z_m
	:=
	\lambda_m^{-1/p}
	\frac{|\nabla u_m|}{u_m}$
	on \(\Omega_m\). 
For every ambient Lipschitz function \(f\), writing the integrals with
respect to Lebesgue measure and using Fatou's lemma together with the
smooth-domain estimate, we obtain
\begin{align*}
\int Z^{-\varepsilon}f^2\,d\mu
&\leq
\liminf_{m\to\infty}
\int Z_m^{-\varepsilon}f^2\,d\mu_m\leq
\lim_{m\to\infty}
7^{2\varepsilon}
\left(
\int f^2\,d\mu_m
+
\frac{4\varepsilon}{7}\lambda_m^{-2/p}
\int|\nabla f|^2\,d\mu_m
\right)\\
&=
7^{2\varepsilon}
\left(
\int f^2\,d\mu
+
\frac{4\varepsilon}{7}\lambda^{-2/p}
\int|\nabla f|^2\,d\mu
\right).
\end{align*}
	The same argument applied to \eqref{eq:half-negative-scaled} gives $\int Z^{-1/2}\,d\mu\leq7.$
	The identities
	\[
	\int Z^p\,d\mu=1,
	\qquad
	\int a^2\,d\mu\leq1
	\]
	follow directly from the eigenvalue equation and H\"older's
	inequality, exactly as above.
	
	\smallskip
	
	By Theorem~\ref{thm:first-eigenfunction-log-concavity}, \(u\) is
	log-concave, so \(u^p\mathbf 1_\Omega\) is a log-concave density on
	\(\mathbb R^N\). Hence, by
	\cite[Proposition~27]{BartheKlartag2020},
	\(C_c^\infty(\mathbb R^N)\) is dense in \(H^1(\mu)\).
	For \(f\in H^1(\mu)\), choose \(f_j\in C_c^\infty(\mathbb R^N)\)
	such that $f_j\rightarrow f$ in $H^1(\mu).$
	After passing to a subsequence, \(f_j\to f\) \(\mu\)-almost everywhere.
Fatou's lemma, the estimate for \(f_j\), and the convergence
\(f_j\to f\) in \(H^1(\mu)\) yield
\[
\int a^{-1}f^2\,d\mu
\leq
\liminf_{j\to\infty}
\int a^{-1}f_j^2\,d\mu
\leq
k_p
\left(
\int f^2\,d\mu
+
d_{p,\Omega}\int|\nabla f|^2\,d\mu
\right).
\]
Thus \eqref{eq:inverse} extends to every \(f\in H^1(\mu)\).
	
	It remains to prove \eqref{eq:Lq}. By \eqref{eq:moments} and
	H\"older's inequality, $\int Z^{1/2}\,d\mu\le 1$
	Therefore
	\[
	\int e^{|\log Z|/2}\,d\mu
	\leq
	\int\bigl(Z^{1/2}+Z^{-1/2}\bigr)\,d\mu
	\leq8.
	\]
	For \(t\in\mathbb R\), $|e^t-1|
	\leq
	|t|e^{|t|}.$
	Thus, with \(s:=|\log Z|\), $|Z^{\pm\varepsilon}-1|^q
	\leq
	\varepsilon^q s^q e^{q\varepsilon s}.$
	If $0<\varepsilon\leq\varepsilon_q:=\frac{1}{4q},$
	then
	\[
	|Z^{\pm\varepsilon}-1|^q
	\leq
	\varepsilon^q s^q e^{s/4}.
	\]
	Since $s^qe^{s/4}
	\leq
	C_qe^{s/2},\,s\geq0,$
	for a constant \(C_q\) depending only on \(q\), we obtain
	\[
	\|Z^{\pm\varepsilon}-1\|_{L^q(\mu)}
	\leq
	(8C_q)^{1/q}\varepsilon.
	\]
	Since $a=Z^\varepsilon,\,
	a^{-1}=Z^{-\varepsilon},$
	after enlarging the constant \(K_q\) we conclude that
	\[
	\|a-1\|_{L^q(\mu)}
	+
	\|a^{-1}-1\|_{L^q(\mu)}
	\leq
	K_q\varepsilon
	=
	K_q(p-2).
	\]
	The constants \(\varepsilon_q\) and \(K_q\) depend only on \(q\),
	and are independent of \(N\) and \(\Omega\).
\end{proof}

The next lemma converts a uniform bound for the degenerate energies
\(\mathcal Q_{p_j,\Omega_j}\) into compactness for the ordinary
Dirichlet forms. The \(L^q\) convergence \(a_j\to1\) alone is not
sufficient, since \(|\nabla f_j|^2\) may concentrate on sets where
\(a_j\) is small.

\begin{lemma}
	\label{lem:deg-compact}
	Let \(p_j\downarrow2\), let \(\Omega_j\in\mathcal C_N\) be contained
	in a fixed bounded set, and let \(u_j\) be the positive first
	eigenfunction on \(\Omega_j\), normalized by $\|u_j\|_{L^{p_j}(\Omega_j)}=1.$
	Set $\mu_j:=u_j^{p_j}\,dx,$
	assume \(\mu_j\rightharpoonup\mu\), and let \(a_j\) be defined by
	\eqref{eq:defs-weight}. If
	\(f_j\in C^{0,1}(\overline{\Omega_j})\) satisfies
	\[
	\sup_j
	\left(
	\int f_j^2\,d\mu_j
	+
	\int a_j|\nabla f_j|^2\,d\mu_j
	\right)
	<\infty,
	\]
	then, up to a subsequence, \(f_j\) converges strongly to some
	\(f\in H^1(\mu)\) in the varying \(L^2\) spaces, and
	\[
	\mathcal E_\mu(f)
	\leq
	\liminf_{j\to\infty}
	\int a_j|\nabla f_j|^2\,d\mu_j.
	\]
\end{lemma}

\begin{proof}
	Set
	\[
	\lambda_j:=\lambda_{1,p_j}(\Omega_j),
	\qquad
	L_j:=L_{\mu_j},
	\qquad
	\mathcal E_j:=\mathcal E_{\mu_j},
	\qquad
	Q_j:=\mathcal Q_{p_j,\Omega_j},
	\]
	and $k_j:=k_{p_j},\,
	d_j:=d_{p_j,\Omega_j},$ where the notation is as in Lemma~\ref{lemma:ordinary2} and
	Proposition~\ref{prop:inverse}.
	After discarding finitely many indices, we may assume
	\(2<p_j<5/2\), so that Proposition~\ref{prop:inverse} applies. 
	
	Suppose the domains \(\Omega_j\) are contained in a fixed ball \(B_R\),
	domain monotonicity gives
	\[
	\lambda_j
	=
	\lambda_{1,p_j}(\Omega_j)
	\geq
	\lambda_{1,p_j}(B_R).
	\]
	Since \(p_j\to2\) and
	\(\lambda_{1,p}(B_R)\to\lambda_{1,2}(B_R)>0\) as \(p\to2\), we have $\inf_j\lambda_j^{1/p_j}>0$
	after discarding finitely many indices. Consequently, $k_j\to1$ and $d_j\to 0.$ Let \(h\in D(L_j)\). By Lemma~\ref{lemma:ordinary2},
	\[
	\nabla h\in H^1(\mu_j;\mathbb R^N),
	\qquad
	\int |D^2h|^2\,d\mu_j
	\leq
	\|L_jh\|_{L^2(\mu_j)}^2.
	\]
	Applying Proposition~\ref{prop:inverse} to each component of \(\nabla h\)
	and summing, we obtain
	\begin{align}
		\int a_j^{-1}|\nabla h|^2\,d\mu_j
		&\leq
		k_j
		\left(
		\int|\nabla h|^2\,d\mu_j
		+
		d_j\int|D^2h|^2\,d\mu_j
		\right)\notag\\
		&\leq
		k_j
		\left(
		\mathcal E_j(h)
		+
		d_j\|L_jh\|_{L^2(\mu_j)}^2
		\right).
		\label{eq:inverse-grad}
	\end{align}
	
	For \(f\in H^1(\mu_j)\), the definition of \(L_j\) gives
	\[
	\langle f,L_jh\rangle_{L^2(\mu_j)}
	=
	\int\nabla f\cdot\nabla h\,d\mu_j.
	\]
	Since \(a_j>0\) \(\mu_j\)-almost everywhere, $2\nabla f\cdot\nabla h
	\leq
	a_j|\nabla f|^2
	+
	a_j^{-1}|\nabla h|^2.$
	Thus \eqref{eq:inverse-grad} yields
	\begin{equation}\label{eq:dual-pre}
		2\langle f,L_jh\rangle
		-
		k_j\left(
		\mathcal E_j(h)
		+
		d_j\|L_jh\|_2^2
		\right)
		\leq
		Q_j(f).
	\end{equation}
	
	We now take the supremum over \(h\in D(L_j)\). Since \(L_j\) is
	nonnegative and self-adjoint, the spectral theorem gives
	\begin{equation}\label{eq:spectral-dual}
		\sup_{h\in D(L_j)}
		\left\{
		2\langle f,L_jh\rangle
		-
		k_j\left(
		\mathcal E_j(h)
		+
		d_j\|L_jh\|_2^2
		\right)
		\right\}
		=
		\frac1{k_j}
		\left\langle
		f,L_j(I+d_jL_j)^{-1}f
		\right\rangle.
	\end{equation}
	Indeed, on the spectral variable \(\xi\geq0\), the expression to be
	maximized is $2\xi\,\widehat f\,\widehat h
	-
	k_j\xi(1+d_j\xi)|\widehat h|^2,$
	whose supremum is
	\[
	\frac{\xi|\widehat f|^2}{k_j(1+d_j\xi)}.
	\]
	Equivalently, the supremum in \eqref{eq:spectral-dual} is attained at
	\[
	h=\frac1{k_j}(I+d_jL_j)^{-1}f.
	\]
	Combining \eqref{eq:dual-pre} and \eqref{eq:spectral-dual}, we obtain
	\begin{equation}\label{eq:resolvent-bound}
		\left\langle
		f,L_j(I+d_jL_j)^{-1}f
		\right\rangle
		\leq
		k_jQ_j(f).
	\end{equation}
	
	Set $z_j:=(I+d_jL_j)^{-1/2}f_j.$
	By the spectral theorem and \eqref{eq:resolvent-bound},
	\begin{equation}\label{eq:smoothing}
		\mathcal E_j(z_j)
		=
		\left\langle
		f_j,L_j(I+d_jL_j)^{-1}f_j
		\right\rangle
		\leq
		k_jQ_j(f_j),
	\end{equation}
	while $\|z_j\|_{L^2(\mu_j)}
	\leq
	\|f_j\|_{L^2(\mu_j)}.$
	Moreover, since
	\[
	\left(1-\frac1{\sqrt{1+x}}\right)^2
	\leq
	\frac{x}{1+x},
	\qquad x\geq0,
	\]
	another application of the spectral theorem gives
	\begin{equation}\label{eq:resolvent-close}
		\|z_j-f_j\|_{L^2(\mu_j)}^2\leq
		d_j
		\left\langle
		f_j,L_j(I+d_jL_j)^{-1}f_j
		\right\rangle\leq
		d_jk_jQ_j(f_j)
		\longrightarrow0,
	\end{equation}
	because \(d_j\to0\), \(k_j\to1\), and \(\sup_jQ_j(f_j)<\infty\).
	
	\smallskip
	
	Choose a subsequence, not relabeled, such that $Q_j(f_j)
	\rightarrow
	\liminf_{\ell\to\infty}Q_\ell(f_\ell).$
	By \eqref{eq:smoothing} and the assumed \(L^2\)-bound, the sequence
	\(\{z_j\}\) satisfies the hypotheses of
	Lemma~\ref{lemma:ordinary1}. Hence, after passing to a further
	subsequence, $z_j\rightarrow f$
	strongly in the varying \(L^2\) spaces for some \(f\in H^1(\mu)\), and $\mathcal E_\mu(f)
	\leq
	\liminf_{j\to\infty}\mathcal E_j(z_j).$
	Using \eqref{eq:smoothing},
	\[
	\mathcal E_\mu(f)
	\leq
	\liminf_{j\to\infty}k_jQ_j(f_j)
	=
	\liminf_{j\to\infty}Q_j(f_j).
	\]
	Finally, \eqref{eq:resolvent-close} implies that
	\(f_j\) has the same strong varying-\(L^2\) limit \(f\). 
\end{proof}

\begin{remark}
	The point of the resolvent regularization is that
	\(a_j\to1\) in \(L^q(\mu_j)\) does not rule out concentration of
	\(|\nabla f_j|^2\) on sets where \(a_j\) is small. The inverse estimate
	from Proposition~\ref{prop:inverse} gives instead the operator bound
	\[
	\left\langle
	f_j,L_j(I+d_jL_j)^{-1}f_j
	\right\rangle
	\leq
	k_jQ_j(f_j),
	\qquad
	k_j\to1,\quad d_j\to0.
	\]
	Thus $z_j=(I+d_jL_j)^{-1/2}f_j$
	has uniformly bounded ordinary Dirichlet energy, while $\|z_j-f_j\|_{L^2(\mu_j)}\to0.$
	This allows the compactness and lower semicontinuity of the ordinary
	Dirichlet forms to be transferred to the degenerate forms by Lemma \ref{lemma:ordinary1}.
\end{remark}

The preceding compactness result allows us to identify the normalized weighted spectral constant along any sequence with \(p_j\downarrow2\). The limit is exactly the ordinary Poincar\'e constant of the weak ground-state measure.

\begin{lemma}
	\label{lem:constant-limit}
	Let \(p_j\downarrow2\), and let \(\Omega_j\in\mathcal C_N\) be
	contained in a fixed bounded set. Let \(u_j\) be the positive first
	eigenfunction on \(\Omega_j\), normalized by $\|u_j\|_{L^{p_j}(\Omega_j)}=1,$
	and set
	\[
	\mu_j:=u_j^{p_j}\,dx,
	\qquad
	D_j:=\operatorname{diam}(\Omega_j).
	\]
	Assume $D_j\rightarrow D>0,\,
	\mu_j\rightharpoonup\mu.$
	Then
	\[
	\frac{
		\lambda_{1,p_j}(\Omega_j)^{\frac{p_j-2}{p_j}}
	}{
		D_j^2\Lambda_{p_j}(\Omega_j)
	}
	\longrightarrow
	\frac{P(\mu)}{D^2}.
	\]
\end{lemma}

\begin{proof}
	Let $S:=\operatorname{aff}(\operatorname{supp}\mu),$
	and let \(\pi_S:\mathbb R^N\to S\) be the orthogonal projection. For a
	smooth function \(g\) on \(S\), choose
	\(\chi\in C_c^\infty(\mathbb R^N)\) such that \(\chi\equiv1\) on a
	neighborhood of the common bounded set containing all \(\Omega_j\), and
	extend \(g\) to \(\mathbb R^N\) by
	\[
	\widetilde g(x):=\chi(x)g(\pi_Sx).
	\]
	On every \(\Omega_j\), $\widetilde g=g\circ\pi_S,$
	and \(|\nabla\widetilde g|\) is uniformly bounded. For simplicity, we continue to denote this extension by \(g\).  By \eqref{eq:Lq} with \(q=1\), $\|a_j-1\|_{L^1(\mu_j)}\rightarrow0.$
	Hence
	\begin{align*}
		\left|
		\mathcal Q_{p_j,\Omega_j}(g)-\int |\nabla g|^2\,d\mu_j
		\right|
		&=
		\left|
		\int (a_j-1)|\nabla g|^2\,d\mu_j
		\right|\leq
		\|\nabla g\|_{L^\infty}^2
		\|a_j-1\|_{L^1(\mu_j)}
		\longrightarrow0.
	\end{align*}
	Since \(|\nabla g|^2\) is bounded and continuous and
	\(\mu_j\rightharpoonup\mu\),
	\[
	\int |\nabla g|^2\,d\mu_j
	\longrightarrow
	\int |\nabla g|^2\,d\mu
	=
	\int_S|\nabla_Sg|^2\,d\mu.
	\]
	Therefore $\mathcal Q_{p_j,\Omega_j}(g)
	\rightarrow
	\mathcal E_\mu(g)
	=
	\int_S|\nabla_Sg|^2\,d\mu,$
	weak convergence also gives $\operatorname{Var}_{\mu_j}(g)
	\rightarrow
	\operatorname{Var}_\mu(g).$
	For every fixed smooth \(g\) on \(S\), \eqref{eq:C1} gives
	\[
	\liminf_{j\to\infty}
	\frac{\lambda_{1,p_j}(\Omega_j)^{\frac{p_j-2}{p_j}}}
	{D_j^2\Lambda_{p_j}(\Omega_j)}
	\geq
	\frac{\operatorname{Var}_\mu(g)}
	{D^2\mathcal E_\mu(g)}.
	\]
	Taking the supremum over smooth \(g\) and using smooth density in
	\(H^1(\mu)\), see \cite[Proposition~27]{BartheKlartag2020}
	\[\liminf_{j\to\infty}
	\frac{\lambda_{1,p_j}(\Omega_j)^{\frac{p_j-2}{p_j}}}
	{D_j^2\Lambda_{p_j}(\Omega_j)}
	\geq
	\frac{P(\mu)}{D^2}.\]
	
	For the reverse inequality, suppose that along a subsequence
	\[
	\frac{\lambda_{1,p_j}(\Omega_j)^{\frac{p_j-2}{p_j}}}
	{D_j^2\Lambda_{p_j}(\Omega_j)}
	>L
	>
	\frac{P(\mu)}{D^2}.
	\]
	By \eqref{eq:C1}, we may choose
	\(g_j\in C^{0,1}(\overline{\Omega_j})\) such that
	\[
	\int g_j\,d\mu_j=0,
	\qquad
	\int g_j^2\,d\mu_j=1,
	\qquad
	\mathcal Q_{p_j,\Omega_j}(g_j)<\frac{1}{LD_j^2}.
	\]
	Lemma~\ref{lem:deg-compact} yields, after passing to a subsequence, a
	function \(g\in H^1(\mu)\) such that \(g_j\) converges strongly to \(g\)
	in the varying \(L^2\) spaces. Hence $\int g\,d\mu=0,\,
	\int g^2\,d\mu=1,$
	and
	\[
	\mathcal E_\mu(g)
	\leq
	\liminf_{j\to\infty}\mathcal Q_{p_j,\Omega_j}(g_j)
	\leq
	\lim_{j\to\infty}\frac{1}{LD_j^2}
	=
	\frac{1}{LD^2}.
	\]
	If \(\mu\) is a point mass, then every \(L^2(\mu)\) function is constant
	\(\mu\)-a.e.; the condition \(\int g\,d\mu=0\) would therefore imply
	\(g=0\) \(\mu\)-a.e., contradicting \(\int g^2\,d\mu=1\). Suppose now that \(\mu\) is not a point mass. Since $\operatorname{Var}_\mu(g)=1$,
	the definition of \(P(\mu)\) gives
	\[
	1
	=
	\operatorname{Var}_\mu(g)
	\leq
	P(\mu)\mathcal E_\mu(g)
	\leq
	\frac{P(\mu)}{LD^2},
	\]
	which is a contradiction. Therefore,
	we complete the proof.
\end{proof}

\subsection{Localized eigenvalues and the limiting convex potential}
\label{sec:local}
In this subsection, we assume
\[
p_j\downarrow2,\qquad
\Omega_j\in\mathcal C_N,\qquad
\lambda_j:=\lambda_{1,p_j}(\Omega_j),
\]
and let \(u_j\) be the positive first eigenfunction on \(\Omega_j\),
normalized by $\|u_j\|_{L^{p_j}(\Omega_j)}=1.$ Set
\[
\mu_j:=u_j^{p_j}\,dx,
\qquad
\mu_j\rightharpoonup\mu,
\qquad
\tau_j:=\lambda_j^{-1/p_j}\longrightarrow0,
\]
and assume that the domains \(\Omega_j\) are contained in a fixed
bounded set. Set
\[
\varepsilon_j:=p_j-2,
\qquad
s_j:=\lambda_j^{\varepsilon_j/p_j},
\qquad
a_j:=\lambda_j^{-\varepsilon_j/p_j}
|\nabla\log u_j|^{\varepsilon_j}.
\]
Assume that \(\mu\) is not a point mass. Recall from
subsection~\ref{natation} that $S:=\operatorname{aff}(\operatorname{supp}\mu),\,
A:=\operatorname{ri}(\operatorname{supp}\mu).$
We identify \(S\) isometrically with \(\mathbb R^d\) and denote by
\(\pi:\mathbb R^N\to S\) the orthogonal projection. Since \(\mu\) is
log-concave, there exists a finite convex function \(v:A\to\mathbb R\)
such that
\[
d\mu=\rho\,dx_S,
\qquad
\rho=e^{-2v}
\quad\text{on }A.
\]
Moreover, \(A\) is a bounded open convex subset of \(S\), and $\operatorname{supp}\mu=\overline A,\,
\mu(\partial_S A)=0.$

For a bounded open convex set \(U\subset S\) with \(U\cap A\neq\varnothing\),
define, for all sufficiently large \(j\),
\[D_j(U):=\Omega_j\cap\pi^{-1}(U),
\qquad
e_j(U):=
\frac{\lambda_{1,p_j}(D_j(U))-\lambda_j}{s_j}\geq0.\]
The normalization by \(s_j\) is natural because, after subtracting the
leading-order eigenvalue \(\lambda_j\), the first nontrivial correction
is quadratic and converges to a Dirichlet energy on the limiting affine
space \(S\), hence to the ordinary Laplacian scale.

Since \(U\) is relatively open in \(S\), the set \(\pi^{-1}(U)\) is
open in \(\mathbb R^N\). Hence, by the  Portmanteau theorem
\cite[Theorem~2.1]{Billingsley1999}
\[
\liminf_{j\to\infty}\mu_j(\pi^{-1}(U))
\geq
\mu(\pi^{-1}(U))
=
\mu(U)>0.
\]
Thus \(D_j(U)=\Omega_j\cap\pi^{-1}(U)\) is nonempty for all
sufficiently large \(j\). For \(U\Subset A\), set
\[e_\mu(U)
:=
\inf_{f\in H_0^1(U)\setminus\{0\}}
\frac{\displaystyle\int_U|\nabla_S f|^2\rho\,dx_S}
{\displaystyle\int_Uf^2\rho\,dx_S}.\]
Since \(\rho\) and \(\rho^{-1}\) are bounded on \(\overline U\), the
weighted and unweighted \(H_0^1(U)\) spaces coincide.

\smallskip

We next localize the first eigenfunction by multiplying it with a
smooth cutoff on the limiting affine space. The following lemma
identifies the second-order energy cost of this localization and
yields the corresponding upper bound for the localized first
eigenvalues.

\begin{lemma}
	\label{lem:local-upper}
	Let \(\psi\in C_c^\infty(S)\) be nonnegative. Then
	\begin{equation}\label{eq:trial-limit}
		\frac{
			\displaystyle\int_{\Omega_j}
			|\nabla(u_j\psi\circ\pi)|^{p_j}\,dx
			-
			\lambda_j\int_{\Omega_j}
			u_j^{p_j}|\psi\circ\pi|^{p_j}\,dx
		}{s_j}
		\longrightarrow
		\int_A|\nabla_S\psi|^2\,d\mu.
	\end{equation}
	If $\operatorname{supp}\psi\Subset U,\,
	\int\psi^2\,d\mu>0,$
	then
	\begin{equation}\label{eq:upper-quotient}
		\limsup_{j\to\infty}e_j(U)
		\leq
		\frac{\displaystyle\int_A|\nabla_S\psi|^2\,d\mu}
		{\displaystyle\int_A\psi^2\,d\mu}.
	\end{equation}
	In particular, if \(\overline A\subset U\), then $e_j(U)\rightarrow0.$
\end{lemma}

\begin{proof}
	We suppress the composition with \(\pi\), writing
\(\psi\) and \(\nabla\psi\) for
\(\psi\circ\pi\) and \(\nabla(\psi\circ\pi)\), respectively. Write
	\[
	\mathbf Z_j:=\tau_j\nabla\log u_j,
	\qquad
	Y_{j,t}:=\psi\mathbf Z_j+t\tau_j\nabla\psi.
	\]
	Then $\nabla(u_j\psi)
	=
	\lambda_j^{1/p_j}u_jY_{j,1}.$ Testing the eigenvalue equation with \(u_j\psi^{p_j}\) gives
	\[
	\int
	\left(
	|\psi\mathbf Z_j|^{p_j}
	+
	p_j\tau_j\psi^{p_j-1}
	|\mathbf Z_j|^{p_j-2}
	\mathbf Z_j\cdot\nabla\psi
	\right)d\mu_j
	=
	\int\psi^{p_j}\,d\mu_j.
	\]
	Applying the integral Taylor formula to $t\longmapsto
	|\psi\mathbf Z_j+t\tau_j\nabla\psi|^{p_j}$
	and using $\frac{\lambda_j\tau_j^2}{s_j}=1,$
	we obtain
	\begin{equation}\label{eq:Taylor}
		\begin{split}
			&\frac{
				\int_{\Omega_j}|\nabla(u_j\psi)|^{p_j}\,dx
				-
				\lambda_j\int_{\Omega_j}u_j^{p_j}\psi^{p_j}\,dx
			}{s_j} \\
			&\qquad = p_j\int_0^1(1-t)\int
			|Y_{j,t}|^{\varepsilon_j}
			\left(
			|\nabla\psi|^2
			+
			\varepsilon_j
			\frac{(Y_{j,t}\cdot\nabla\psi)^2}{|Y_{j,t}|^2}
			\right)d\mu_j\,dt,
		\end{split}
	\end{equation}
	where the quotient in the second term is set equal to zero when
	\(Y_{j,t}=0\).
	
	\smallskip
	
	We next prove that, for every \(\eta>0\),
	\begin{equation}\label{eq:Y-to-one}
		\sup_{0\leq t\leq1}
		\int_{\{\psi\geq\eta\}}
		\bigl||Y_{j,t}|^{\varepsilon_j}-1\bigr|\,d\mu_j
		\longrightarrow0.
	\end{equation}
	By \eqref{eq:moments} in Proposition \ref{prop:inverse},
	\[
	\mu_j\!\left(\{|\mathbf Z_j|<\sqrt{\tau_j}\}\right)
	\leq
	\tau_j^{1/4}
	\int |\mathbf Z_j|^{-1/2}\,d\mu_j
	\leq
	7\tau_j^{1/4}
	\longrightarrow0.
	\]
	Set $G_{j,\eta}
	:=
	\{\psi\geq\eta,\ |\mathbf Z_j|\geq\sqrt{\tau_j}\}.$
	On \(G_{j,\eta}\), uniformly for \(t\in[0,1]\),
	\[
	\left|
	\frac{|Y_{j,t}|}{\psi|\mathbf Z_j|}-1
	\right|
	\leq
	\frac{\tau_j|\nabla\psi|}
	{\psi|\mathbf Z_j|}
	\leq
	C_\psi\frac{\sqrt{\tau_j}}{\eta}
	\longrightarrow0.
	\]
	Hence $|Y_{j,t}|^{\varepsilon_j}
	=
	\psi^{\varepsilon_j}
	|\mathbf Z_j|^{\varepsilon_j}(1+o(1))$
	uniformly on \(G_{j,\eta}\) and in \(t\). Since
	\[
	\psi^{\varepsilon_j}\longrightarrow1
	\quad\text{uniformly on }\{\psi\geq\eta\},
	\qquad
	|\mathbf Z_j|^{\varepsilon_j}=a_j\longrightarrow1
	\quad\text{in }L^2(\mu_j),
	\]
	we obtain
	\[
	\sup_{0\leq t\leq1}
	\int_{G_{j,\eta}}
	\bigl||Y_{j,t}|^{\varepsilon_j}-1\bigr|\,d\mu_j
	\longrightarrow0.
	\]
	
	It remains to control the complement of \(G_{j,\eta}\).
	For large \(j\), \(2\varepsilon_j<1\) and \(\tau_j\leq1\). Using $|x+y|^{2\varepsilon_j}
	\leq
	|x|^{2\varepsilon_j}+|y|^{2\varepsilon_j}$
	and $\int|\mathbf Z_j|^{2\varepsilon_j}\,d\mu_j
	=
	\int a_j^2\,d\mu_j
	\leq1,$ see \eqref{eq:moments} in Proposition \ref{prop:inverse},
	we have
	\[
	\sup_{0\leq t\leq1}
	\int |Y_{j,t}|^{2\varepsilon_j}\,d\mu_j
	\leq C_\psi.
	\]
	The preceding \(L^2\)-bound and
	Cauchy--Schwarz therefore imply
	\[
	\sup_{0\leq t\leq1}
	\int_{\{\psi\geq\eta\}\setminus G_{j,\eta}}
	\bigl||Y_{j,t}|^{\varepsilon_j}-1\bigr|\,d\mu_j
	\longrightarrow0.
	\]
	Combining the two regions proves \eqref{eq:Y-to-one}.

	\smallskip
	
	Since \(\psi\geq0\) is smooth and compactly supported,
	\[
	\sup_{\{0\leq\psi<\eta\}}|\nabla\psi|
	\longrightarrow0
	\qquad\text{as }\eta\downarrow0.
	\]
	Together with the uniform \(L^2\)-bound above and
	\eqref{eq:Y-to-one}, this implies
	\[
	\sup_{0\leq t\leq1}
	\left|
	\int
	|Y_{j,t}|^{\varepsilon_j}|\nabla\psi|^2\,d\mu_j
	-
	\int|\nabla\psi|^2\,d\mu_j
	\right|
	\longrightarrow0.
	\]
	Furthermore,
	\[
	\varepsilon_j
	\int
	|Y_{j,t}|^{\varepsilon_j}
	\frac{(Y_{j,t}\cdot\nabla\psi)^2}{|Y_{j,t}|^2}
	\,d\mu_j
	\longrightarrow0
	\]
	uniformly in \(t\). Since \(\mu_j\rightharpoonup\mu\),
	\[
	\int|\nabla\psi|^2\,d\mu_j
	\longrightarrow
	\int_A|\nabla_S\psi|^2\,d\mu,
	\]
	passing to the limit in \eqref{eq:Taylor} proves
	\eqref{eq:trial-limit}.
	
	\smallskip
	
	Now assume $\operatorname{supp}\psi\Subset U.$
	Then $u_j\psi\circ\pi\in W_0^{1,p_j}(D_j(U)).$
	Hence the Rayleigh principle gives
	\[
	\lambda_{1,p_j}(D_j(U))
	\leq
	\frac{
		\displaystyle\int_{\Omega_j}
		|\nabla(u_j\psi)|^{p_j}\,dx
	}{
		\displaystyle\int_{\Omega_j}
		u_j^{p_j}\psi^{p_j}\,dx
	}.
	\]
	Therefore
	\[
	e_j(U)
	\leq
	\frac{
		s_j^{-1}
		\left[
		\displaystyle\int_{\Omega_j}|\nabla(u_j\psi)|^{p_j}\,dx
		-
		\lambda_j\int_{\Omega_j}u_j^{p_j}\psi^{p_j}\,dx
		\right]
	}{
		\displaystyle\int\psi^{p_j}\,d\mu_j
	}.
	\]
	Since \(p_j\to2\), \(\psi^{p_j}\to\psi^2\) uniformly, and
	\(\mu_j\rightharpoonup\mu\),
	\[
	\int\psi^{p_j}\,d\mu_j
	\longrightarrow
	\int_A\psi^2\,d\mu>0.
	\]
	Combining this with \eqref{eq:trial-limit} yields
	\eqref{eq:upper-quotient}.
	
	\smallskip
	
	Finally, if \(\overline A\subset U\), choose
	\(\psi\in C_c^\infty(U)\) such that $\psi\equiv1$ in a neighborhood of $\overline A.$
	Since \(\operatorname{supp}\mu=\overline A\),
	\[
	\int_A|\nabla_S\psi|^2\,d\mu=0,
	\qquad
	\int_A\psi^2\,d\mu=1.
	\]
	Thus \eqref{eq:upper-quotient} gives $\limsup_{j\to\infty}e_j(U)\leq0.$
	Since \(e_j(U)\geq0\), we conclude that $e_j(U)\rightarrow0.$
\end{proof}

The localization estimate above gives the upper bound for the limiting local eigenvalue. The next lemma supplies the matching lower bound and therefore identifies the full limit.

\begin{lemma}\label{lem:window-limit}
	For every bounded open convex set \(U\Subset A\), $e_j(U)\rightarrow e_\mu(U).$
\end{lemma}

\begin{proof}
	We first prove the upper bound. By Lemma~\ref{lem:local-upper},
	\[
	\limsup_{j\to\infty}e_j(U)
	\leq
	\frac{\displaystyle\int_U|\nabla_S\psi|^2\,d\mu}
	{\displaystyle\int_U\psi^2\,d\mu}
	\]
	for every nonnegative \(0\neq\psi\in C_c^\infty(U)\), and thus $\limsup_{j\to\infty}e_j(U)\leq e_\mu(U).$
	
	\smallskip
	
	For the reverse inequality, pass to a subsequence, such
	that $e_j(U)\rightarrow
	\liminf_{\ell\to\infty}e_\ell(U).$
	The preceding upper bound shows that this sequence is bounded. For all
	sufficiently large \(j\), choose a nonnegative
	\(z_j\in C_c^\infty(D_j(U))\) such that
	\[
	\int z_j^{p_j}\,dx=1,
	\qquad
	\int|\nabla z_j|^{p_j}\,dx
	\leq
	\lambda_{1,p_j}(D_j(U))+\frac{s_j}{j}.
	\]
	Set $w_j:=\frac{z_j}{u_j},\,
	F_j:=w_j^{p_j/2}.$
	Since \(\operatorname{supp}z_j\Subset\Omega_j\), \(u_j>0\) and
	\(u_j\in C^1_{\mathrm{loc}}(\Omega_j)\), the quotient
	\(w_j:=z_j/u_j\) is well defined and belongs to \(C_c^1(\Omega_j)\).
	Hence, since \(p_j/2>1\),
	\[
	F_j:=w_j^{p_j/2}\in C_c^1(\Omega_j)
	\subset C^{0,1}(\overline{\Omega_j}).
	\] Moreover,
	\begin{equation}\label{eq:Fj-normalization}
		\int F_j^2\,d\mu_j
		=
		\int w_j^{p_j}u_j^{p_j}\,dx
		=
		1.
	\end{equation}
	
	Apply \eqref{eq:cpaboptimal} with $a=w_j\nabla u_j,\,
	b=u_j\nabla w_j,$
	integration gives
	\begin{align*}
		\int|\nabla z_j|^{p_j}\,dx
		&-
		\int w_j^{p_j}|\nabla u_j|^{p_j}\,dx-
		p_j\int
		u_jw_j^{p_j-1}
		|\nabla u_j|^{p_j-2}
		\nabla u_j\cdot\nabla w_j\,dx\\
		&\geq
		\frac{p_j}{2}
		\int
		w_j^{p_j-2}|\nabla u_j|^{p_j-2}
		u_j^2|\nabla w_j|^2\,dx.
	\end{align*}
	Testing the eigenvalue equation for \(u_j\) with
	\(u_jw_j^{p_j}\) yields
	\[
	\int w_j^{p_j}|\nabla u_j|^{p_j}\,dx
	+
	p_j\int
	u_jw_j^{p_j-1}
	|\nabla u_j|^{p_j-2}
	\nabla u_j\cdot\nabla w_j\,dx
	=
	\lambda_j\int z_j^{p_j}\,dx
	=
	\lambda_j.
	\]
	Furthermore, $\nabla F_j
	=
	\frac{p_j}{2}w_j^{p_j/2-1}\nabla w_j,$
	and hence
	\begin{align*}
		\frac{p_j}{2}
		w_j^{p_j-2}|\nabla u_j|^{p_j-2}
		u_j^2|\nabla w_j|^2=
		\frac{2s_j}{p_j}
		a_j|\nabla F_j|^2\,d\mu_j/dx.
	\end{align*}
	Consequently,
	\begin{equation}\label{eq:Picone-lower}
		\frac{
			\displaystyle\int|\nabla z_j|^{p_j}\,dx-\lambda_j
		}{s_j}
		\geq
		\frac{2}{p_j}
		\mathcal Q_{p_j,\Omega_j}(F_j).
	\end{equation}
	By the choice of \(z_j\),
	\[
	\frac{2}{p_j}\mathcal Q_{p_j,\Omega_j}(F_j)
	\leq
	e_j(U)+\frac1j.
	\]
	Together with \eqref{eq:Fj-normalization}, this gives the uniform
	bounds required in Lemma~\ref{lem:deg-compact}. Hence, after passing to
	a further subsequence, $F_j\rightarrow F$
	strongly in the varying \(L^2\) spaces for some \(F\in H^1(\mu)\), and
	\begin{equation}\label{eq:F-liminf}
		\mathcal E_\mu(F)
		\leq
		\liminf_{j\to\infty}
		\mathcal Q_{p_j,\Omega_j}(F_j).
	\end{equation}
	Strong convergence and \eqref{eq:Fj-normalization} give $\int_A F^2\,d\mu=1.$
	
	\smallskip
	
	Since \(F_j\) is supported in \(\pi^{-1}(U)\), localization of the
	strong varying-\(L^2\) convergence gives
	\[
	F=0
	\qquad
	\mu\text{-a.e. on }A\setminus\overline U.
	\]
	Choose an open set \(V\) such that $\overline U\Subset V\Subset A.$
	Since \(\rho\) and \(\rho^{-1}\) are bounded on \(\overline V\),
	\(F\in H^1(V)\) in the ordinary Sobolev sense. Moreover, $F=0$ a.e. on $V\setminus U,$
	because \(\partial_SU\) has zero \(d\)-dimensional Lebesgue measure.
	As \(U\) is a bounded convex domain in \(S\), it is Lipschitz, and the
	zero-extension characterization of \(H_0^1(U)\) therefore gives $F|_U\in H_0^1(U).$
	Since \(F=0\) \(\mu\)-a.e. on \(A\setminus U\), $\int_UF^2\,d\mu=1.$
	Thus, by the definition of \(e_\mu(U)\), $e_\mu(U)
	\leq
	\mathcal E_\mu(F).$
	Thus, we obtain
	\begin{align*}
		e_\mu(U)
		&\leq
		\liminf_{j\to\infty}
		\mathcal Q_{p_j,\Omega_j}(F_j)\leq
		\liminf_{j\to\infty}
		\frac{p_j}{2}
		\left(e_j(U)+\frac1j\right)=
		\liminf_{j\to\infty}e_j(U),
	\end{align*}
	we complete the proof.
\end{proof}

We next show that the localized limiting eigenvalues inherit the
Minkowski convexity of the first Dirichlet \(p_j\)-eigenvalues. This
convexity will be used to identify the spatial dependence of the
limiting local spectral quantity and, subsequently, to recover the
limiting convex potential.

\begin{lemma}
	\label{lem:Minkowski}
	Let \(U_0,U_1\subset S\) be bounded open convex sets with
	\(U_i\cap A\neq\varnothing\) for \(i=0,1\), and let \(0<t<1\). Then, for all sufficiently large \(j\),
	\begin{equation}\label{eq:ej-convex}
		e_j((1-t)U_0+tU_1)
		\leq
		(1-t)e_j(U_0)+t e_j(U_1).
	\end{equation}
	If \(U_0,U_1\Subset A\), then
	\begin{equation}\label{eq:emu-convex}
		e_\mu((1-t)U_0+tU_1)
		\leq
		(1-t)e_\mu(U_0)+t e_\mu(U_1).
	\end{equation}
	Consequently, for every \(r>0\), the map $x\longmapsto e_\mu(B_r^S(x))$
	is convex on
	\[
	A_r:=\{x\in A:\operatorname{dist}_S(x,\partial_S A)>r\}.
	\]
\end{lemma}

\begin{proof}
	Since \(D_j(U_i)\) is nonempty for all sufficiently large \(j\), fix such a \(j\), and set
	\[
	D_i:=D_j(U_i),
	\qquad
	\Lambda_i:=\lambda_{1,p_j}(D_i),
	\qquad i=0,1.
	\]
	The sets \(D_0,D_1\) are bounded open convex domains. The
	Brunn--Minkowski inequality for the first \(p_j\)-eigenvalue
	\cite[Theorem~1 and Example~3]{CrastaFragala2020} gives
	\begin{equation}\label{eq:BM-p}
		\lambda_{1,p_j}((1-t)D_0+tD_1)^{-1/p_j}
		\geq
		(1-t)\Lambda_0^{-1/p_j}
		+
		t\Lambda_1^{-1/p_j}.
	\end{equation}
	Since \(s\mapsto s^{-1/p_j}\) is convex and decreasing on
	\((0,\infty)\),
	\[
	(1-t)\Lambda_0^{-1/p_j}
	+
	t\Lambda_1^{-1/p_j}
	\geq
	\bigl((1-t)\Lambda_0+t\Lambda_1\bigr)^{-1/p_j}.
	\]
	Combining this with \eqref{eq:BM-p} and using the monotonicity of
	\(s\mapsto s^{-1/p_j}\), we obtain
	\begin{equation}\label{eq:lambda-convex}
		\lambda_{1,p_j}((1-t)D_0+tD_1)
		\leq
		(1-t)\Lambda_0+t\Lambda_1.
	\end{equation}
	
	By the convexity of \(\Omega_j\) and the linearity of \(\pi\),
	\[
	(1-t)D_j(U_0)+tD_j(U_1)
	\subset
	D_j((1-t)U_0+tU_1).
	\]
	Hence domain monotonicity and \eqref{eq:lambda-convex} yield
	\[
	\lambda_{1,p_j}(D_j((1-t)U_0+tU_1))
	\leq
	(1-t)\lambda_{1,p_j}(D_j(U_0))
	+
	t\lambda_{1,p_j}(D_j(U_1)).
	\]
	Subtracting \(\lambda_j\) and dividing by \(s_j>0\) proves
	\eqref{eq:ej-convex}.
	
	\smallskip
	
	Now assume \(U_0,U_1\Subset A\). Since \(A\) is convex, $(1-t)U_0+tU_1\Subset A.$
	Lemma~\ref{lem:window-limit} therefore gives
	\[
	e_j(U_i)\longrightarrow e_\mu(U_i),
	\qquad i=0,1,
	\]
	and $e_j((1-t)U_0+tU_1)
	\longrightarrow
	e_\mu((1-t)U_0+tU_1).$
	Passing to the limit in \eqref{eq:ej-convex} proves
	\eqref{eq:emu-convex}.
	
	\smallskip
	
	Finally, let \(x_0,x_1\in A_r\). Then $B_r^S(x_0),B_r^S(x_1)\Subset A$
	and
	\[
	(1-t)B_r^S(x_0)+tB_r^S(x_1)
	=
	B_r^S((1-t)x_0+tx_1).
	\]
	In particular, the latter ball is contained compactly in \(A\), so
	\((1-t)x_0+tx_1\in A_r\). Applying \eqref{eq:emu-convex} gives
	\[
	e_\mu(B_r^S((1-t)x_0+tx_1))
	\leq
	(1-t)e_\mu(B_r^S(x_0))
	+
	t e_\mu(B_r^S(x_1)),
	\]
	which proves the asserted convexity.
\end{proof}

We next show that the convexity of the localized limiting eigenvalues
determines a genuine convex Schr\"odinger potential. More precisely,
the small-ball expansion of \(e_\rho(B_r(x))\) identifies the quantity
\(|\nabla v|^2-\Delta v\) almost everywhere, while the convexity in the
ball center promotes this a.e. quantity to a finite convex function and
forces the Hessian measure of \(v\) to be absolutely continuous.

\begin{proposition}\label{prop:potential}
	Let \(A\subset\mathbb R^d\) be a bounded open convex set, let
	\(v:A\to\mathbb R\) be finite and convex, and set $\rho:=e^{-2v}.$
	For \(U\Subset A\), define
	\[
	e_\rho(U)
	:=
	\inf_{f\in H_0^1(U)\setminus\{0\}}
	\frac{\displaystyle\int_U|\nabla f|^2\rho\,dx}
	{\displaystyle\int_Uf^2\rho\,dx}.
	\]
	Assume that, for every sufficiently small \(r>0\), the map $x\longmapsto e_\rho(B_r(x))$
	is convex on
	\[
	A_{-r}:=\{x\in A:B_r(x)\Subset A\}.
	\]
	Then \(v\in C_{\mathrm{loc}}^{1,1}(A)\), and there exists a finite
	convex function \(V:A\to\mathbb R\) such that
	\begin{equation}\label{eq:potential}
		V
		=
		|\nabla v|^2-\Delta v
		\qquad\text{a.e. in }A.
	\end{equation} In particular, if
	\(w:=e^{-v}\), then $-\Delta w+Vw=0$ in $A$
	in the distributional sense. No boundary condition on \(w\) is asserted.
\end{proposition}

\begin{proof}
	Set
	\[
	\kappa_d:=\lambda_{1,2}(B_1^d),
	\qquad
	V_r(x):=e_\rho(B_r(x))-\frac{\kappa_d}{r^2}.
	\]
	By assumption, \(V_r\) is convex on \(A_{-r}\) for all sufficiently
	small \(r\).
	
	\smallskip
	\noindent\textbf{Step 1: small-ball asymptotics at Alexandrov points.}
	By Alexandrov's theorem for convex functions
	\cite[Theorem 6.9]{EG1992}, for a.e. \(x\in A\),
	\[
	v(x+h)
	=
	v(x)+b\cdot h+\frac12 h^THh+o(|h|^2),
	\]
	where \(b=\nabla v(x)\) and \(H\geq0\) is symmetric. For such \(x\), set
	\[
	q_x(x+h)
	:=
	v(x)+b\cdot h+\frac12h^THh
	\]
	and $\delta_r
	:=
	\sup_{|h|\leq r}
	|v(x+h)-q_x(x+h)|.$
	Then $\delta_r=o(r^2).$
	On \(B_r(x)\), $e^{-2\delta_r}e^{-2q_x}
	\leq
	\rho
	\leq
	e^{2\delta_r}e^{-2q_x},$
	and hence
	\begin{equation}\label{eq:weight-comparison}
		e^{-4\delta_r}
		e_{e^{-2q_x}}(B_r(x))
		\leq
		e_\rho(B_r(x))
		\leq
		e^{4\delta_r}
		e_{e^{-2q_x}}(B_r(x)).
	\end{equation}
	
	Multiplication by \(e^{-q_x}\) transforms the weighted Rayleigh
	quotient with weight \(e^{-2q_x}\) into the Schr\"odinger quotient
	with potential
	\[
	W_x(x+h)
	=
	|\nabla q_x(x+h)|^2-\Delta q_x(x+h)
	=
	|b+Hh|^2-\operatorname{tr}H.
	\]
	Thus, $e_{e^{-2q_x}}(B_r(x))
	=
	\lambda_1(-\Delta+W_x,B_r(x)).$
	Since
	\[
	W_x(x+h)
	=
	|b|^2-\operatorname{tr}H+O_x(r)
	\qquad |h|\leq r,
	\]
	the min--max principle gives
	\[
	e_{e^{-2q_x}}(B_r(x))
	=
	\frac{\kappa_d}{r^2}
	+
	|b|^2-\operatorname{tr}H
	+
	O_x(r).
	\]
	Moreover, $e^{4\delta_r}-1=o(r^2),$
	while $e_{e^{-2q_x}}(B_r(x))=O_x(r^{-2}).$
	Hence \eqref{eq:weight-comparison} changes the preceding eigenvalue by
	\(o(1)\). Therefore
	\begin{equation}\label{eq:a.e.-potential}
		V_r(x)
		\longrightarrow
		|\nabla v(x)|^2-\operatorname{tr}H(x)
		\qquad\text{for a.e. }x\in A.
	\end{equation}
	
	\smallskip
	\noindent\textbf{Step 2: local uniform convergence.}
	Let \(G\subset A\) be the full-measure set on which
	\eqref{eq:a.e.-potential} holds. Fix \(K\Subset A\). Since \(G\) is
	dense in \(A\), we may choose a convex polytope \(P\) such that $K\Subset\operatorname{int}P\Subset A,$
	all vertices of \(P\) belong to \(G\), and $y\in G\cap\operatorname{int}P.$
	For all sufficiently small \(r\), \(V_r\) is convex on a neighborhood
	of \(P\). Since \(V_r\) converges at every vertex of \(P\), convexity
	gives $\sup_P V_r\leq C_P.$
	
	Choose $K\Subset K_1\Subset\operatorname{int}P.$
	Since \(y\in\operatorname{int}P\), there exists \(\theta\in(0,1)\),
	independent of \(x\in K_1\), such that
	\[
	z_x:=\frac{y-\theta x}{1-\theta}\in P
	\qquad
	\text{for every }x\in K_1.
	\]
	Thus $y=\theta x+(1-\theta)z_x,$
	and convexity gives
	\[
	V_r(y)
	\leq
	\theta V_r(x)+(1-\theta)V_r(z_x).
	\]
	Since \(V_r(y)\) converges and \(V_r(z_x)\leq C_P\), $\inf_{K_1}V_r\geq-C_{K_1}.$
	Hence \(\{V_r\}\) is uniformly bounded on \(K_1\). Standard interior
	estimates for convex functions then give a uniform Lipschitz bound on
	\(K\).

	It follows that every sequence \(r_j\downarrow0\) has a subsequence
	for which \(V_{r_j}\) converges locally uniformly to a finite convex
	function. By \eqref{eq:a.e.-potential}, any two such limits agree on
	\(G\), hence everywhere by continuity. Consequently there exists a
	finite convex function \(V\) on \(A\) such that
	\begin{equation}\label{eq:uniform-potential}
		V_r\longrightarrow V
		\qquad\text{locally uniformly in }A.
	\end{equation}
	
	\smallskip
	\noindent\textbf{Step 3: elimination of the singular part of \(D^2v\).}
	Since \(v\) is convex, its distributional Hessian $M:=D^2v$
	is a positive semidefinite matrix-valued Radon measure. Set $m:=\operatorname{tr}M=\Delta v.$
	Let \(\zeta\) be the positive \(L^2\)-normalized first Dirichlet
	eigenfunction of \(B_1^d\), and define
	\[
	\zeta_{r,x}(y)
	:=
	r^{-d/2}
	\zeta\!\left(\frac{y-x}{r}\right).
	\]
	Let $w:=e^{-v}.$
	For \(B_r(x)\Subset A\), both \(w\) and \(w^{-1}\) are locally
	Lipschitz and bounded, so $\frac{\zeta_{r,x}}{w}\in H_0^1(B_r(x)).$
	Using this function in the Rayleigh quotient gives
	\[
	e_\rho(B_r(x))
	\leq
	\int_{B_r(x)}
	|\nabla\zeta_{r,x}+\zeta_{r,x}\nabla v|^2\,dy.
	\]
	Since $\int|\nabla\zeta_{r,x}|^2\,dy
	=
	\frac{\kappa_d}{r^2},$
	and $2\int
	\zeta_{r,x}\nabla\zeta_{r,x}\cdot\nabla v\,dy
	=
	-\int
	\zeta_{r,x}^2\,dm,$
	we obtain
	\begin{equation}\label{eq:trace-test}
		V_r(x)
		\leq
		\int_{B_r(x)}
		\zeta_{r,x}^2|\nabla v|^2\,dy
		-
		\int_{B_r(x)}
		\zeta_{r,x}^2\,dm.
	\end{equation}
	The integration by parts is legitimate because the zero extension of
	\(\zeta_{r,x}^2\) is \(C_c^1(\mathbb R^d)\).
	
	\smallskip
	
	Fix \(K\Subset A\), and choose $K\Subset K_1\Subset A.$
	For all sufficiently small \(r\) and \(x\in K\),
	\(B_r(x)\subset K_1\). By \eqref{eq:uniform-potential}, $V_r(x)\geq-C_K.$
	Since a finite convex function is locally Lipschitz,
	\[
	|\nabla v|\leq L_{K_1}
	\qquad\text{a.e. on }K_1.
	\]
	Because \(\|\zeta_{r,x}\|_2=1\), \eqref{eq:trace-test} yields $\int\zeta_{r,x}^2\,dm
	\leq
	C_K.$
	Let $c_d:=\min_{\overline{B_{1/2}}}\zeta^2>0.$
	Then $\zeta_{r,x}^2
	\geq
	c_dr^{-d}$ on $B_{r/2}(x),$
	and hence $m(B_{r/2}(x))
	\leq
	C_Kr^d.$
	Thus, after changing the constant, $m(B_s(x))
	\leq
	C_Ks^d$
	for \(x\in K\) and all sufficiently small \(s>0\). A standard
covering argument then gives \(m\ll dx\) on \(K\). Writing \(m=g\,dx\), the Lebesgue
	differentiation theorem then gives
	\[
	g(x)
	=
	\lim_{r\downarrow0}
	\frac{m(B_r(x))}{|B_r|}
	\leq
	\frac{C_K}{|B_1|}
	\quad\text{for a.e. }x\in K.
	\]
	Hence, $\frac{dm}{dx}\in L^\infty_{\mathrm{loc}}(A).$
	
	\smallskip
	
	Since \(M=D^2v\) is positive semidefinite and \(m=\operatorname{tr}M\),
	we have \(M\ll m\). Indeed, \(m(E)=0\) implies that the positive
	semidefinite matrix \(M(E)\) has zero trace, hence \(M(E)=0\).
	Writing $M=B\,m,$
	the Radon--Nikodym density \(B\) is positive semidefinite and $\operatorname{tr}B=1\, m\text{-a.e.}$
	Thus \(|B|\leq1\) \(m\)-a.e., and consequently $D^2v\in L^\infty_{\mathrm{loc}}(A).$
	Therefore
	\[
	v\in W_{\mathrm{loc}}^{2,\infty}(A)
	\subset
	C_{\mathrm{loc}}^{1,1}(A).
	\]
	At almost every point, the Alexandrov Hessian agrees with the weak
	Hessian. Combining \eqref{eq:a.e.-potential} and
	\eqref{eq:uniform-potential}, we obtain
	\[
	V
	=
	|\nabla v|^2-\Delta v
	\qquad\text{a.e. in }A.
	\]
	Thus the right-hand side of \eqref{eq:potential} has the finite convex
	representative \(V\).
	
	Finally, since \(v\in C_{\mathrm{loc}}^{1,1}(A)\), the chain rule gives
	\[
	\Delta(e^{-v})
	=
	\bigl(|\nabla v|^2-\Delta v\bigr)e^{-v}
	=
	Ve^{-v}
	\]
	almost everywhere, hence also distributionally. Therefore $-\Delta w+Vw=0$ in $A.$
\end{proof}

Return to the collapsing sequence of Section~\ref{sec:local}.
Lemmas~\ref{lem:Minkowski} and Proposition \ref{prop:potential} yield a finite
convex function \(V\) on \(A\) such that, with \(w:=\sqrt{\rho}=e^{-v}\),
\[
-\Delta_S w+Vw=0
\qquad\text{in }A.
\]
We next recover the boundary information needed to apply the linear
fundamental-gap estimate.

\begin{proposition}\label{prop:boundary}
	There exist bounded open convex sets \(U_n\Subset A\) such that $e_\mu(U_n)\rightarrow0.$
	Let \(\varphi_n\) be the positive first Dirichlet eigenfunction of
	\(-\Delta_S+V\) on \(U_n\), normalized by $\int_{U_n}\varphi_n^2\,dx_S=1,$
	and extended by zero to \(S\). Extending \(w=\sqrt{\rho}\) by zero
	outside \(A\), we have
	\[
	\varphi_n\longrightarrow w
	\qquad\text{strongly in }L^2(S,dx_S),
	\]
	and $\varphi_n^2\,dx_S\rightarrow\mu$ in total variation.
	Consequently, $P(\mu)
	\leq
	\frac{\operatorname{diam}(A)^2}{3\pi^2}.$
\end{proposition}

\begin{proof}
	For a bounded convex set \(K\subset S\), let
	\[
	h_K(\nu):=\sup_{x\in K}x\cdot\nu,
	\qquad |\nu|=1,
	\]
	denote its support function. Fix a ball \(B\Subset A\). Choose \(\sigma>0\) such that $\overline B+\sigma\overline B_1\Subset A.$
	For \(0<t<1/2\), set
	\[
	\eta_t:=\frac{t\sigma}{2},
	\qquad
	U_0^{(t)}:=A+B_{\eta_t},
	\qquad
	U_t:=(1-t)U_0^{(t)}+tB.
	\]
	Then $\overline A\subset U_0^{(t)}.$
	Moreover, from $\overline B+\sigma\overline B_1\subset\overline A$
	we obtain $h_{\overline B}(\nu)+\sigma
	\leq
	h_{\overline A}(\nu)$ when $|\nu|=1.$
	Since $h_{\overline{U_0^{(t)}}}
	=
	h_{\overline A}+\eta_t,$
	we have
	\begin{align*}
		h_{\overline{U_t}}(\nu)
		&=
		(1-t)\bigl(h_{\overline A}(\nu)+\eta_t\bigr)
		+t h_{\overline B}(\nu)\\
		&\leq
		h_{\overline A}(\nu)-t\sigma+(1-t)\eta_t\leq
		h_{\overline A}(\nu)-\frac{t\sigma}{2}.
	\end{align*}
	Thus $\overline{U_t}+\frac{t\sigma}{2}\overline B_1
	\subset\overline A,$
	and hence $U_t\Subset A.$
	Applying the same estimate to \(\nu\) and \(-\nu\), and using the
	width representation of the diameter of a convex set, gives
	\[
	\operatorname{diam}(U_t)
	=
	\sup_{|\nu|=1}
	\bigl(
	h_{\overline{U_t}}(\nu)+h_{\overline{U_t}}(-\nu)
	\bigr)
	\leq
	\operatorname{diam}(A).
	\]
	
	For each fixed \(t\), Lemma~\ref{lem:local-upper} gives $e_j(U_0^{(t)})\rightarrow0.$
	By Lemma~\ref{lem:Minkowski},
	\[
	e_j(U_t)
	\leq
	(1-t)e_j(U_0^{(t)})+t e_j(B).
	\]
	Since \(U_t,B\Subset A\), Lemma~\ref{lem:window-limit} allows us to
	pass to the limit \(j\to\infty\), and hence
	\begin{equation}\label{eq:capacity}
		e_\mu(U_t)
		\leq
		t\,e_\mu(B)
		\longrightarrow0
		\qquad\text{as }t\downarrow0.
	\end{equation}
	
	Let \(f_t\geq0\) be the first minimizer for \(e_\mu(U_t)\), normalized
	by $\int_{U_t}f_t^2\,d\mu=1,$
	and extend \(f_t\) by zero to \(A\). Then $f_t\in H^1(\mu),\,
	\mathcal E_\mu(f_t)=e_\mu(U_t).$
	By Lemma~\ref{lemma:ordinary3},
	\[
	\operatorname{Var}_\mu(f_t)
	\leq
	\frac{\operatorname{diam}(A)^2}{\pi^2}
	\mathcal E_\mu(f_t).
	\]
	Using \eqref{eq:capacity} and the normalization of \(f_t\), $1-
	\left(
	\int_A f_t\,d\mu
	\right)^2
	\rightarrow0.$
	Since \(f_t\geq0\), $\int_A f_t\,d\mu\rightarrow1,$
	and consequently
	\begin{equation}\label{eq:ft-one}
		\|f_t-1\|_{L^2(\mu)}^2
		=
		2\left(
		1-\int_Af_t\,d\mu
		\right)
		\longrightarrow0.
	\end{equation}
	
	Recall from Proposition~\ref{prop:potential} that
	\[
	w=e^{-v}=\sqrt{\rho},
	\qquad
	-\Delta_S w+Vw=0
	\quad\text{in }A.
	\]
	For \(f\in C_c^\infty(U_t)\), integration by parts gives
	\begin{align*}
		\int_{U_t}|\nabla_S(wf)|^2\,dx_S
		&=
		\int_{U_t}w^2|\nabla_Sf|^2\,dx_S
		+
		\int_{U_t}f^2|\nabla_Sw|^2\,dx_S+
		2\int_{U_t}wf\nabla_Sw\cdot\nabla_Sf\,dx_S\\
		&=
		\int_{U_t}w^2|\nabla_Sf|^2\,dx_S
		-
		\int_{U_t}Vw^2f^2\,dx_S.
	\end{align*}
	Hence
	\begin{equation}\label{eq:linear-GST}
		\int_{U_t}w^2|\nabla_S f|^2\,dx_S
		=
		\int_{U_t}
		\left(
		|\nabla_S(wf)|^2+Vw^2f^2
		\right)dx_S.
	\end{equation}
	Since \(U_t\Subset A\),
	\[
	0<c_t\leq w\leq C_t,
	\qquad
	w,w^{-1}\in W^{1,\infty}(U_t),
	\qquad
	V\in L^\infty(U_t).
	\]
	Thus multiplication by \(w\) is an isomorphism of \(H_0^1(U_t)\),
	and \eqref{eq:linear-GST} extends to every \(f\in H_0^1(U_t)\).
	
	It follows that $\varphi_t:=wf_t$
	is the normalized positive first Dirichlet eigenfunction of
	\(-\Delta_S+V\) on \(U_t\), and
	\[
	\lambda_1(-\Delta_S+V,U_t)=e_\mu(U_t).
	\]
	After extending \(f_t\) by zero on \(A\setminus U_t\), and \(w\) and
	\(\varphi_t\) by zero to \(S\), \eqref{eq:ft-one} yields
	\[
	\|\varphi_t-w\|_{L^2(S)}^2
	=
	\int_A|f_t-1|^2\,d\mu
	\longrightarrow0.
	\]
	Since $\|\varphi_t\|_{L^2(S)}
	=
	\|w\|_{L^2(S)}
	=
	1,$
	we also have
	\[
	\|\varphi_t^2-w^2\|_{L^1(S)}
	\leq
	\|\varphi_t-w\|_{L^2(S)}
	\bigl(
	\|\varphi_t\|_{L^2(S)}+\|w\|_{L^2(S)}
	\bigr)
	\longrightarrow0.
	\]
	Therefore $\varphi_t^2\,dx_S
	\rightarrow
	w^2\,dx_S
	=
	\mu$
	in total variation.
	
	We next apply the sharp fundamental-gap estimate of
	Andrews--Clutterbuck \cite[Corollary~1.4]{AndrewsClutterbuck2011}.
	Since \(U_t\) is a bounded convex domain in \(S\) and \(V\) is convex
	on \(A\), we obtain
	\begin{equation}\label{eq:linear-gap-Ut}
		\lambda_2(-\Delta_S+V,U_t)
		-
		\lambda_1(-\Delta_S+V,U_t)
		\geq
		\frac{3\pi^2}{\operatorname{diam}(U_t)^2}.
	\end{equation}
	
	Let \(g\in C^\infty(S)\), and set $\overline g_t
	:=
	\int_Sg\varphi_t^2\,dx_S.$
	Then $(g-\overline g_t)\varphi_t
	\perp\varphi_t$ in $L^2(U_t).$
	The variational characterization of the second eigenvalue, together
	with the ground-state transform and \eqref{eq:linear-gap-Ut}, gives
	\[
	\operatorname{Var}_{\varphi_t^2dx_S}(g)
	\leq
	\frac{\operatorname{diam}(U_t)^2}{3\pi^2}
	\int_S|\nabla_Sg|^2\varphi_t^2\,dx_S.
	\]
	Since $\operatorname{diam}(U_t)\leq\operatorname{diam}(A),$
	we obtain
	\[
	\operatorname{Var}_{\varphi_t^2dx_S}(g)
	\leq
	\frac{\operatorname{diam}(A)^2}{3\pi^2}
	\int_S|\nabla_Sg|^2\varphi_t^2\,dx_S.
	\]
	Letting \(t\downarrow0\) and using $\varphi_t^2\rightarrow\rho$ in $L^1(S),$
	we conclude that
	\[
	\operatorname{Var}_\mu(g)
	\leq
	\frac{\operatorname{diam}(A)^2}{3\pi^2}
	\int_S|\nabla_Sg|^2\,d\mu.
	\]
	By density of smooth functions in \(H^1(\mu)\), this proves
	$P(\mu)
	\leq
	\frac{\operatorname{diam}(A)^2}{3\pi^2}.$
	
	\smallskip
	
	Finally, choose any sequence \(t_n\downarrow0\) and set $U_n:=U_{t_n},\,
	\varphi_n:=\varphi_{t_n}.$
	Then
	\[
	e_\mu(U_n)\to0,
	\qquad
	\varphi_n\to w
	\quad\text{in }L^2(S),
	\]
	and $\varphi_n^2\,dx_S\to\mu$
	in total variation.
\end{proof}

\subsection{The sharp limit measure}
\label{sec:completion}

We first treat the case in which the domains do not collapse.

\begin{lemma}\label{lem:noncollapse}
	Let \(p_j\downarrow2\), and let \(\Omega_j\in\mathcal C_N\) are contained in a fixed bounded set, and assume that $\overline{\Omega_j}\rightarrow K$
	in the Hausdorff distance, where \(K\subset\mathbb R^N\) is compact,
	convex, and has nonempty interior $\Omega_\infty:=\operatorname{int}K.$
	Let \(u_j\) be the positive first $p_j$-eigenfunction of \eqref{eq:introduction-eigenvalue-problem} on \(\Omega_j\),
	normalized by $\int_{\Omega_j}u_j^{p_j}\,dx=1.$
	Then $\lambda_{1,p_j}(\Omega_j)
	\rightarrow
	\lambda_{1,2}(\Omega_\infty),$
	and, after zero extension,
	\[
	u_j\longrightarrow u_\infty
	\quad\text{strongly in }L^2(\mathbb R^N),
	\qquad
	u_j^{p_j}\longrightarrow u_\infty^2
	\quad\text{in }L^1(\mathbb R^N),
	\]
	where \(u_\infty\) is the positive first Dirichlet
	eigenfunction of \(-\Delta\) on \(\Omega_\infty\) with $\int_{\Omega_\infty}u_\infty^2\,dx=1.$
\end{lemma}

\begin{proof}
	Every compact subset of \(\Omega_\infty\) is contained in
	\(\Omega_j\) for all sufficiently large \(j\). Hence, using fixed
	functions in \(C_c^\infty(\Omega_\infty)\) in the Rayleigh quotient, we obtain that
	\[
	\limsup_{j\to\infty}\lambda_{1,p_j}(\Omega_j)
	\leq
	\lambda_{1,2}(\Omega_\infty).
	\]
	
	Choose a ball \(B_R\) containing all \(\overline{\Omega_j}\) in its
	interior and extend \(u_j\) by zero to \(B_R\). Since \(p_j>2\),
	H\"older's inequality gives
	\[
	\|\nabla u_j\|_{L^2(B_R)}
	\leq
	|B_R|^{\frac12-\frac1{p_j}}
	\lambda_{1,p_j}(\Omega_j)^{1/p_j},
	\qquad
	\|u_j\|_{L^2(B_R)}
	\leq
	|B_R|^{\frac12-\frac1{p_j}}.
	\]
	The preceding upper bound on the eigenvalues therefore yields a uniform
	\(H_0^1(B_R)\) bound. Up to a subsequence, $u_j \rightharpoonup u$ in $H_0^1(B_R)$, and $u_j \to u$ in $L^2(B_R)$ and almost everywhere in \(B_R\).
	
	Choose \(q>2\) such that $\sup_j\|u_j\|_{L^q(B_R)}<\infty.$
	For large \(j\), \(p_j<q\) with \(q-p_j\geq\delta>0\). For every
	\(M>1\),
	\[
	\int_{\{u_j>M\}}
	\bigl(u_j^{p_j}+u_j^2\bigr)\,dx
	\leq
	C M^{-\delta},
	\]
	whereas on \(\{u_j\leq M\}\), $\sup_{0\leq s\leq M}|s^{p_j}-s^2|
	\rightarrow0.$
	Hence, first letting \(j\to\infty\) and then \(M\to\infty\),
	\[
	\|u_j^{p_j}-u_j^2\|_{L^1(B_R)}
	\longrightarrow0.
	\]
	Since \(u_j\to u\) strongly in \(L^2(B_R)\), $u_j^2\rightarrow u^2$ in $L^1(B_R),$
	and therefore
	\[
	u_j^{p_j}\longrightarrow u^2
	\quad\text{in }L^1(B_R),
	\qquad
	\int_{\Omega_\infty}u^2\,dx=1.
	\]
	
	The function \(u\) vanishes almost everywhere outside \(K\). Since
	\(\Omega_\infty\) is a bounded convex domain, it is Lipschitz and
	\(|\partial\Omega_\infty|=0\); hence the zero-trace characterization
	gives $u\in H_0^1(\Omega_\infty).$ By weak lower semicontinuity and
	H\"older's inequality,
	\[
	\lambda_{1,2}(\Omega_\infty)\le\int_{\Omega_\infty}|\nabla u|^2\,dx
	\leq
	\liminf_{j\to\infty}\int_{B_R}|\nabla u_j|^2\,dx
	\leq
	\liminf_{j\to\infty}\lambda_{1,p_j}(\Omega_j).
	\]
	Together with the upper bound, this proves $\lambda_{1,p_j}(\Omega_j)
	\rightarrow
	\lambda_{1,2}(\Omega_\infty).$
	By simplicity, \(u=u_\infty\), and the asserted \(L^1\)-convergence follows
	from the estimates above.
\end{proof}

The preceding results show that every collapsing ground-state measure is approximated by squared ground states of linear Schr\"odinger operators with convex potentials. We now record the sharp Poincar\'e consequence in a form that will be applied to exact minimizers.

\begin{theorem}\label{thm:limit-measures}
	Let \(N\geq1\) and \(p_j\downarrow2\). Let \(\Omega_j\in\mathcal C_N\) be contained in a fixed bounded set and satisfy $\operatorname{diam}(\Omega_j)=1.$
	Let \(u_j\) be the positive first \(p_j\)-eigenfunction on \(\Omega_j\), normalized by $\int_{\Omega_j}u_j^{p_j}\,dx=1,$
	and define \(d\mu_j:=u_j^{p_j}\mathbf 1_{\Omega_j}\,dx\). Suppose that
	\(\mu_j\rightharpoonup\mu\) weakly as probability measures on \(\mathbb R^N\). Then
	\[
	P(\mu)
	\leq
	\frac{\operatorname{diam}(\operatorname{supp}\mu)^2}{3\pi^2}
	\leq
	\frac{1}{3\pi^2}.
	\]
\end{theorem}

\begin{proof}
	By the Blaschke selection theorem \cite[Theorem~1.8.5]{Schneider2014}, after passing to a subsequence, $\overline{\Omega_j}\rightarrow K$
	in the Hausdorff distance, where \(K\) is compact and convex. Since
	\(\operatorname{diam}(\Omega_j)=1\), $\operatorname{diam}(K)=1.$
	
	Assume first that \(K\) has nonempty interior $\Omega_\infty:=\operatorname{int}K.$
	By Lemma~\ref{lem:noncollapse}, $\mu=u_\infty^2\,dx.$
	By Theorem \ref{thm:ground-state-weighted-spectral-gap}(ii),
	\[
	P(\mu)=\frac{1}{\Lambda_2(\Omega_\infty)}
	\leq
	\frac{\operatorname{diam}(K)^2}{3\pi^2}
	=
	\frac1{3\pi^2}.
	\]
	
	Assume now that \(K\) has empty interior. Then $S_K:=\operatorname{aff}(K)$
	is a proper affine subspace of \(\mathbb R^N\). Write $S_K=x_0+E,$
	where \(E\subset\mathbb R^N\) is the linear subspace parallel to \(S_K\).
	Since \(\dim E<N\), choose a unit vector \(e\in E^\perp\). Define
	\[
	w_j
	:=
	\sup_{x\in\Omega_j}x\cdot e
	-
	\inf_{x\in\Omega_j}x\cdot e,
	\]
	the width of \(\Omega_j\) in the direction \(e\). Since \(K\subset S_K\), the quantity
	\(x\cdot e\) is constant on \(K\). Hence the Hausdorff convergence
	\(\overline{\Omega_j}\to K\) implies $w_j\rightarrow0.$
	
	Write $x=y+te,\,y\in e^\perp,\,t\in\mathbb R,$
	and, for \(y\in e^\perp\), define
	\[
	I_{j,y}
	:=
	\{t\in\mathbb R:y+te\in\Omega_j\}.
	\]
	By convexity, \(I_{j,y}\) is either empty or an open interval, and $|I_{j,y}|\leq w_j.$
	For almost every \(y\) with \(I_{j,y}\neq\varnothing\),
	the restriction \(t\mapsto u_j(y+te)\) belongs to
	\(W_0^{1,p_j}(I_{j,y})\). Therefore the one-dimensional Dirichlet
	\(p_j\)-Poincar\'e inequality in Lemma \ref{lem:first-eigenvalue-diameter-lower-bound} gives
	\[
	\int_{I_{j,y}}
	|\partial_e u_j(y+te)|^{p_j}\,dt
	\geq
	(p_j-1)
	\left(\frac{\pi_{p_j}}{|I_{j,y}|}\right)^{p_j}
	\int_{I_{j,y}}
	|u_j(y+te)|^{p_j}\,dt,
	\]
	and hence, since \(|I_{j,y}|\leq w_j\), integrating with respect to \(y\in e^\perp\), we obtain
	\[
	\lambda_{1,p_j}(\Omega_j)
	=
	\int_{\Omega_j}|\nabla u_j|^{p_j}\,dx
	\geq
	(p_j-1)
	\left(\frac{\pi_{p_j}}{w_j}\right)^{p_j},
	\]
	thus, $\lambda_{1,p_j}(\Omega_j)^{1/p_j}
	\rightarrow\infty.$ If \(\mu\) is a point mass, then $P(\mu)=0$. Otherwise the assumptions of
	Section~\ref{sec:local} are satisfied, by
	Proposition~\ref{prop:boundary},
	\[
	P(\mu)
	\leq
	\frac{\operatorname{diam}(\operatorname{supp}\mu)^2}{3\pi^2}
	\leq
	\frac1{3\pi^2}.
	\]
	Thus we complete the proof.
\end{proof}

\begin{corollary}
\label{cor:continuity-weighted-poincare}
For every \(N\geq1\),
\[
\lim_{p\downarrow2}C_{p,N}^{\mathrm{opt}}
=
C_{2,N}^{\mathrm{opt}}
=
\frac1{3\pi^2}.
\]
Moreover, $\lim_{p\downarrow2}\mathcal G_{p,N}
=
3\pi^2.$
\end{corollary}

\begin{proof}
We first prove $\limsup_{p\downarrow2}C_{p,N}^{\mathrm{opt}}
\leq
\frac1{3\pi^2}.$
Otherwise, there exist
\[
L>\frac1{3\pi^2},
\qquad
p_j\downarrow2,
\qquad
\Omega_j\in\mathcal C_N,
\]
such that, writing $D_j:=\operatorname{diam}(\Omega_j),$
we have
\[
\frac{
\lambda_{1,p_j}(\Omega_j)^{\frac{p_j-2}{p_j}}
}{
D_j^2\Lambda_{p_j}(\Omega_j)
}
>L.
\]
By scaling and translation, we may assume $D_j=1,\,
\Omega_j\subset B_R$
for a fixed ball \(B_R\). After passing to a subsequence,
\[
d\mu_j
=
u_j^{p_j}\mathbf1_{\Omega_j}\,dx
\rightharpoonup\mu
\]
weakly as probability measures. By
Lemma~\ref{lem:constant-limit} and
Theorem~\ref{thm:limit-measures},
\[
\frac{
\lambda_{1,p_j}(\Omega_j)^{\frac{p_j-2}{p_j}}
}{
\Lambda_{p_j}(\Omega_j)
}
\longrightarrow
P(\mu)
\leq
\frac1{3\pi^2},
\]
contradicting the choice of \(L\). 

\smallskip

We next prove the reverse inequality. Suppose first that \(N\geq2\). By the proof of Theorem~\ref{thm:existence-gap-minimizer}, we have $\limsup_{p\downarrow2}\mathcal G_{p,N}
\leq
3\pi^2.$
For \(N=1\), the same conclusion follows directly from the explicit
formula for \(\mathcal G_{p,1}\). On the other hand, Corollary~\ref{cor:enhanced-fundamental-gap} gives
\[
C_{p,N}^{\mathrm{opt}}
\geq
\frac{2}{p\,\mathcal G_{p,N}}
(p-1)^{\frac{p-2}{p}}
\pi_p^{p-2}.
\]
Since
\[
\frac2p
(p-1)^{\frac{p-2}{p}}
\pi_p^{p-2}
\longrightarrow1
\qquad\text{as }p\downarrow2,
\]
and $\limsup_{p\downarrow2}\mathcal G_{p,N}
\leq
3\pi^2,$
we obtain $\liminf_{p\downarrow2}C_{p,N}^{\mathrm{opt}}
\geq
\frac1{3\pi^2}.$
Together with the preceding upper bound, $\lim_{p\downarrow2}C_{p,N}^{\mathrm{opt}}
=
\frac1{3\pi^2}.$ Finally, returning to the enhanced gap estimate and using this limit,
we obtain $\liminf_{p\downarrow2}\mathcal G_{p,N}
\geq
3\pi^2,$
this yields $\lim_{p\downarrow2}\mathcal G_{p,N}
=
3\pi^2.$
\end{proof}

\subsection{Exact minimizers and the equality case}

For exact minimizers the upper bound in
Theorem~\ref{thm:limit-measures} is attained. We determine the equality
case. The proof is divided into three lemmas. The first identifies the
limiting density with the ground state of a Schr\"odinger operator and
passes the sharp two-point estimate to the limit. The second records the
one-dimensional quantitative gap estimate needed on the localization
needles. The third converts equality of the conditional one-dimensional
densities into a restriction on the dimension of the support.

\begin{lemma}
\label{lem:limit-ground-state-modulus}
Let \(\mu\) be a limit measure as in
Propositions~\ref{prop:potential} and \ref{prop:boundary}, and set
\[
S:=\operatorname{aff}(\operatorname{supp}\mu),
\qquad
A:=\operatorname{ri}(\operatorname{supp}\mu).
\]
Assume that \(\operatorname{diam}(\operatorname{supp}\mu)=1\).
Then there exist a finite convex potential \(V\) on \(A\) and a
positive function \(w\) such that $d\mu=w^2\,dx_S,$
\(w\) is the normalized ground state of the Friedrichs realization $H_A:=-\Delta_S+V,$
and
\begin{equation}
\label{eq:limit-ground-state-gap}
\frac1{P(\mu)}
=
\lambda_2(H_A)-\lambda_1(H_A).
\end{equation}
Moreover,
\begin{equation}
\label{eq:limit-sharp-modulus}
\left\langle
-\nabla_S\log w(y)+\nabla_S\log w(x),
\frac{y-x}{|y-x|}
\right\rangle
\geq
2\pi\tan\!\left(\frac{\pi|y-x|}{2}\right)
\end{equation}
for all distinct \(x,y\in A\).
\end{lemma}

\begin{proof}
By Propositions~\ref{prop:potential} and \ref{prop:boundary},
\[
d\mu=w^2\,dx_S,
\qquad
w>0\ \text{in }A,
\qquad
-\Delta_Sw+Vw=0
\quad\text{in }\mathcal D'(A),
\]
for a finite convex potential \(V\) on \(A\). Choosing the domains in
the construction of Proposition~\ref{prop:boundary} along a sequence
tending to zero, we may assume that \(U_n\Subset A\) are bounded convex
domains exhausting \(A\), with $\operatorname{diam}(U_n)\leq1,$
and that their normalized positive first eigenfunctions \(\varphi_n\)
satisfy
\[
\lambda_{1,n}
:=
\lambda_1(-\Delta_S+V,U_n)
=
e_\mu(U_n)
\longrightarrow0,
\qquad
\varphi_n\longrightarrow w
\quad\text{in }L^2(A).
\]

Since \(V\) is finite and convex, it is bounded below on the bounded
convex set \(A\). Extending \(\varphi_n\) by zero to \(A\), we therefore
have
\[
\int_A|\nabla_S\varphi_n|^2\,dx_S
\leq C+\lambda_{1,n}
\]
for some \(C>0\). Thus \((\varphi_n)\) is bounded in \(H_0^1(A)\).
After passing to a subsequence,
\[
\varphi_n\rightharpoonup w
\quad\text{weakly in }H_0^1(A),
\]
where the limit is identified with \(w\) by the strong \(L^2\)
convergence. Hence \(w\in H_0^1(A)\).

Let \(Q_A\) denote the closed quadratic form of the Friedrichs
realization \(H_A=-\Delta_S+V\):
\[
Q_A(\phi)
:=
\int_A
\left(
|\nabla_S\phi|^2+V\phi^2
\right)dx_S.
\]
For \(\phi\in C_c^\infty(A)\), the ground-state identity gives
\begin{equation}
\label{eq:limit-ground-state-transform}
Q_A(\phi)
=
\int_A
w^2
\left|
\nabla_S\!\left(\frac{\phi}{w}\right)
\right|^2dx_S
\geq0.
\end{equation}
By closure, \(Q_A\geq0\) on its form domain. Since \(V\geq-C\), the
closed form
\[
\phi\longmapsto Q_A(\phi)+C\|\phi\|_{L^2(A)}^2
\]
is lower semicontinuous. Therefore
\[
Q_A(w)+C\|w\|_2^2
\leq
\liminf_{n\to\infty}
\left(
Q_A(\varphi_n)+C\|\varphi_n\|_2^2
\right)
=
C.
\]
As \(\|w\|_2=1\), this gives \(Q_A(w)\leq0\). Together with
\(Q_A\geq0\), we obtain $Q_A(w)=0,\,
\lambda_1(H_A)=0.$
Thus \(w\) is the normalized positive ground state of \(H_A\).

Under the unitary correspondence
\[
f\longmapsto wf,
\qquad
L^2(A,d\mu)\longrightarrow L^2(A,dx_S),
\]
the identity \eqref{eq:limit-ground-state-transform} extends by closure
to
\[
Q_A(wf)-\lambda_1(H_A)\|wf\|_2^2
=
\int_A|\nabla_Sf|^2\,d\mu.
\]
Moreover,
\[
wf\perp w
\quad\Longleftrightarrow\quad
\int_Af\,d\mu=0.
\]
Since \(A\) is bounded, \(H_A\) has compact resolvent. Hence the
min--max characterization of its second eigenvalue yields
\[
\lambda_2(H_A)-\lambda_1(H_A)
=
\inf_{\substack{
f\in H^1(\mu)\setminus\{0\}\\
\int_Af\,d\mu=0}}
\frac{\displaystyle\int_A|\nabla_Sf|^2\,d\mu}
{\displaystyle\int_Af^2\,d\mu}
=
\frac1{P(\mu)},
\]
which proves \eqref{eq:limit-ground-state-gap}.

It remains to prove \eqref{eq:limit-sharp-modulus}. On every
\(K\Subset A\), for all sufficiently large \(n\),
\[
-\Delta_S\varphi_n
=
(\lambda_{1,n}-V)\varphi_n
\quad\text{in }K.
\]
Since \(V\in L^\infty_{\mathrm{loc}}(A)\), interior elliptic estimates
and a standard bootstrap, together with
\(\varphi_n\to w\) in \(L^2(A)\), give $\varphi_n\rightarrow w$ in $C^1_{\mathrm{loc}}(A).$
Writing $D_n:=\operatorname{diam}(U_n)\leq1,$
the sharp log-concavity estimate of Andrews--Clutterbuck
\cite[Theorem~1.5]{AndrewsClutterbuck2011} gives, for distinct
\(x,y\in U_n\),
\[
\left\langle
-\nabla_S\log\varphi_n(y)+\nabla_S\log\varphi_n(x),
\frac{y-x}{|y-x|}
\right\rangle
\geq
\frac{2\pi}{D_n}
\tan\!\left(
\frac{\pi|y-x|}{2D_n}
\right).
\]
Since \(D_n\leq1\),
\[
\frac{2\pi}{D_n}
\tan\!\left(
\frac{\pi r}{2D_n}
\right)
\geq
2\pi\tan\!\left(\frac{\pi r}{2}\right),
\qquad
0<r<D_n.
\]
For fixed distinct \(x,y\in A\), we have \(x,y\in U_n\) for all
sufficiently large \(n\). Letting \(n\to\infty\) therefore gives
\[
\left\langle
-\nabla_S\log w(y)+\nabla_S\log w(x),
\frac{y-x}{|y-x|}
\right\rangle
\geq
2\pi\tan\!\left(\frac{\pi|y-x|}{2}\right),
\]
which is \eqref{eq:limit-sharp-modulus}.
\end{proof}

\begin{lemma}
\label{lem:one-dimensional-quantitative-gap}
Let
\[
I=
\left(c-\frac{\ell}{2},c+\frac{\ell}{2}\right),
\qquad
0<\ell\leq1,
\]
and let \(q\) be a positive log-concave weight on \(I\). Define
\[
\Lambda(I,q)
:=
\inf_{\substack{
f\in H^1(I,q\,dt)\setminus\{0\}\\
\int_I f q\,dt=0}}
\frac{\displaystyle\int_I|f'|^2q\,dt}
{\displaystyle\int_I f^2q\,dt}.
\]
Assume that the monotone representative of $\psi_q:=-\frac12(\log q)'$
satisfies
\begin{equation}
\label{eq:needle-psi-condition}
\psi_q(t)-\psi_q(s)
\geq
2\pi\tan\!\left(\frac{\pi(t-s)}2\right),
\qquad
c-\frac{\ell}{2}<s<t<c+\frac{\ell}{2}.
\end{equation}
Then there exists an absolute constant \(c_*>0\) such that
\begin{equation}
\label{eq:one-dimensional-quantitative-gap}
\Lambda(I,q)
\geq
3\pi^2+c_*(1-\ell)^3.
\end{equation}
If \(\ell=1\), equality $\Lambda(I,q)=3\pi^2$
holds if and only if there exists $C>0$ such that
\begin{equation}
\label{eq:one-dimensional-cosine-weight}
q(t)
=
C\cos^2\!\bigl(\pi(t-c)\bigr).
\end{equation}
\end{lemma}

\begin{proof}
Set $d:=\pi\ell,\,
I_d:=\left(-\frac d2,\frac d2\right),$
and define
\[
\widetilde q(x)
:=
q\left(c+\frac{x}{\pi}\right),
\qquad
x\in I_d.
\]
Then
\[
\widetilde\psi(x)
:=
-\frac12(\log\widetilde q)'(x)
=
\frac1\pi
\psi_q\left(c+\frac{x}{\pi}\right).
\]
By \eqref{eq:needle-psi-condition},
\[
\widetilde\psi(y)-\widetilde\psi(x)
\geq
2\tan\!\left(\frac{y-x}{2}\right),
\qquad
x<y,
\]
so that \(\widetilde\psi\) belongs to the class
\(\mathcal A(I_\pi)\) introduced in
\cite{AmatoBucurFragala2024}, with
\(\operatorname{dom}(\widetilde\psi)=I_d\).

Let \(\mu_1(I_d,\widetilde q)\) denote the first nonzero weighted
Neumann eigenvalue associated with \(\widetilde q\). The change of
variables \(x=\pi(t-c)\) gives $\Lambda(I,q)
=
\pi^2\mu_1(I_d,\widetilde q).$
By \cite[Lemma~6]{AmatoBucurFragala2024},
\[
\mu_1(I_d,\widetilde q)
\geq
\lambda_1
\left(
I_d,
\widetilde\psi'
+
\widetilde\psi^2
\right),
\]
where the derivative is understood in the distributional sense. The
one-dimensional quantitative estimate
\cite[Theorem~20(i)]{AmatoBucurFragala2024} yields an absolute
constant \(C_0>0\) such that
\[
\lambda_1
\left(
I_d,
\widetilde\psi'
+
\widetilde\psi^2
\right)
\geq
3+C_0(\pi-d)^3.
\]
Since \(d=\pi\ell\), we obtain
\[
\Lambda(I,q)
\geq
3\pi^2
+
C_0\pi^5(1-\ell)^3.
\]
Thus \eqref{eq:one-dimensional-quantitative-gap} holds with $c_*:=C_0\pi^5.$

Assume now that \(\ell=1\) and
\(\Lambda(I,q)=3\pi^2\). Then \(d=\pi\), and $\mu_1(I_\pi,\widetilde q)=3.$
The equality case
\cite[Corollary~11]{AmatoBucurFragala2024} yields $\widetilde q(x)=C\cos^2x$
for some \(C>0\). Returning to the variable \(t\) gives
\eqref{eq:one-dimensional-cosine-weight}. Conversely, the weight in
\eqref{eq:one-dimensional-cosine-weight} has first nonzero weighted
Neumann eigenvalue \(3\pi^2\).
\end{proof}

\begin{lemma}
\label{lem:full-length-ray-cluster}
Let \(S\simeq\mathbb R^m\) with \(m\geq2\), and let
\(K\subset S\) be a compact convex set with $\operatorname{diam}(K)=1.$
Consider a ray cluster associated with a \(1\)-Lipschitz function in
the sense of Klartag, written in the canonical ray-cluster coordinates as
\[
F(y,t)=x(y)+t e(y),
\qquad
|e(y)|=1,
\]
where \(y\) ranges over a measurable subset of
\(\mathbb R^{m-1}\), and \(t=0\) is the distinguished transverse
section of the cluster. Let \(E\) be a measurable set of ray
parameters. Assume that for almost every \(y\in E\):

\begin{enumerate}
\item[\rm(i)]
the closure of the corresponding ray is a segment of length \(1\),
which can be written as
\[
L_y
=
\left\{
c(y)+s e(y):
-\frac12\leq s\leq\frac12
\right\}
\subset K;
\]

\item[\rm(ii)]
the conditional Lebesgue density \(h_y\) along the ray satisfies $(\log h_y)''=0$
on the interior of \(L_y\), in the distributional sense.
\end{enumerate}

Then
\begin{equation}
\label{eq:full-ray-zero-volume}
\mathcal H^m
\left(
\bigcup_{y\in E}L_y
\right)
=
0.
\end{equation}
\end{lemma}

\begin{proof}
We work at the regular density points of the transverse parameter set
for which the ray-cluster differential identities hold; these form a
set of full \((m-1)\)-dimensional measure. Fix such a parameter \(y\).
In the canonical ray-cluster coordinates define
\[
J_i(t):=\partial_{y_i}F(y,t),
\qquad
i=1,\ldots,m-1.
\]
The parameter \(t=0\) is the transverse section appearing in the
ray-cluster construction and is kept fixed throughout the argument.

Since \(S\) is Euclidean, $J_i''(t)=0.$
At \(t=0\), the vectors \(J_i(0)\) and \(J_i'(0)\) are orthogonal to
\(e(y)\); see \cite[Proposition~3.5]{Klartag2017}. Since the
conditional density is positive on the interior of the ray, the
transverse Jacobian is nonzero at \(t=0\). Thus
\[
J_1(0),\ldots,J_{m-1}(0)
\]
span \(e(y)^\perp\). After a linear change of transverse coordinates
at \(y\), we may assume that they form an orthonormal basis of
\(e(y)^\perp\).

The symmetry relation
\[
\langle J_i'(0),J_j(0)\rangle
=
\langle J_j'(0),J_i(0)\rangle
\]
from \cite[Proposition~3.5(iii)]{Klartag2017} yields a symmetric matrix
\(B_y=B_y^{\mathsf T}\) such that
\[
J_i'(0)
=
\sum_{j=1}^{m-1}
(B_y)_{ji}J_j(0).
\]
Since \(J_i''=0\),
\[
J_i(t)
=
\sum_{j=1}^{m-1}
\bigl(\delta_{ji}+t(B_y)_{ji}\bigr)J_j(0).
\]
Equivalently, the transverse differential at time \(t\) is represented
by \(I+tB_y\).

By the Jacobian formula for ray clusters
\cite[Proposition~3.5(iv), Lemma~3.18]{Klartag2017}, the conditional
Lebesgue density along this ray differs by a positive factor depending
only on \(y\) from the transverse Jacobian. Hence
\[
h_y(t)
=
C_y\left|\det(I+tB_y)\right|,
\qquad
C_y>0.
\]
Since \(h_y>0\) on the interior of the ray and
\(\det(I)=1\), the determinant cannot vanish there and has positive
sign. Therefore
\[
h_y(t)
=
C_y\det(I+tB_y)
\]
throughout the interior of the ray. It follows that
\begin{equation}
\label{eq:ray-log-density-second}
(\log h_y)''(t)
=
-
\operatorname{tr}
\left[
\bigl((I+tB_y)^{-1}B_y\bigr)^2
\right].
\end{equation}
The right-hand side is smooth. By assumption~\rm(ii),
\((\log h_y)''=0\) distributionally, hence pointwise. Evaluating
\eqref{eq:ray-log-density-second} at \(t=0\) gives $0
=
-\operatorname{tr}(B_y^2)
=
-|B_y|_{\mathrm{HS}}^2,$
where the last equality uses the symmetry of \(B_y\). Consequently, $B_y=0.$
Since $J_i'(0)=\partial_{y_i}e(y),$
we obtain $D_y e(y)=0$
for almost every \(y\in E\).

We now use the diameter constraint. For \(y,z\in E\), set $d:=c(y)-c(z).$
The four endpoints of \(L_y\) and \(L_z\) belong to \(K\). In
particular, the two pairs of opposite endpoints satisfy
\[
\left|
d+\frac{e(y)+e(z)}2
\right|
\leq1,
\qquad
\left|
d-\frac{e(y)+e(z)}2
\right|
\leq1.
\]
Squaring and adding gives
\[
2|d|^2
+
\frac12|e(y)+e(z)|^2
\leq2.
\]
Since \(|e(y)|=|e(z)|=1\), $|e(y)+e(z)|^2+|e(y)-e(z)|^2=4.$
Therefore
\begin{equation}
\label{eq:midpoint-direction}
|c(y)-c(z)|
\leq
\frac12|e(y)-e(z)|.
\end{equation}

The direction field \(e\) is locally Lipschitz on a ray cluster.
Hence \eqref{eq:midpoint-direction} implies that \(c|_E\) is locally
Lipschitz as well. At almost every density point \(y\in E\),
$D_y e(y)=0$ implies that the approximate differential
of \(e|_E\) vanishes. Dividing
\eqref{eq:midpoint-direction} by \(|z-y|\) and letting \(z\to y\)
through density points of \(E\), we obtain $D_Ec(y)=0$
for almost every \(y\in E\).

Consider the midpoint parametrization
\[
\widetilde F:
E\times\left(-\frac12,\frac12\right)
\longrightarrow S,
\qquad
\widetilde F(y,s)
=
c(y)+s e(y).
\]
By the local Lipschitz regularity of \(c|_E\) and \(e|_E\),
\(\widetilde F\) is locally Lipschitz on its measurable domain. At
almost every \((y,s)\),
\[
D_y\widetilde F(y,s)
=
D_Ec(y)+sD_Ee(y)
=
0,
\]
whereas $\partial_s\widetilde F(y,s)=e(y).$
Hence, $\operatorname{rank}D\widetilde F(y,s)\leq1$
almost everywhere. Since \(m\geq2\), its \(m\)-dimensional Jacobian
vanishes:
\[
J_m\widetilde F(y,s)=0
\qquad\text{for a.e. }(y,s).
\]
The area formula for Lipschitz maps on measurable subsets of
\(\mathbb R^m\) therefore yields
\[
\mathcal H^m
\left(
\widetilde F
\left(
E\times\left(-\frac12,\frac12\right)
\right)
\right)
=
0.
\]

Finally, each of the two endpoint sets
\[
\left\{
c(y)\pm\frac12e(y):y\in E
\right\}
\]
is a locally Lipschitz image of a subset of
\(\mathbb R^{m-1}\), and hence has zero
\(\mathcal H^m\)-measure. Since
\[
\bigcup_{y\in E}L_y
=
\widetilde F
\left(
E\times\left(-\frac12,\frac12\right)
\right)
\cup
\left\{
c(y)\pm\frac12e(y):y\in E
\right\},
\]
we conclude that
\[
\mathcal H^m
\left(
\bigcup_{y\in E}L_y
\right)
=
0,
\]
which completes the proof.
\end{proof}

We now combine the preceding three lemmas to show that equality in Theorem~\ref{thm:limit-measures} forces the limit measure to be one-dimensional and uniquely determines its density.

\begin{proposition}
\label{prop:one-dimensional-limit-measure}
Let \(p_j\downarrow2\), let
\(\Omega_j\in\mathcal C_N\) be contained in a fixed bounded set with
\(\operatorname{diam}(\Omega_j)=1\), and let $d\mu_j
=
u_j^{p_j}\mathbf 1_{\Omega_j}\,dx
\rightharpoonup\mu,$
where \(u_j\) is the normalized positive first
\(p_j\)-eigenfunction. Assume
\[
\lambda_{1,p_j}(\Omega_j)^{1/p_j}
\longrightarrow+\infty,
\qquad
P(\mu)=\frac1{3\pi^2},
\qquad
\operatorname{diam}(\operatorname{supp}\mu)=1.
\]
Then $\dim\operatorname{aff}(\operatorname{supp}\mu)=1.$
Consequently, there exist
\(a\in\mathbb R^N\) and \(e\in\mathbb S^{N-1}\) such that
\[
\operatorname{supp}\mu
=
\left\{
a+te:
|t|\leq\frac12
\right\},
\]
and
\begin{equation}
\label{eq:cosine-limit-measure}
\int_{\mathbb R^N}f\,d\mu
=
2\int_{-1/2}^{1/2}
f(a+te)\cos^2(\pi t)\,dt,
\qquad
f\in C_b(\mathbb R^N).
\end{equation}
\end{proposition}

\begin{proof}
Set
\[
S:=\operatorname{aff}(\operatorname{supp}\mu),
\qquad
m:=\dim S,
\qquad
A:=\operatorname{ri}(\operatorname{supp}\mu).
\]
Since
\(\operatorname{diam}(\operatorname{supp}\mu)=1\), we have \(m\geq1\).
By Lemma~\ref{lem:limit-ground-state-modulus},
\[
d\mu=\rho\,dx_S,
\qquad
\rho=w^2>0
\quad\text{in }A,\qquad \frac1{P(\mu)}=3\pi^2.
\]

\smallskip
\noindent
\textbf{Step 1: Localization of the equality case.}
Assume by contradiction that \(m\geq2\).
By Lemma~\ref{lem:limit-ground-state-modulus}, the ground-state
transform identifies the weighted Dirichlet form associated with
\(d\mu\) with the shifted Schr\"odinger operator
\(H_A-\lambda_1(H_A)\). Since \(H_A\) has compact resolvent, the
weighted Poincar\'e quotient attains its first nonzero eigenvalue.
Let \(g\) be a corresponding eigenfunction, normalized by
\begin{equation}
\label{eq:g-normalization}
\int_A g\,d\mu=0,
\qquad
\int_A g^2\,d\mu=1.
\end{equation}
Then $\int_A|\nabla_S g|^2\,d\mu
=
3\pi^2.$

We apply Klartag's localization theorem with a guiding function
\cite[Theorem~1.5]{Klartag2017} to $f:=g\rho$
with respect to Lebesgue measure \(dx_S\) on \(A\). This function is
integrable because \(A\) is bounded and \(g\in L^2(\mu)\).
Hence there exist a \(1\)-Lipschitz function \(u\), a quotient measure
\(\eta\), and unit-speed parametrizations
\[
\gamma_\xi:I_\xi\longrightarrow A,
\qquad \xi\in\Xi,
\]
of the nondegenerate transport rays such that
\begin{equation}
\label{eq:strain-disintegration}
\int_{\operatorname{Strain}[u]}\Phi\,dx_S
=
\int_\Xi\int_{I_\xi}
\Phi(\gamma_\xi(t))h_\xi(t)\,dt\,d\eta(\xi)
\end{equation}
for every nonnegative Borel function \(\Phi\). Here \(h_\xi\) is
positive and log-concave on \(I_\xi\). Moreover,
\begin{equation}
\label{eq:localization-zero-mean}
g\rho=0
\quad\text{a.e. on }
A\setminus\operatorname{Strain}[u],
\end{equation}
and
\begin{equation}
\label{eq:needle-zero-mean}
\int_{I_\xi}
(g\circ\gamma_\xi)(t)
\rho(\gamma_\xi(t))h_\xi(t)\,dt
=
0
\end{equation}
for \(\eta\)-almost every \(\xi\). Set
\[
q_\xi(t)
:=
\rho(\gamma_\xi(t))h_\xi(t),
\qquad
M_\xi
:=
\int_{I_\xi}
(g\circ\gamma_\xi)^2q_\xi\,dt.
\]
Since \(\rho>0\) in \(A\), \eqref{eq:localization-zero-mean} gives $g=0$ a.e. on $A\setminus\operatorname{Strain}[u].$
Moreover \(g\in H^1_{\mathrm{loc}}(A)\), and hence the Sobolev
level-set property yields
\[
\nabla_Sg=0
\quad\text{a.e. on }
A\setminus\operatorname{Strain}[u].
\]
Using \eqref{eq:g-normalization} and
\eqref{eq:strain-disintegration}, we obtain
\begin{equation}
\label{eq:needle-mass-normalization}
\int_\Xi M_\xi\,d\eta(\xi)=1.
\end{equation}

\smallskip
\noindent
\textbf{Step 2: Every effective ray has full length and cosine density.}
We verify the hypotheses of
Lemma~\ref{lem:one-dimensional-quantitative-gap}.
Let \(s<t\) be points of differentiability on a fixed ray. Since
\(\gamma_\xi\) has unit speed, the estimate
\eqref{eq:limit-sharp-modulus} gives
\[
(\log\rho\circ\gamma_\xi)'(t)
-
(\log\rho\circ\gamma_\xi)'(s)
\leq
-4\pi
\tan\!\left(
\frac{\pi(t-s)}2
\right).
\]
In particular, \(\rho\circ\gamma_\xi\) is log-concave. Since
\(h_\xi\) is log-concave, \(q_\xi\) is positive and log-concave.
Moreover,
\[
(\log q_\xi)'(t)-(\log q_\xi)'(s)
\leq
-4\pi
\tan\!\left(
\frac{\pi(t-s)}2
\right).
\]
Thus the monotone representative of $\psi_{q_\xi}:=-\frac12(\log q_\xi)'$
satisfies
\[
\psi_{q_\xi}(t)-\psi_{q_\xi}(s)
\geq
2\pi
\tan\!\left(
\frac{\pi(t-s)}2
\right).
\]
Since the rays are contained in
\(\operatorname{supp}\mu\) and are parametrized by arclength, $|I_\xi|
\leq
\operatorname{diam}(\operatorname{supp}\mu)
=
1.$
Therefore Lemma~\ref{lem:one-dimensional-quantitative-gap} gives
\[\Lambda(I_\xi,q_\xi)
\geq
3\pi^2+c_*(1-|I_\xi|)^3.\]

By \eqref{eq:localization-zero-mean},
\eqref{eq:strain-disintegration}, and the vanishing of
\(\nabla_Sg\) outside the transport set,
\[
\int_A|\nabla_Sg|^2\,d\mu
=
\int_\Xi\int_{I_\xi}
|\nabla_Sg|^2(\gamma_\xi(t))q_\xi(t)\,dt\,d\eta(\xi).
\]
Since $|(g\circ\gamma_\xi)'|
\leq
|\nabla_Sg|\circ\gamma_\xi$
almost everywhere, \eqref{eq:needle-zero-mean} and the variational
definition of \(\Lambda(I_\xi,q_\xi)\) yield
\begin{align*}
3\pi^2
&=
\int_A|\nabla_Sg|^2\,d\mu
\nonumber\geq
\int_\Xi\int_{I_\xi}
|(g\circ\gamma_\xi)'|^2q_\xi\,dt\,d\eta(\xi)
\nonumber\\
&\geq
\int_\Xi
\Lambda(I_\xi,q_\xi)M_\xi\,d\eta(\xi)
\nonumber\geq
3\pi^2
+
c_*
\int_\Xi
(1-|I_\xi|)^3M_\xi\,d\eta(\xi),
\end{align*}
where the last line uses
\eqref{eq:needle-mass-normalization}. Consequently,
\[
\int_\Xi
(1-|I_\xi|)^3M_\xi\,d\eta(\xi)
=
0.
\]
Hence $|I_\xi|=1$
for \(\eta\)-almost every \(\xi\) such that \(M_\xi>0\). We call such
rays effective.

Furthermore,
\[
3\pi^2
\geq
\int_\Xi
\Lambda(I_\xi,q_\xi)M_\xi\,d\eta(\xi)
\geq
3\pi^2\int_\Xi M_\xi\,d\eta(\xi)
=
3\pi^2.
\]
Thus
\[
\int_\Xi
\bigl(\Lambda(I_\xi,q_\xi)-3\pi^2\bigr)
M_\xi\,d\eta(\xi)
=
0,
\]
and therefore $\Lambda(I_\xi,q_\xi)=3\pi^2$
for almost every effective ray. After translating its arclength
parameter, such a ray may be written as
\[
I_\xi=\left(-\frac12,\frac12\right).
\]
The equality case in
Lemma~\ref{lem:one-dimensional-quantitative-gap} then gives
\begin{equation}
\label{eq:effective-ray-cosine}
q_\xi(t)
=
C_\xi\cos^2(\pi t),
\qquad
-\frac12<t<\frac12,
\end{equation}
for some \(C_\xi>0\).

\smallskip
\noindent
\textbf{Step 3: Rigidity of the effective rays and exclusion of
\(m\geq2\).}
We first show that the conditional Lebesgue density is exponential
affine on almost every effective ray. Write such a ray as
\[
\gamma_\xi(t)=z_\xi+t e_\xi,
\qquad
-\frac12<t<\frac12,
\]
where \(z_\xi\) is its midpoint, and set
\[
a(t):=\log\rho(\gamma_\xi(t)),
\qquad
b(t):=\log h_\xi(t).
\]
By \eqref{eq:effective-ray-cosine}, $a(t)+b(t)
=
\log C_\xi+2\log\cos(\pi t).$
Hence, for almost every \(s\in(0,1/2)\),
\begin{equation}
\label{eq:sum-log-derivatives}
[a'(s)-a'(-s)]
+
[b'(s)-b'(-s)]
=
-4\pi\tan(\pi s).
\end{equation}
On the other hand, applying
\eqref{eq:limit-sharp-modulus} to
\(\gamma_\xi(-s)\) and \(\gamma_\xi(s)\) gives
\begin{equation}
\label{eq:rho-ray-concavity}
a'(s)-a'(-s)
\leq
-4\pi\tan(\pi s),
\end{equation}
while the concavity of \(b\) gives
\begin{equation}
\label{eq:h-ray-concavity}
b'(s)-b'(-s)\leq0.
\end{equation}
Comparing
\eqref{eq:sum-log-derivatives},
\eqref{eq:rho-ray-concavity}, and
\eqref{eq:h-ray-concavity}, we obtain
\[
a'(s)-a'(-s)
=
-4\pi\tan(\pi s),
\qquad
b'(s)=b'(-s)
\]
for almost every \(s\in(0,1/2)\).
Since \(b\) is concave, its derivative has a nonincreasing
representative. The identity
\(b'(s)=b'(-s)\) for almost every \(s\) therefore forces this
representative to be constant almost everywhere on
\((-1/2,1/2)\). Hence
\begin{equation}
\label{eq:effective-log-h-affine}
(\log h_\xi)''=0
\end{equation}
in the distributional sense on almost every effective ray.

By the ray-cluster decomposition
\cite[Proposition~3.9]{Klartag2017}, the transport set is covered, up
to a null set, by countably many ray clusters. Within each cluster,
the conditional density in
\eqref{eq:strain-disintegration} differs from the Jacobian density of
the ray-cluster parametrization by a positive factor depending only on
the ray parameter. Thus
\eqref{eq:effective-log-h-affine} also holds for the corresponding
ray-cluster Jacobian density.

By \eqref{eq:needle-mass-normalization}, the effective rays carry
positive \(g^2\,d\mu\)-mass. Hence one of the countably many ray
clusters contains a measurable family \(E\) of effective rays whose
union
\[
\mathcal U_E
:=
\bigcup_{\xi\in E}L_\xi,
\qquad
L_\xi
:=
\overline{\gamma_\xi(I_\xi)},
\]
has positive \(g^2\,d\mu\)-mass. In particular, $\mathcal H^m(\mathcal U_E)>0,$
because \(d\mu=\rho\,dx_S\).

On the other hand, $K:=\operatorname{supp}\mu$
is compact and convex, with
\(\operatorname{diam}(K)=1\). Every \(L_\xi\), \(\xi\in E\), is a
segment of length \(1\), and its ray-cluster density satisfies
\eqref{eq:effective-log-h-affine}. Lemma~\ref{lem:full-length-ray-cluster}
therefore yields $\mathcal H^m(\mathcal U_E)=0,$
a contradiction. Hence \(m\geq2\) is impossible. Since
\(\operatorname{diam}(\operatorname{supp}\mu)=1\), \(m\neq0\), and
therefore $m=1.$

\smallskip
\noindent
\textbf{Step 4: Identification of the one-dimensional limit measure.}
Since \(\operatorname{supp}\mu\) is compact and convex, has affine
dimension \(1\), and has diameter \(1\), there exist
\(a\in\mathbb R^N\) and \(e\in\mathbb S^{N-1}\) such that
\[
\operatorname{supp}\mu
=
\left\{
a+te:
-\frac12\leq t\leq\frac12
\right\}.
\]
Identifying \(a+te\) with \(t\), write
\[
d\mu=\rho(t)\,dt,
\qquad
-\frac12<t<\frac12.
\]
By the definition of \(P(\mu)\),
\[
\Lambda
\left(
\left(-\frac12,\frac12\right),\rho
\right)
=
3\pi^2.
\]
Moreover, \eqref{eq:limit-sharp-modulus} gives, for
\(-1/2<s<t<1/2\),
\[
(\log\rho)'(t)-(\log\rho)'(s)
\leq
-4\pi
\tan\!\left(
\frac{\pi(t-s)}2
\right).
\]
Thus \(\rho\) satisfies the hypotheses of
Lemma~\ref{lem:one-dimensional-quantitative-gap} with \(\ell=1\).
Its equality case yields $\rho(t)=C\cos^2(\pi t)$
for some \(C>0\). Since \(\mu\) is a probability measure,
\[
1
=
C\int_{-1/2}^{1/2}\cos^2(\pi t)\,dt
=
\frac C2,
\]
and hence \(C=2\). Therefore
\[
d\mu(t)
=
2\cos^2(\pi t)\,dt,
\qquad
-\frac12<t<\frac12.
\]
Equivalently,
\[
\int_{\mathbb R^N}f\,d\mu
=
2\int_{-1/2}^{1/2}
f(a+te)\cos^2(\pi t)\,dt,
\qquad
f\in C_b(\mathbb R^N),
\]
which is \eqref{eq:cosine-limit-measure}.
\end{proof}

We can now prove the complete degeneration statement for exact minimizers. The first step identifies the sharp limiting Poincar\'e constant; the second excludes a nondegenerate Hausdorff limit, and the equality-case proposition then determines the limiting measure.

\begin{proof}[\textbf{Proof of Theorem~\ref{thm:limits-gap-minimizers}.}]
After translating one point of \(\overline{\Omega_j}\) to the origin,
Blaschke compactness gives, up to a subsequence,
\[
\overline{\Omega_j}\longrightarrow K
\quad\text{in the Hausdorff distance},
\qquad
\mu_j\rightharpoonup\mu,
\]
where \(K\) is compact and convex, $\operatorname{diam}(K)=1,\,
\operatorname{supp}\mu\subset K.$ Set
\[
\lambda_j:=\lambda_{1,p_j}(\Omega_j,0),
\qquad
\alpha_j:=\frac{p_j-2}{p_j},
\qquad
Q_j:=\frac{\lambda_j^{\alpha_j}}{\Lambda_{p_j}(\Omega_j)}.
\]
By Lemma~\ref{lem:constant-limit}, $Q_j\rightarrow P(\mu).$
Since \(\Omega_j\) is a diameter-normalized minimizer, $\Gamma_{p_j}(\Omega_j,0)=\mathcal G_{p_j,N}.$
By Corollary \ref{cor:continuity-weighted-poincare}, $\mathcal G_{p_j,N}\longrightarrow3\pi^2.$
Moreover, Theorem~\ref{thm:dimension-free-higher-dimensional-gap}\textnormal{(ii)} and
Lemma~\ref{lem:first-eigenvalue-diameter-lower-bound} give
\[
\Lambda_{p_j}(\Omega_j)
\leq
\frac{p_j}{2}\Gamma_{p_j}(\Omega_j,0),
\qquad
\liminf_{j\to\infty}\lambda_j^{\alpha_j}\geq1.
\]
Therefore, $P(\mu)
=
\lim_{j\to\infty}Q_j
\geq
\frac1{3\pi^2}.$
On the other hand, Theorem~\ref{thm:limit-measures} yields
\[
P(\mu)
\leq
\frac{\operatorname{diam}(\operatorname{supp}\mu)^2}{3\pi^2}
\leq
\frac1{3\pi^2}.
\]
Hence
\begin{equation}\label{eq:minimizer-limit-P}
P(\mu)=\frac1{3\pi^2},
\qquad
\operatorname{diam}(\operatorname{supp}\mu)=1.
\end{equation}

We next prove that \(K\) has empty interior. Suppose instead that
\(\Omega_\infty:=\operatorname{int}K\neq\varnothing\). By
Lemma~\ref{lem:noncollapse}, $\mu=u_\infty^2\,dx,$
where \(u_\infty\) is the normalized first Dirichlet eigenfunction of
\(-\Delta\) on \(\Omega_\infty\). The linear ground-state transform gives $P(\mu)
=
\frac1{\Gamma_2(\Omega_\infty,0)}.$
Since \(\Omega_\infty\) is a nondegenerate \(N\)-dimensional convex domain of diameter \(1\), the strict fundamental-gap estimate of
\cite[Theorem~1]{AmatoBucurFragala2024} gives $\Gamma_2(\Omega_\infty,0)>3\pi^2,$
contradicting \eqref{eq:minimizer-limit-P}. Thus, $\operatorname{int}_{\mathbb R^N}K=\varnothing.$

By the proof of Theorem \ref{thm:limit-measures}, we obtain $\lambda_j^{1/p_j}\rightarrow+\infty.$
The hypotheses of Proposition~\ref{prop:one-dimensional-limit-measure} are now satisfied by \eqref{eq:minimizer-limit-P}. Hence $\dim\operatorname{aff}(\operatorname{supp}\mu)=1,$
and there exist \(a\in\mathbb R^N\) and \(e_0\in\mathbb S^{N-1}\) such that
\[
\operatorname{supp}\mu
=
\{a+te_0:|t|\leq\tfrac12\},
\]
with
\[
\int_{\mathbb R^N}f\,d\mu
=
2\int_{-1/2}^{1/2}f(a+te_0)\cos^2(\pi t)\,dt,
\qquad
f\in C_b(\mathbb R^N).
\]

If \(N=2\), a compact convex set of diameter \(1\) with empty interior is a line segment of length \(1\). Since
\(\operatorname{supp}\mu\subset K\) and
\(\operatorname{diam}(\operatorname{supp}\mu)=1\), it follows that $K=\operatorname{supp}\mu.$
This proves the theorem.
\end{proof}

The same collapse argument also gives the eigenvalue limits quoted in the introduction.

\begin{corollary}\label{cor:minimizer-spectral-collapse}
Under the assumptions of Theorem~\ref{thm:limits-gap-minimizers},
\[
\lambda_{k,p_j}(\Omega_j,0)\longrightarrow+\infty,
\qquad k=1,2,
\]
while
\[
\Gamma_{p_j}(\Omega_j,0)\longrightarrow3\pi^2,
\qquad
\frac{\lambda_{2,p_j}(\Omega_j,0)}{\lambda_{1,p_j}(\Omega_j,0)}\longrightarrow1.
\]
\end{corollary}

\begin{proof}
The proof of Theorem~\ref{thm:limits-gap-minimizers} gives
\(\lambda_{1,p_j}(\Omega_j,0)\to+\infty\), while
Theorem~\ref{thm:optimal-gap-asymptotics}\textnormal{(i)} gives
\[
\Gamma_{p_j}(\Omega_j,0)
=
\mathcal G_{p_j,N}
\longrightarrow3\pi^2.
\]
Since
\[
\lambda_{2,p_j}(\Omega_j,0)
=
\lambda_{1,p_j}(\Omega_j,0)+\Gamma_{p_j}(\Omega_j,0),
\]
the remaining assertions follow immediately.
\end{proof}

\section{Log-concavity of the First Eigenfunction}
\label{sec:logconca}

In this section, we prove the log-concavity of the positive first eigenfunction. The regularity and boundary properties of the first eigenfunction, together with the construction, regularity, boundary nondegeneracy, and convergence of the uniformly elliptic regularized eigenfunctions, are collected in Appendix~\ref{app:logconca-technical}. We then prove uniform boundary concavity and apply a two-point maximum principle to the logarithms of the regularized eigenfunctions. Passing to the limit yields the log-concavity of the original positive first eigenfunction.

 Throughout this section and Appendix~\ref{app:logconca-technical}, let \(1<p<\infty\) and \(N\geq2\). We first
assume that \(\Omega\subset\mathbb R^N\) is a bounded convex domain with
\(\partial\Omega\in C^{2,\alpha_\Omega}\) for some
\(\alpha_\Omega\in(0,1)\), and that
\(V\in C^{0,\alpha_V}(\overline\Omega)\) for some
\(\alpha_V\in(0,1)\) is convex, namely,
\[
V\bigl((1-t)x+ty\bigr)
\leq
(1-t)V(x)+tV(y)
\]
for every \(x,y\in\overline\Omega\) and \(t\in[0,1]\).
These additional regularity assumptions are used only in the
regularization and boundary arguments below; the final results will be
extended to arbitrary bounded convex domains and continuous convex
potentials by approximation.

The eigenvalue problem \eqref{eq:introduction-eigenvalue-problem} is
understood in the weak sense. A number \(\lambda\in\mathbb R\) is called
an eigenvalue if there exists a nontrivial function
\(u\in W_0^{1,p}(\Omega)\) such that
\begin{equation}\label{eq:weighted-p-laplacian-eigenvalue}
\int_\Omega
|\nabla u|^{p-2}\nabla u\cdot\nabla\varphi\,dx
+
\int_\Omega
V(x)|u|^{p-2}u\varphi\,dx
=
\lambda
\int_\Omega
|u|^{p-2}u\varphi\,dx
\end{equation}
for every \(\varphi\in W_0^{1,p}(\Omega)\). Such nontrivial
function \(u\) is called an eigenfunction associated with \(\lambda\). 

The standard Ljusternik--Schnirelmann variational construction for the
\(p\)-Laplacian yields that each
\(\lambda_{k,p}(\Omega,V)\) defined in \eqref{eq:eigenvalue} is an eigenvalue, and any corresponding
nontrivial critical point is called an eigenfunction associated with
\(\lambda_{k,p}(\Omega,V)\), see \cite{Amann1972,GarciaAzoreroPeral1987,Szulkin1988,
CuestaDeFigueiredoGossez1999}.

It follows immediately from
\eqref{eq:weighted-p-laplacian-eigenvalue} that adding a constant
\(c\in\mathbb R\) to \(V\) shifts the first eigenvalue by \(c\) without
changing the \(L^p\)-normalized first eigenfunctions. Therefore, upon
replacing \(V\) by \(V-\min_{\overline\Omega}V\), we may assume
throughout that \(V\geq0\) in \(\overline\Omega\).

Recall that
\[
\mathcal E_V(w)
=
\int_\Omega|\nabla w|^p\,dx
+
\int_\Omega V(x)|w|^p\,dx,
\qquad
w\in W_0^{1,p}(\Omega),
\]
and
\[
\mathcal K
=
\left\{
w\in W_0^{1,p}(\Omega):
\|w\|_{L^p(\Omega)}=1
\right\}.
\]
The first Dirichlet eigenvalue of
\eqref{eq:weighted-p-laplacian-eigenvalue} is characterized
variationally by
\cite[Theorem 1]{otani-teshima} or \cite[Theorem 2.7]{del-pezzo-fernandez-bonder}
\begin{equation}\label{eq:first-eigenvalue-variational}
    \lambda_{1,p}(\Omega,V)
:=
\inf_{w\in\mathcal K}\mathcal E_V(w).
\end{equation}

The infimum in \eqref{eq:first-eigenvalue-variational} is attained by a
nonnegative function \(u_1\in\mathcal K\), which is an eigenfunction
associated with \(\lambda_{1,p}(\Omega,V)\); see
\cite[Theorem~2.7]{del-pezzo-fernandez-bonder}. By the strong maximum
principle, \(u_1>0\) in \(\Omega\); see
\cite[Theorem~2.9]{del-pezzo-fernandez-bonder}. Moreover,
\(\lambda_{1,p}(\Omega,V)\) is simple, meaning that any two
corresponding eigenfunctions are proportional; see
\cite[Theorem~2.13]{del-pezzo-fernandez-bonder}. Consequently, there
exists a unique positive first eigenfunction normalized by
\(\|u_1\|_{L^p(\Omega)}=1\).

The regularity, boundary nondegeneracy, and convergence results needed
below are proved in Appendix~\ref{app:logconca-technical}. In
particular, we shall use
Propositions~\ref{prop:first-eigenfunction-regularity},
\ref{prop:first-eigenfunction-boundary-regularity},
\ref{prop:regularized-first-eigenpair}, and
\ref{prop:regularized-eigenfunction-positivity-regularity}, together
with Corollary~\ref{cor:regularized-uniform-convergence}.

\subsection{Boundary concavity and its stability under regularization}

We first record the following boundary concavity result, adapted from
\cite[Lemmas~2.1 and~2.4]{korevaar-convex-solutions}; notably, the
boundary regularity assumption
\(\partial\Omega\in C^{2,\alpha_\Omega}\) is sufficient for the argument.

\begin{proposition}
\label{prop:korevaar-boundary-geometry}
Let \(\Omega\subset\mathbb R^N\) be a bounded strongly convex domain
with \(\partial\Omega\in C^{2,\alpha_\Omega}\) for some
\(\alpha_\Omega\in(0,1)\), and let \(\mathcal N\) be a neighborhood of
\(\partial\Omega\). Suppose that
\[
u\in C^1(\overline\Omega)\cap C^2(\mathcal N\cap\Omega),
\qquad
u>0\quad\text{in }\Omega,
\qquad
u=0\quad\text{on }\partial\Omega,
\]
and $\frac{\partial u}{\partial n}>0$ on $\partial\Omega,$
where \(n\) denotes the inward unit normal. Let
\(\varphi\in C^2((0,\infty))\) satisfy
\[
\lim_{t\to0^+}\varphi'(t)=+\infty,
\qquad
\varphi''(t)<0<\varphi'(t)
\quad\text{for all sufficiently small }t>0,
\]
and
\[
\lim_{t\to0^+}\frac{\varphi(t)}{\varphi'(t)}
=
\lim_{t\to0^+}\frac{\varphi'(t)}{\varphi''(t)}
=
0.
\]
Set \(w:=\varphi(u)\) and
\[ \Omega_\delta
:=
\bigl\{x\in\Omega:
\operatorname{dist}(x,\partial\Omega)>\delta
\bigr\}.\]
Then there exists \(\delta_*>0\) such that $-D^2w>0$ in $\Omega\setminus\overline{\Omega_{\delta_*}}.$
Moreover, for every \(\delta\in(0,\delta_*)\) and every
\(y\in\partial\Omega_\delta\),
\[
w(z)
<
w(y)+\nabla w(y)\cdot(z-y)
\qquad
\text{for every }
z\in\overline{\Omega_\delta}\setminus\{y\}.
\]
\end{proposition}

In the remainder of this subsection, we assume that \(\Omega\) is
strongly convex and that \(\delta>0\) is sufficiently small so that
\(\Omega_\delta\) is also strongly convex.

\smallskip

We next establish the strict concavity of \(\log u\) near the boundary
of each inner parallel domain, together with a global supporting-plane
property that will exclude boundary maxima in the subsequent two-point
argument.

\begin{proposition}
\label{prop:log-eigenfunction-global-boundary-support}
Let \(u\in\mathcal K\) be the positive first eigenfunction of
\eqref{eq:weighted-p-laplacian-eigenvalue}, and set \(v:=\log u\).
There exists \(\delta_1\in(0,\delta_0)\) such that, for every
\(\delta\in(0,\delta_1]\), there exist a relatively open neighborhood
\(\mathcal U_\delta\) of \(\partial\Omega_\delta\) in
\(\overline{\Omega_\delta}\) and \(\kappa_\delta>0\) satisfying $-D^2v\geq2\kappa_\delta I$ in $\mathcal U_\delta.$
Moreover, for every \(y\in\partial\Omega_\delta\),
\[
v(z)
<
v(y)+\nabla v(y)\cdot(z-y)
\qquad
\text{for every }
z\in\overline{\Omega_\delta}\setminus\{y\},
\]
where $\delta_0$ is given in Proposition~\ref{prop:first-eigenfunction-boundary-regularity}.
\end{proposition}

\begin{proof}
By Proposition~\ref{prop:first-eigenfunction-regularity} and Proposition~\ref{prop:first-eigenfunction-boundary-regularity}, we know that
\[
u\in C^1(\overline\Omega)
\cap C^2(\mathcal N_{\delta_0}\cap\Omega),
\qquad
u>0\ \text{in }\Omega,
\qquad
u=0\ \text{on }\partial\Omega,
\]
and $\frac{\partial u}{\partial n}>0$ on $\partial\Omega,$
where \(n=-\nu\) is the inward unit normal. For \(\varphi(t):=\log t\),
\[
\varphi'(t)=\frac1t\to+\infty,
\qquad
\varphi''(t)=-\frac1{t^2}<0<\varphi'(t)
\quad\text{as }t\to0^+,
\]
and
\[
\frac{\varphi(t)}{\varphi'(t)}
=t\log t\to0,
\qquad
\frac{\varphi'(t)}{\varphi''(t)}
=-t\to0\quad\text{as }t\to0^+.
\]
Hence Proposition~\ref{prop:korevaar-boundary-geometry}, applied with
\(w=\varphi(u)=v\), yields \(\delta_1\in(0,\delta_0)\) such that $-D^2v>0$
in a boundary neighborhood containing
\(\partial\Omega_\delta\) for every
\(\delta\in(0,\delta_1]\), and
\[
v(z)
<
v(y)+\nabla v(y)\cdot(z-y)
\]
for every \(y\in\partial\Omega_\delta\) and
\(z\in\overline{\Omega_\delta}\setminus\{y\}\).

For each fixed \(\delta\), the compactness of
\(\partial\Omega_\delta\) and the continuity of \(D^2v\) allow us to
choose a relatively open neighborhood \(\mathcal U_\delta\) of
\(\partial\Omega_\delta\) and \(\kappa_\delta>0\) such that $-D^2v\geq2\kappa_\delta I$ in $\mathcal U_\delta.$
\end{proof}

Using the \(C^2\)-convergence established above, we transfer the strict
boundary concavity and global supporting-plane property of \(\log u\)
to the regularized logarithmic eigenfunctions \(\log u_\varepsilon\),
uniformly for sufficiently small \(\varepsilon>0\).

\begin{proposition}
\label{prop:regularized-log-boundary-geometry}
Let \(\delta_0>0\) be given by
Proposition~\ref{prop:first-eigenfunction-boundary-regularity}, and let
\(u_\varepsilon\) be the positive regularized first eigenfunction
constructed in
Proposition~\ref{prop:regularized-first-eigenpair}. Fix
\(0<\delta<\delta_1\), and set $v_\varepsilon:=\log u_\varepsilon.$
Then there exist \(\varepsilon_\delta>0\), a relatively open
neighborhood \(\mathcal U_\delta\) of
\(\partial\Omega_\delta\) in \(\overline{\Omega_\delta}\), and a
constant \(\kappa_\delta>0\) such that, for every
\(0<\varepsilon\leq\varepsilon_\delta\), $-D^2v_\varepsilon\geq\kappa_\delta I$ in $\mathcal U_\delta,$
and
\[
v_\varepsilon(z)
<
v_\varepsilon(y)
+
\nabla v_\varepsilon(y)\cdot(z-y)
\]
for every \(y\in\partial\Omega_\delta\) and every
\(z\in\overline{\Omega_\delta}\setminus\{y\}\).
\end{proposition}

\begin{proof}
By Proposition~\ref{prop:log-eigenfunction-global-boundary-support},
there exist a relatively open neighborhood
\(\mathcal U_\delta^*\) of \(\partial\Omega_\delta\), with $\overline{\mathcal U_\delta^*} \Subset \Omega \setminus Z_u,$
and \(\kappa_\delta>0\) such that $-D^2v\geq2\kappa_\delta I$ on $\overline{\mathcal U_\delta^*}.$
The convergence in
Corollary~\ref{cor:regularized-uniform-convergence}
therefore gives, after decreasing \(\varepsilon_\delta\), $-D^2v_\varepsilon\geq\kappa_\delta I$ on $\overline{\mathcal U_\delta^*}.$

 Since \(\partial\Omega_\delta\) is compact, there exists \(r_\delta>0\)
such that $[y,z]\subset\mathcal U_\delta^*$
whenever \(y\in\partial\Omega_\delta\),
\(z\in\overline{\Omega_\delta}\), and \(|z-y|<r_\delta\). Hence
Taylor's formula yields
\[
\begin{aligned}
&v_\varepsilon(y)
+\nabla v_\varepsilon(y)\cdot(z-y)
-v_\varepsilon(z)=
-\int_0^1(1-t)
D^2v_\varepsilon\bigl(y+t(z-y)\bigr)
[z-y,z-y]\,dt
\geq
\frac{\kappa_\delta}{2}|z-y|^2
\end{aligned}
\]
whenever \(0<|z-y|<r_\delta\). For the remaining pairs, set
\[
\mathcal K_\delta
:=
\left\{
(y,z)\in\partial\Omega_\delta
\times\overline{\Omega_\delta}:
|z-y|\geq r_\delta
\right\}.
\]
By Proposition~\ref{prop:log-eigenfunction-global-boundary-support},
\[
v(y)+\nabla v(y)\cdot(z-y)-v(z)>0
\qquad\text{on }\mathcal K_\delta.
\]
Since \(\mathcal K_\delta\) is compact, its minimum is positive.
Moreover, $v_\varepsilon\rightarrow v$ in $C^1(\overline{\Omega_\delta})$
by Corollary~\ref{cor:regularized-uniform-convergence}.
Thus the same strict inequality holds on \(\mathcal K_\delta\) for all
sufficiently small \(\varepsilon\), completing the proof.
\end{proof}

\subsection{The two-point maximum principle and log-concavity}

Fix \(\delta>0\) sufficiently small and define the inner parallel set
\[
 \Omega_\delta=
\bigl\{x\in\Omega:
\operatorname{dist}(x,\partial\Omega)>\delta
\bigr\}.
\]
Since \(\Omega\) is strongly convex, \(\Omega_\delta\) is also strongly convex. For $(y,z,t)\in
\overline \Omega_\delta\times\overline \Omega_\delta\times[0,1],$ define
\[\mathcal C_\varepsilon(y,z,t)
:=
(1-t)v_\varepsilon(y)
+
tv_\varepsilon(z)
-
v_\varepsilon((1-t)y+tz).\]
Moreover, \(v_\varepsilon\) is concave in \(\Omega_\delta\) if and only if $\mathcal C_\varepsilon(y,z,t)\leq0$
for every $(y,z,t)\in
\overline \Omega_\delta\times\overline \Omega_\delta\times[0,1].$

\begin{lemma}
\label{lem:regularized-two-point-boundary-exclusion}
Let \(0<\varepsilon\leq\varepsilon_\delta\), where
\(\varepsilon_\delta\) is given by
Proposition~\ref{prop:regularized-log-boundary-geometry}. If
\(\mathcal C_\varepsilon\) attains a positive maximum on $\overline{\Omega_\delta}\times
\overline{\Omega_\delta}\times[0,1],$
then every maximizing point belongs to $\Omega_\delta\times\Omega_\delta\times(0,1).$
Furthermore, its first two components are distinct.
\end{lemma}

\begin{proof}
Since $\mathcal C_\varepsilon(y,z,0)
=
\mathcal C_\varepsilon(y,z,1)
=
\mathcal C_\varepsilon(y,y,t)
=
0,$
a positive maximum must satisfy \(t\in(0,1)\) and \(y\neq z\). Suppose that a positive maximum is attained at $(y,z,t)\in
\partial\Omega_\delta\times
\overline{\Omega_\delta}\times(0,1).$
By Proposition~\ref{prop:regularized-log-boundary-geometry},
\[
\kappa
:=
\nabla v_\varepsilon(y)\cdot(z-y)
-\bigl(v_\varepsilon(z)-v_\varepsilon(y)\bigr)
>0.
\]
Setting $g(s):=v_\varepsilon\bigl((1-s)y+sz\bigr),$
we obtain, as \(s\downarrow0\),
\[
\mathcal C_\varepsilon(y,z,s)
=
(1-s)g(0)+sg(1)-g(s)
=
-\kappa s+o(s)<0.
\]
Choose \(s\in(0,t)\) sufficiently small and define
\[
y_s:=(1-s)y+sz,
\qquad
\tau:=\frac{t-s}{1-s}.
\]
Since \(\Omega_\delta\) is strongly convex, \(y_s\in\Omega_\delta\);
moreover, \(\tau\in(0,1)\) and
\[
(1-\tau)y_s+\tau z=(1-t)y+tz.
\]
A direct computation gives
\[
(1-s)\mathcal C_\varepsilon(y_s,z,\tau)
=
(1-s)\mathcal C_\varepsilon(y,z,t)
-
(1-t)\mathcal C_\varepsilon(y,z,s).
\]
Since \(\mathcal C_\varepsilon(y,z,s)<0\), it follows that $\mathcal C_\varepsilon(y_s,z,\tau)
>
\mathcal C_\varepsilon(y,z,t),$
contradicting maximality. Hence \(y\notin\partial\Omega_\delta\).
Interchanging \(y\) and \(z\) yields
\(z\notin\partial\Omega_\delta\), and the conclusion follows.
\end{proof}

We now apply the two-point maximum principle to
\(v_\varepsilon=\log u_\varepsilon\). The strict convexity of the
potential rules out a positive interior maximum of the concavity
defect, while Proposition~\ref{prop:regularized-log-boundary-geometry}
excludes boundary maxima.

\begin{proposition}
\label{prop:regularized-log-concavity-uniformly-convex-potential}
Assume that \(\Omega\) is strongly convex and that \(V\) is uniformly
convex in the sense that, for some \(\mu>0\), $V-\frac{\mu}{2}|\cdot|^2$ is convex in $\Omega.$
Fix \(\delta>0\) sufficiently small. Then there exists
\(\varepsilon_\delta>0\) such that, for every
\(\varepsilon\in(0,\varepsilon_\delta]\), $v_\varepsilon:=\log u_\varepsilon$
is concave in \(\Omega_\delta\).
\end{proposition}

\begin{proof}
Choose \(\varepsilon_\delta>0\) so that
Lemma~\ref{lem:regularized-two-point-boundary-exclusion} applies, and
fix \(0<\varepsilon\leq\varepsilon_\delta\). Suppose that the concavity
defect
\[
\mathcal C_\varepsilon(y,z,t)
:=
(1-t)v_\varepsilon(y)+tv_\varepsilon(z)
-v_\varepsilon((1-t)y+tz)
\]
has a positive maximum. By
Lemma~\ref{lem:regularized-two-point-boundary-exclusion}, it is attained
at some
\[
(y,z,t)\in\Omega_\delta\times\Omega_\delta\times(0,1),
\qquad y\neq z.
\]
Set \(x:=(1-t)y+tz\). The first and second order conditions at this
maximum give
\[
\nabla v_\varepsilon(y)
=
\nabla v_\varepsilon(z)
=
\nabla v_\varepsilon(x)
=:q
\]
and
\[
(1-t)D^2v_\varepsilon(y)
+tD^2v_\varepsilon(z)
-D^2v_\varepsilon(x)
\leq0
\]
in the sense of quadratic forms. For \(q\in\mathbb R^N\), define
\[
A_\varepsilon(q)
:=
(\varepsilon+|q|^2)^{\frac{p-2}{2}}I
+
(p-2)(\varepsilon+|q|^2)^{\frac{p-4}{2}}
q\otimes q,
\]
where \(I\) denotes the \(N\times N\) identity matrix.
Since \(A_\varepsilon(q)\) is positive definite, taking the trace after
multiplication by \(A_\varepsilon(q)\) yields
\[
\begin{aligned}
0\geq{}&
(1-t)\operatorname{tr}
\bigl(A_\varepsilon(q)D^2v_\varepsilon(y)\bigr)
+t\operatorname{tr}
\bigl(A_\varepsilon(q)D^2v_\varepsilon(z)\bigr)-\operatorname{tr}
\bigl(A_\varepsilon(q)D^2v_\varepsilon(x)\bigr).
\end{aligned}
\]
On the other hand, \(v_\varepsilon\) satisfies
\[
\operatorname{tr}\!\left[
A_\varepsilon(\nabla v_\varepsilon)D^2v_\varepsilon
\right]
+\lambda_\varepsilon-V
+H_\varepsilon(\nabla v_\varepsilon)
=0,
\]
where
\[
H_\varepsilon(q)
:=
(\varepsilon+|q|^2)^{\frac{p-2}{2}}
\bigl((p-1)|q|^2-\varepsilon\bigr).
\]
Using the equality of the three gradients, the terms involving
\(\lambda_\varepsilon\) and \(H_\varepsilon(q)\) cancel, and hence
\[
(1-t)V(y)+tV(z)-V(x)\leq0.
\]
Since \(V-\frac{\mu}{2}|\cdot|^2\) is convex,
\[
(1-t)V(y)+tV(z)-V(x)
\geq
\frac{\mu}{2}t(1-t)|y-z|^2>0,
\]
which is a contradiction. Therefore
\(\mathcal C_\varepsilon\leq0\), and \(v_\varepsilon\) is concave in
\(\Omega_\delta\).
\end{proof}

Passing to the limit as \(\varepsilon\downarrow0\) in
Proposition~\ref{prop:regularized-log-concavity-uniformly-convex-potential}
and then letting the inner domains exhaust \(\Omega\), we obtain the
log-concavity of the original first eigenfunction.

\begin{theorem}
\label{thm:log-concavity-uniformly-convex-potential}
Assume that \(\Omega\) is strongly convex and that, for some \(\mu>0\), $V-\frac{\mu}{2}|\cdot|^2$ is convex in $\Omega.$
Then the positive first eigenfunction \(u\) of
\eqref{eq:weighted-p-laplacian-eigenvalue} is log-concave in \(\Omega\).
\end{theorem}

\begin{proof}
Fix \(x,y\in\Omega\). Since \([x,y]\Subset\Omega\), there exists
\(\delta>0\) sufficiently small such that $[x,y]\subset\Omega_\delta.$
By Proposition~
\ref{prop:regularized-log-concavity-uniformly-convex-potential},
\(v_\varepsilon=\log u_\varepsilon\) is concave in \(\Omega_\delta\)
for all sufficiently small \(\varepsilon>0\). Hence, for every
\(t\in[0,1]\),
\[
u_\varepsilon\bigl((1-t)x+ty\bigr)
\geq
u_\varepsilon(x)^{1-t}u_\varepsilon(y)^t.
\]
Since \(u_\varepsilon\to u\) locally uniformly in \(\Omega\) by
Corollary~\ref{cor:regularized-uniform-convergence},
letting \(\varepsilon\downarrow0\) yields
\[
u\bigl((1-t)x+ty\bigr)
\geq
u(x)^{1-t}u(y)^t.
\]
Thus \(u\) is log-concave in \(\Omega\).
\end{proof}

We now remove the uniform convexity assumption on the potential by an
approximation argument. 

\begin{theorem}
\label{thm:log-concavity-strongly-convex-domain}
Assume that \(\Omega\) is strongly convex and that
\(V\in C^{0,\alpha_V}(\overline\Omega)\) is convex. Then the 
positive first eigenfunction \(u\) of
\eqref{eq:weighted-p-laplacian-eigenvalue} is log-concave in
\(\Omega\).
\end{theorem}

\begin{proof}
For \(\eta>0\), set $V_\eta(x):=V(x)+\eta|x|^2$
and let \((\lambda_\eta,u_\eta)\) be the corresponding normalized
positive first eigenpair. Since $V_\eta-\eta|\cdot|^2=V$
is convex, \(V_\eta\) is uniformly convex with modulus \(2\eta\).
Hence Theorem~
\ref{thm:log-concavity-uniformly-convex-potential} implies that
\(u_\eta\) is log-concave in \(\Omega\).

We claim that \(u_\eta\to u\) locally uniformly as \(\eta\downarrow0\).
Indeed, since
\[
\|V_\eta-V\|_{L^\infty(\Omega)}
\leq
\eta\sup_{x\in\Omega}|x|^2\longrightarrow0,
\]
the variational characterization gives $|\lambda_\eta-\lambda_1|
\leq
\|V_\eta-V\|_{L^\infty(\Omega)}
\rightarrow0.$
Moreover, \(\{u_\eta\}\) is bounded in \(W_0^{1,p}(\Omega)\). Thus,
along any sequence \(\eta_j\downarrow0\),
\[
u_{\eta_j}\rightharpoonup u_*
\quad\text{in }W_0^{1,p}(\Omega),
\qquad
u_{\eta_j}\to u_*
\quad\text{in }L^p(\Omega)
\]
for a subsequence. The weak lower semicontinuity of the energy shows
that \(u_*\) is a nonnegative normalized minimizer for the potential
\(V\); hence \(u_*=u\) by simplicity of the first eigenvalue. In
addition,
\[
\int_\Omega|\nabla u_{\eta_j}|^p\,dx
=
\lambda_{\eta_j}
-\int_\Omega V_{\eta_j}u_{\eta_j}^p\,dx
\longrightarrow
\int_\Omega|\nabla u|^p\,dx,
\]
so \(u_{\eta_j}\to u\) strongly in \(W_0^{1,p}(\Omega)\).

The uniform boundedness of \(\{\lambda_\eta\}\) and
\(\{V_\eta\}\), together with the standard Moser and local H\"older
estimates (similar to Proposition \ref{prop:regularized-eigenfunction-positivity-regularity}) for
\[
-\Delta_pu_\eta=(\lambda_\eta-V_\eta)u_\eta^{p-1},
\]
gives, for every \(K\Subset\Omega\), $\|u_\eta\|_{C^{0,\alpha_K}(K)}\leq C_K$
uniformly for small \(\eta\). The Arzelà--Ascoli theorem and the uniqueness of the
\(L^p\)-limit therefore yield $u_\eta\rightarrow u$ locally uniformly in  $\Omega
.$ Finally, the log-concavity of \(u_\eta\) gives
\[
u_\eta\bigl((1-t)x+ty\bigr)
\geq
u_\eta(x)^{1-t}u_\eta(y)^t.
\]
Letting \(\eta\downarrow0\) proves the desired inequality for \(u\).
\end{proof}

We are now in a position to prove the log-concavity of the positive
first eigenfunction on an arbitrary bounded convex domain.

\begin{proof}[\textbf{Proof of Theorem~\ref{thm:first-eigenfunction-log-concavity}.}]
Let \(\{\Omega_j\}\) be a smooth strongly convex inner exhaustion of
\(\Omega\), namely,
\[
\overline{\Omega_j}\subset\Omega,
\qquad
\Omega_j\to\Omega
\quad\text{in the Hausdorff sense},
\]
and every compact subset of \(\Omega\) is contained in \(\Omega_j\)
for all sufficiently large \(j\). Since \(V\) is convex,
\(V|_{\overline{\Omega_j}}\) is Lipschitz for every \(j\). Hence
Theorem~\ref{thm:log-concavity-strongly-convex-domain} applies on
\(\Omega_j\). Let $\lambda_j:=\lambda_{1,p}(\Omega_j,V),$
and let \(u_j\) be the corresponding normalized positive first
eigenfunction, extended by zero to \(\Omega\). By
Theorem~\ref{thm:log-concavity-strongly-convex-domain}, \(u_j\) is
log-concave in \(\Omega_j\).

We first show that \(u_j\to u\) locally uniformly in \(\Omega\).
For every \(w\in C_c^\infty(\Omega)\setminus\{0\}\), one has
\(\operatorname{supp}w\subset\Omega_j\) for all sufficiently large
\(j\), and hence
\[
\limsup_{j\to\infty}\lambda_j
\leq
\frac{\displaystyle\int_\Omega|\nabla w|^p\,dx
+\int_\Omega V|w|^p\,dx}
{\displaystyle\int_\Omega|w|^p\,dx}.
\]
Taking the infimum over \(w\) gives $\limsup_{j\to\infty}\lambda_j
\leq
\lambda_{1,p}(\Omega,V)=:\lambda.$
Since \(\|u_j\|_{L^p(\Omega)}=1\) and \(\{\lambda_j\}\) is bounded,
\(\{u_j\}\) is bounded in \(W_0^{1,p}(\Omega)\). Thus, along a
subsequence,
\[
u_j\rightharpoonup u_*
\quad\text{in }W_0^{1,p}(\Omega),
\qquad
u_j\to u_*
\quad\text{in }L^p(\Omega).
\]
By weak lower semicontinuity,
\[
\lambda
\leq
\mathcal E_V(u_*)
\leq
\liminf_{j\to\infty}\mathcal E_V(u_j)
=
\liminf_{j\to\infty}\lambda_j.
\]
Consequently, \(\lambda_j\to\lambda\), and \(u_*\) is a nonnegative
normalized first eigenfunction on \(\Omega\). By simplicity,
\(u_*=u\). Moreover,
\[
\int_\Omega|\nabla u_j|^p\,dx
=
\lambda_j-\int_\Omega Vu_j^p\,dx
\longrightarrow
\lambda-\int_\Omega Vu^p\,dx
=
\int_\Omega|\nabla u|^p\,dx,
\]
so the uniform convexity of \(L^p\) yields $u_j\longrightarrow u$ strongly in $W_0^{1,p}(\Omega).$

The uniform bounds for \(\lambda_j\) and \(V\), together with the
standard Moser and interior H\"older estimates (similar to Proposition \ref{prop:regularized-eigenfunction-positivity-regularity}) for $-\Delta_pu_j=(\lambda_j-V)u_j^{p-1},$
imply that, for every \(K\Subset\Omega\), $\|u_j\|_{C^{0,\alpha_K}(K)}\leq C_K$
for some \(\alpha_K\in(0,1)\) and \(C_K>0\) independent of large
\(j\). The Arzelà--Ascoli theorem and the uniqueness of the \(L^p\)-limit then give
\[
u_j\longrightarrow u
\qquad
\text{locally uniformly in }\Omega.
\]

Finally, fix \(x,y\in\Omega\) and \(t\in[0,1]\). Since
\([x,y]\Subset\Omega\), one has \([x,y]\subset\Omega_j\) for all
sufficiently large \(j\). The log-concavity of \(u_j\) gives $u_j\bigl((1-t)x+ty\bigr)
\geq
u_j(x)^{1-t}u_j(y)^t.$
Let \(j\to\infty\), we obtain
\[
u\bigl((1-t)x+ty\bigr)
\geq
u(x)^{1-t}u(y)^t,
\]
which proves the assertion.
\end{proof}

\section{Ground-state Stability and Degenerate Weighted Poincar\'e Estimates}
\label{sec:ground-state-stability-weighted-poincare}

Throughout this section, we assume that $2\leq p<\infty$ and write $D:=\operatorname{diam}(\Omega).$
For $V\in C(\overline\Omega)$, consider
\[
\mathcal E_V(w)
:=
\int_\Omega |\nabla w|^p\,dx
+
\int_\Omega V(x)|w|^p\,dx,
\qquad
w\in W_0^{1,p}(\Omega),
\]
and set
\[
\mathcal K
:=
\left\{
w\in W_0^{1,p}(\Omega):
\|w\|_{L^p(\Omega)}=1
\right\}.
\]
Let $u_1\in\mathcal K$ denote the positive first eigenfunction associated
with $\lambda_{1,p}(\Omega,V)$.

\subsection{Ground-state stability and the dimension-free gap}
\label{subsec:ground-state-stability-gap}

We first record the weighted Poincar\'e inequality for log-concave
weights that will be used below.

\begin{lemma}
\label{lem:weighted-poincare-log-concave-gap}
Let $p>1$ and let $\omega$ be a positive log-concave function on a
bounded convex domain $\Omega$. Then every Lipschitz function $f$ on
$\Omega$ satisfies
\[\int_\Omega |\nabla f|^p\omega\,dx
\geq
(p-1)\left(\frac{\pi_p}{D}\right)^p
\inf_{c\in\mathbb R}
\int_\Omega |f-c|^p\omega\,dx.\]
\end{lemma}

\begin{proof}
For every $f\in L^p(\Omega,\omega\,dx)$, let $a_f\in\mathbb R$ be the
unique weighted $p$-mean of $f$, characterized by
\[
\int_\Omega
|f-a_f|^{p-2}(f-a_f)\omega\,dx=0.
\]
Applying \cite[Theorem~1.1]{ferone-nitsch-trombetti} to
$g:=f-a_f$, we obtain
\[
\begin{aligned}
\int_\Omega|\nabla f|^p\omega\,dx
&=
\int_\Omega|\nabla g|^p\omega\,dx\geq
\left(\frac{\widehat\pi_p}{D}\right)^p
\inf_{c\in\mathbb R}
\int_\Omega|g-c|^p\omega\,dx\\
&=
(p-1)\left(\frac{\pi_p}{D}\right)^p
\inf_{c\in\mathbb R}
\int_\Omega|f-c|^p\omega\,dx,
\end{aligned}
\]
where
\[
\widehat\pi_p
=
2\int_0^\infty
\frac{ds}{1+\frac{s^p}{p-1}}
=
(p-1)^{1/p}\pi_p,
\]
which completes the proof.
\end{proof}

We next combine the quantitative Picone identity associated with the
first eigenfunction with the preceding weighted inequality.

\begin{proposition}
\label{prop:potential-poincare-stability}
Let $p\geq2$, $V\in C(\overline\Omega)$, and let $u_1$ be the positive
first eigenfunction associated with $\lambda_{1,p}(\Omega,V)$. Set $d\mu:=u_1^p\,dx.$
Then, for every $w\in W_0^{1,p}(\Omega)$, the following hold.

\medskip
\noindent\textnormal{\rm (i)}
Assume in addition that $V$ is convex. Then
\begin{equation}
\label{eq:potential-poincare-stability}
\mathcal E_V(w)
-
\lambda_{1,p}(\Omega,V)\|w\|_{L^p(\Omega)}^p
\geq
\mathfrak c_p(p-1)
\left(\frac{\pi_p}{D}\right)^p
\inf_{c\in\mathbb R}
\|w-cu_1\|_{L^p(\Omega)}^p.
\end{equation}

\medskip
\noindent\textnormal{\rm (ii)}
We have
\begin{equation}
\label{eq:potential-weighted-stability}
\mathcal E_V(w)
-
\lambda_{1,p}(\Omega,V)\|w\|_{L^p(\Omega)}^p
\geq
\frac{2}{p}\Lambda_p(\Omega,V)
\inf_{c\in\mathbb R}
\int_\Omega |f_w-c|^2\,d\mu,
\end{equation}
where
\[
\Lambda_p(\Omega,V)
:=
\inf_{\substack{
g\in C^{0,1}(\overline\Omega)\setminus\{0\}\\
\int_\Omega g\,d\mu=0}}
\frac{
\displaystyle\int_\Omega
|\nabla\log u_1|^{p-2}|\nabla g|^2\,d\mu
}{
\displaystyle\int_\Omega g^2\,d\mu
},
\qquad
f_w
:=
\operatorname{sgn}\left(\frac{w}{u_1}\right)
\left|\frac{w}{u_1}\right|^{p/2}.
\]
\end{proposition}

\begin{proof}
\textbf{(i).}
We first take $w\in C_c^\infty(\Omega)$. Testing the equation for
$u_1$ with $|w|^p/u_1^{p-1}$ gives
\[
\mathcal E_V(w)
-
\lambda_{1,p}(\Omega,V)\|w\|_{L^p(\Omega)}^p
=
\int_\Omega
\mathcal C_p\left(
\nabla w,
u_1\nabla\left(\frac{w}{u_1}\right)
\right)\,dx.
\]
Hence, by \eqref{eq:cp-remainder-estimate},
\[
\mathcal E_V(w)
-
\lambda_{1,p}(\Omega,V)\|w\|_{L^p(\Omega)}^p
\geq
\mathfrak c_p
\int_\Omega
u_1^p
\left|
\nabla\left(\frac{w}{u_1}\right)
\right|^pdx.
\]
By Theorem~\ref{thm:first-eigenfunction-log-concavity}, the weight
$u_1^p$ is log-concave. Applying
Lemma~\ref{lem:weighted-poincare-log-concave-gap} with
$f=w/u_1$ and $\omega=u_1^p$, we obtain
\[
\begin{aligned}
\mathcal E_V(w)
-
\lambda_{1,p}(\Omega,V)\|w\|_{L^p(\Omega)}^p
&\geq
\mathfrak c_p(p-1)
\left(\frac{\pi_p}{D}\right)^p
\inf_{c\in\mathbb R}
\int_\Omega
\left|\frac{w}{u_1}-c\right|^pu_1^p\,dx\\
&=
\mathfrak c_p(p-1)
\left(\frac{\pi_p}{D}\right)^p
\inf_{c\in\mathbb R}
\|w-cu_1\|_{L^p(\Omega)}^p.
\end{aligned}
\]
The conclusion extends to every $w\in W_0^{1,p}(\Omega)$ by density.
Indeed, let $w_n\in C_c^\infty(\Omega)$ satisfy
$w_n\to w$ strongly in $W_0^{1,p}(\Omega)$. Since
$V\in L^\infty(\Omega)$,
\[
\mathcal E_V(w_n)\to\mathcal E_V(w),
\qquad
\|w_n\|_{L^p}\to\|w\|_{L^p}.
\]
Moreover, the distance to a fixed subset of a normed space is
$1$-Lipschitz; hence
\[
\operatorname{dist}_{L^p(\Omega)}
\bigl(w_n,\operatorname{span}\{u_1\}\bigr)
\to
\operatorname{dist}_{L^p(\Omega)}
\bigl(w,\operatorname{span}\{u_1\}\bigr).
\]
Passing to the limit proves \eqref{eq:potential-poincare-stability}.

\smallskip
\textbf{(ii).}
We first take $w\in C_c^\infty(\Omega)$ and write $w=u_1z$.
Testing the equation for $u_1$ with $u_1|z|^p$ gives the exact
ground-state identity
\[
\mathcal E_V(w)
-
\lambda_{1,p}(\Omega,V)\int_\Omega|w|^p\,dx
=
\int_\Omega
\mathcal C_p\left(
u_1\nabla z+z\nabla u_1,
u_1\nabla z\right)\,dx.
\]
By \eqref{eq:cpaboptimal},
\[
\mathcal E_V(w)
-
\lambda_{1,p}(\Omega,V)\int_\Omega|w|^p\,dx
\geq
\frac p2
\int_\Omega
u_1^2|\nabla u_1|^{p-2}|z|^{p-2}|\nabla z|^2\,dx.
\]
For $f_w=\operatorname{sgn}(z)|z|^{p/2}\in C_c^{0,1}(\Omega),$
we have $|\nabla f_w|^2
=
\frac{p^2}{4}|z|^{p-2}|\nabla z|^2.$
Thus
\[\mathcal E_V(w)
-
\lambda_{1,p}(\Omega,V)\int_\Omega|w|^p\,dx
\geq
\frac{2}{p}
\int_\Omega
|\nabla\log u_1|^{p-2}|\nabla f_w|^2\,d\mu.\]
Since $f_w$ is Lipschitz, the definition of $\Lambda_p(\Omega,V)$
gives
\[
\int_\Omega
|\nabla\log u_1|^{p-2}|\nabla f_w|^2\,d\mu
\geq
\Lambda_p(\Omega,V)
\inf_{c\in\mathbb R}
\int_\Omega |f_w-c|^2\,d\mu.
\]
Hence \eqref{eq:potential-weighted-stability} holds for
$w\in C_c^\infty(\Omega)$.

For general $w\in W_0^{1,p}(\Omega)$, choose
$w_n\in C_c^\infty(\Omega)$ with
$w_n\to w$ strongly in $W_0^{1,p}(\Omega)$. Set $T(t):=\operatorname{sgn}(t)|t|^{p/2}.$
Since $u_1^{p/2}f_{w_n}=T(w_n),\,
u_1^{p/2}f_w=T(w),$
and the Nemytskii map $T:L^p(\Omega)\to L^2(\Omega)$ is continuous,
\[
\|f_{w_n}-f_w\|_{L^2(\mu)}
=
\|T(w_n)-T(w)\|_{L^2(\Omega)}
\longrightarrow0.
\]
Therefore
\[
\inf_{c\in\mathbb R}
\|f_{w_n}-c\|_{L^2(\mu)}
\longrightarrow
\inf_{c\in\mathbb R}
\|f_w-c\|_{L^2(\mu)}.
\]
Applying the already proved estimate to $w_n$ and passing to the limit
proves \eqref{eq:potential-weighted-stability} for every
$w\in W_0^{1,p}(\Omega)$.
\end{proof}

We use the mountain-pass characterization
\begin{equation}
\label{eq:lambda2-mountain-pass-potential}
\lambda_{2,p}(\Omega,V)
=
\inf_{\gamma\in\Gamma(u_1,-u_1)}
\max_{s\in[0,1]}\mathcal E_V(\gamma(s)),
\end{equation}
where
\[
\Gamma(u_1,-u_1)
:=
\left\{
\gamma\in C([0,1],\mathcal K):
\gamma(0)=u_1,\ \gamma(1)=-u_1
\right\}.
\]

\begin{lemma}
\label{lem:nonlinear-orthogonal-section}
Define
\[
\mathcal H_{u_1}
:=
\left\{
w\in\mathcal K:
\int_\Omega |w|^{p-2}wu_1\,dx=0
\right\}.
\]
Then every path $\gamma\in\Gamma(u_1,-u_1)$ intersects
$\mathcal H_{u_1}$. Moreover, for every $w\in\mathcal H_{u_1}$,
\[\inf_{c\in\mathbb R}
\|w-cu_1\|_{L^p(\Omega)}^p
=1.\]
\end{lemma}

\begin{proof}
For $\gamma\in\Gamma(u_1,-u_1)$, define
\[
F(s)
:=
\int_\Omega
|\gamma(s)|^{p-2}\gamma(s)u_1\,dx.
\]
Then $F$ is continuous and satisfies $F(0)=1$ and $F(1)=-1$.
Hence $F(s_0)=0$ for some $s_0\in(0,1)$.

If $w\in\mathcal H_{u_1}$, the function $\Phi(c):=\|w-cu_1\|_{L^p(\Omega)}^p$
is strictly convex and satisfies
\[
\Phi'(0)
=-p\int_\Omega|w|^{p-2}wu_1\,dx=0.
\]
Thus $c=0$ is its unique minimizer, and
\[
\inf_{c\in\mathbb R}
\|w-cu_1\|_{L^p(\Omega)}^p
=
\|w\|_{L^p(\Omega)}^p
=1,
\]
which completes the proof.
\end{proof}

We can now prove the first higher-dimensional gap estimate.

\begin{proof}[\textbf{Proof of Theorem~\ref{thm:dimension-free-higher-dimensional-gap}.}]
\textbf{(i).}
Let $\gamma\in\Gamma(u_1,-u_1)$. By
Lemma~\ref{lem:nonlinear-orthogonal-section}, there exists
$s_\gamma\in(0,1)$ such that
$w_\gamma:=\gamma(s_\gamma)\in\mathcal H_{u_1}$. Since
$w_\gamma\in\mathcal K$ and $\inf_{c\in\mathbb R}
\|w_\gamma-cu_1\|_{L^p(\Omega)}^p=1,$
Proposition~\ref{prop:potential-poincare-stability}(i) yields
\[
\mathcal E_V(w_\gamma)
\geq
\lambda_{1,p}(\Omega,V)
+
\mathfrak c_p(p-1)
\left(\frac{\pi_p}{D}\right)^p.
\]
Therefore
\[
\max_{s\in[0,1]}\mathcal E_V(\gamma(s))
\geq
\lambda_{1,p}(\Omega,V)
+
\mathfrak c_p(p-1)
\left(\frac{\pi_p}{D}\right)^p.
\]
Taking the infimum over $\gamma\in\Gamma(u_1,-u_1)$ and using
\eqref{eq:lambda2-mountain-pass-potential} proves (i).

\smallskip
\textbf{(ii).}
Define
\[
H(w)
:=
\int_\Omega
u_1^{p/2}\operatorname{sgn}(w)|w|^{p/2}\,dx.
\]
The continuity of $T(w)=\operatorname{sgn}(w)|w|^{p/2}:
L^p(\Omega)\rightarrow L^2(\Omega)$
shows that $H$ is continuous on $\mathcal K$. Moreover,
$H(u_1)=1$ and $H(-u_1)=-1$. Hence, for every
$\gamma\in\Gamma(u_1,-u_1)$, there exists $t_0\in(0,1)$ such that
$H(\gamma(t_0))=0$. Set $w:=\gamma(t_0)$ and $f:=f_w$. Since $u_1^{p/2}f
=
\operatorname{sgn}(w)|w|^{p/2},$
we have
\[
\int_\Omega f\,d\mu=0,
\qquad
\int_\Omega f^2\,d\mu
=
\int_\Omega|w|^p\,dx=1.
\]
Thus $\inf_{c\in\mathbb R}
\int_\Omega|f-c|^2\,d\mu=1.$
Applying Proposition~\ref{prop:potential-poincare-stability}(ii) gives
\[
\mathcal E_V(w)
-
\lambda_{1,p}(\Omega,V)
\geq
\frac{2}{p}\Lambda_p(\Omega,V).
\]
Since such a point $w$ exists on every path
$\gamma\in\Gamma(u_1,-u_1)$, \eqref{eq:lambda2-mountain-pass-potential}
yields
\[
\Gamma_p(\Omega,V)
\geq
\frac{2}{p}\Lambda_p(\Omega,V),
\]
which proves (ii).
\end{proof}

\subsection{A ground-state Hardy inequality}
\label{subsec:degenerate-weighted-poincare}

We now assume $V\equiv0,\,2\le p<\infty$ and  write $\lambda:=\lambda_{1,p}(\Omega,0)$
and let \(u_1>0\) be the first eigenfunction, satisfying
\begin{equation}
\label{eq:section4-first-eigenfunction}
\begin{cases}
-\Delta_pu_1=\lambda u_1^{p-1}&\text{in }\Omega,\\
u_1=0&\text{on }\partial\Omega,
\end{cases}
\qquad
\int_\Omega u_1^p\,dx=1.
\end{equation}
We introduce
\[d\mu:=u_1^p\,dx,
\qquad
v:=-\log u_1,
\qquad
r:=|\nabla v|=\frac{|\nabla u_1|}{u_1},\qquad \operatorname{div}_{\mu}X
:=
u_1^{-p}\operatorname{div}(u_1^pX).\]
By the log-concavity result established in Theorem \ref{thm:first-eigenfunction-log-concavity},
\(v\) is convex in \(\Omega\).

\smallskip

We first record the differential inequalities satisfied by $v$.

\begin{lemma}
\label{lem:log-ground-state-equations}
Let $\mathcal{Z}_{u_1}$ be defined as in \eqref{eq:zeroset}. Then, in the classical pointwise sense in $\Omega \setminus \mathcal{Z}_{u_1}$,
\begin{equation}
\label{eq:log-ground-state-equation}
\operatorname{div}\bigl(r^{p-2}\nabla v\bigr) = \lambda + (p-1)r^p.
\end{equation}
Furthermore, in the weak sense across the entirety of $\Omega$,
\begin{equation}
\label{eq:weighted-log-ground-state-equation}
\operatorname{div}_{\mu}\bigl(r^{p-2}\nabla v\bigr) = \lambda - r^p.
\end{equation}
Finally, as an inequality of Radon measures on $\Omega$,
\begin{equation}
\label{eq:convex-log-key-inequality}
\lambda r^{2-p} + (p-1)r^2 \leq (p-1)\Delta v.
\end{equation}
\end{lemma}

\begin{proof}
By the relation $u_1 = e^{-v}$, direct computation yields $|\nabla u_1|^{p-2}\nabla u_1 = -u_1^{p-1}r^{p-2}\nabla v.$
Substituting this into \eqref{eq:section4-first-eigenfunction} immediately yields \eqref{eq:log-ground-state-equation} in $\Omega \setminus \mathcal{Z}_{u_1}$. To deduce \eqref{eq:weighted-log-ground-state-equation}, we test the weak eigenvalue equation of \eqref{eq:section4-first-eigenfunction} with $u_1\varphi$ for $\varphi \in C_c^\infty(\Omega)$. Recalling $d\mu = u_1^p \,dx$, we obtain
\[
\int_\Omega r^{p-2}\nabla v \cdot \nabla \varphi \, d\mu = \int_\Omega (r^p - \lambda)\varphi \, d\mu,
\]
which is exactly the weak formulation of \eqref{eq:weighted-log-ground-state-equation}.

To prove \eqref{eq:convex-log-key-inequality}, we first work at points
where \(v\) is twice differentiable and \(r=|\nabla v|>0\). Note that
\[
\operatorname{div}\bigl(r^{p-2}\nabla v\bigr)
=
r^{p-2}
\bigl(\Delta v+(p-2)D^2v[e,e]\bigr),\quad e:=\frac{\nabla v}{|\nabla v|}.
\]
Since \(v\) is convex, \(D^2v\) is positive semidefinite at every such
point, and hence $0\leq D^2v[e,e]\leq\Delta v.$
Combining this with \eqref{eq:log-ground-state-equation} and \(p\geq2\)
gives
\begin{equation}
\label{eq:convex-log-key-ae}
\lambda r^{2-p}+(p-1)r^2
\leq
(p-1)\Delta v.
\end{equation}

We next pass from \eqref{eq:convex-log-key-ae} to a global measure
inequality. If \(p=2\), then \(r^{2-p}=1\), and
\eqref{eq:log-ground-state-equation} reduces to $\Delta v=\lambda+r^2,$
so \eqref{eq:convex-log-key-inequality} follows directly. Assume now that \(p>2\). By
\cite[Theorem~C.5]{BrascoLindgren2023},
\[
\int_\Omega |\nabla u_1|^{-s}\,dx<\infty
\qquad\text{for every }s<p-1.
\]
Taking \(s=p-2\) yields $|\mathcal Z_{u_1}|=0$ and $r^{2-p}\in L^1_{\mathrm{loc}}(\Omega).$
Since also \(r^2\in L^1_{\mathrm{loc}}(\Omega)\),
\[
\lambda r^{2-p}+(p-1)r^2\in L^1_{\mathrm{loc}}(\Omega),
\]
and hence \eqref{eq:convex-log-key-ae} holds almost everywhere in
\(\Omega\).

Since \(v\) is convex, its distributional Hessian \(D^2v\) is a
positive semidefinite matrix-valued Radon measure; see
\cite[Theorems~6.8, 6.9]{EvansGariepy2015}. Its absolutely continuous
part agrees almost everywhere with the classical Hessian. Thus
\eqref{eq:convex-log-key-ae} can be written as
\[
\lambda r^{2-p}+(p-1)r^2
\leq
(p-1)(\Delta v)_{\mathrm{ac}}
\qquad\text{a.e. in }\Omega.
\]
Writing the Lebesgue decomposition $\Delta v
=
(\Delta v)_{\mathrm{ac}}\,dx+(\Delta v)_{\mathrm{s}},$
the positivity of \(D^2v\) implies $(\Delta v)_{\mathrm{s}}\geq0.$
Consequently,
\[\bigl(\lambda r^{2-p}+(p-1)r^2\bigr)\,dx
\leq
(p-1)\Delta v\]
as Radon measures on \(\Omega\).
\end{proof}

The preceding inequality gives the inverse-weight estimate needed for
the weighted Poincar\'e inequality.

\begin{proposition}
\label{prop:ground-state-hardy}
Let \(p\geq2\). Then every \(g\in C^{0,1}(\overline\Omega)\) satisfies
\[
\int_\Omega r^{2-p}g^2\,d\mu
\leq
\lambda^{-\frac{p-2}{p}}
\left[
A_p\int_\Omega g^2\,d\mu
+
B_p\lambda^{-1}
\int_\Omega r^{p-2}|\nabla g|^2\,d\mu
\right],
\]
where $A_2=1,\,B_2=0,$ and, for \(p>2\),
\[
A_p=2(p-1)(2p-3),
\qquad
B_p=4(p-1)^2.
\]
In particular,
\[
\int_\Omega r^{2-p}g^2\,d\mu
\leq
C_p\lambda^{-\frac{p-2}{p}}
\left(
\int_\Omega g^2\,d\mu
+
\lambda^{-1}
\int_\Omega r^{p-2}|\nabla g|^2\,d\mu
\right),
\]
where one may take $C_2=1,\,
C_p=4(p-1)^2$ for $p>2.$
\end{proposition}

\begin{proof}
For \(p=2\), since \(r^{2-p}\equiv1\), the assertion is immediate.
We therefore assume \(p>2\).

By the change of variables \(y=\lambda^{1/p}x\), together with a
multiplicative renormalization of \(u_1\), it suffices to consider
\(\lambda=1\). Set
\[
M:=\int_\Omega g^2\,d\mu,
\qquad
E:=\int_\Omega r^{p-2}|\nabla g|^2\,d\mu,
\qquad
R_s:=\int_\Omega r^s g^2\,d\mu,
\qquad
I:=R_{2-p}.
\]

Choose \(\eta_\varepsilon\in C_c^\infty(\Omega)\) such that
\[
0\leq\eta_\varepsilon\leq1,\qquad
\eta_\varepsilon=0\ \text{if }d(x,\partial\Omega)\leq\varepsilon,
\qquad
\eta_\varepsilon=1\ \text{if }d(x,\partial\Omega)\geq2\varepsilon,
\]
and \(|\nabla\eta_\varepsilon|\leq C/\varepsilon\).
Since \(\eta_\varepsilon g^2u_1^p\) is a nonnegative compactly supported
Lipschitz function, it is admissible in
\eqref{eq:convex-log-key-inequality} by smooth approximation. Hence
\[
\int_\Omega\eta_\varepsilon r^{2-p}g^2\,d\mu
+
(p-1)\int_\Omega\eta_\varepsilon r^2g^2\,d\mu
\leq
(p-1)\int_\Omega\eta_\varepsilon g^2u_1^p\,d(\Delta v).
\]
Using \(\nabla(u_1^p)=-pu_1^p\nabla v\), we obtain
\[
\begin{aligned}
\int_\Omega\eta_\varepsilon g^2u_1^p\,d(\Delta v)
=
p\int_\Omega\eta_\varepsilon g^2r^2\,d\mu
-
2\int_\Omega\eta_\varepsilon
g\nabla v\cdot\nabla g\,d\mu
-
E_\varepsilon,
\end{aligned}
\]
where $E_\varepsilon
:=
\int_\Omega
g^2u_1^p\nabla v\cdot\nabla\eta_\varepsilon\,dx.$
The support of \(\nabla\eta_\varepsilon\) is contained in
\[
A_\varepsilon
:=
\{x\in\Omega:\varepsilon<d(x,\partial\Omega)<2\varepsilon\}.
\]
By Proposition~\ref{prop:first-eigenfunction-boundary-regularity}, we can obtain that the boundary estimates $u_1(x)\asymp d(x,\partial\Omega),\,
|\nabla u_1(x)|\asymp1$
in a fixed boundary collar. Since \(g\in C^{0,1}(\overline\Omega)\),
we have
\[
|E_\varepsilon|
\leq
C\varepsilon^{-1}
\int_{A_\varepsilon}u_1^{p-1}|\nabla u_1|\,dx
\leq
C\varepsilon^{p-1}
\longrightarrow0.
\]
Moreover, $r^{2-p}g^2,\,
r^2g^2,\,
|g|\,|\nabla g|\,r
\in L^1(\Omega,\mu).$
Indeed, the boundary integrability follows from
Proposition~\ref{prop:first-eigenfunction-boundary-regularity};
the possible singularity of \(r^{2-p}\) near
\(\mathcal Z_{u_1}\) is controlled by
\cite[Theorem~C.5]{BrascoLindgren2023}.
Letting \(\varepsilon\downarrow0\), we obtain
\[
I+(p-1)R_2
\leq
(p-1)\left(
pR_2
-
2\int_\Omega g\nabla v\cdot\nabla g\,d\mu
\right),
\]
and therefore
\begin{equation}
\label{eq:hardy-first-estimate}
I
\leq
(p-1)^2R_2-2(p-1)C,
\qquad
C:=
\int_\Omega g\nabla v\cdot\nabla g\,d\mu.
\end{equation}
Similarly, testing \eqref{eq:weighted-log-ground-state-equation} against
\(\eta_\varepsilon g^2\) and arguing as above to remove the cutoff, we obtain
\[R_p
=
M+2B,
\qquad
B:=
\int_\Omega
g r^{p-2}\nabla v\cdot\nabla g\,d\mu.\]

We now keep the sign in \eqref{eq:hardy-first-estimate}.
For \(\tau>0\),
\[
C-\tau B
=
\int_\Omega
g(1-\tau r^{p-2})
\nabla v\cdot\nabla g\,d\mu.
\]
By Cauchy--Schwarz,
\[
|C-\tau B|
\leq
E^{1/2}
\left(
R_{4-p}-2\tau R_2+\tau^2R_p
\right)^{1/2}.
\]
Using \(2B=R_p-M\) in \eqref{eq:hardy-first-estimate}, we obtain
\[
\begin{aligned}
I
\leq{}&
(p-1)\tau M
+
(p-1)^2R_2
-
(p-1)\tau R_p+
2(p-1)E^{1/2}
\left(
R_{4-p}-2\tau R_2+\tau^2R_p
\right)^{1/2}.
\end{aligned}
\]
Hence, for every \(\delta>0\), Young's inequality gives
\begin{equation}
\label{eq:hardy-two-parameter}
\begin{aligned}
I
\leq{}&
(p-1)\tau M+\delta E+
\left(
(p-1)^2-\frac{2(p-1)^2\tau}{\delta}
\right)R_2\\
&+
\left(
\frac{(p-1)^2\tau^2}{\delta}
-(p-1)\tau
\right)R_p
+
\frac{(p-1)^2}{\delta}R_{4-p}.
\end{aligned}
\end{equation}
Choose $\tau=p-1,\,
\delta=2(p-1)^2,$
Then \eqref{eq:hardy-two-parameter} becomes
\begin{equation}
\label{eq:hardy-special-parameter}
\begin{aligned}
I
\leq{}&
(p-1)^2M
+
2(p-1)^2E\\
&+
(p-1)(p-2)R_2
-
\frac{(p-1)^2}{2}R_p
+
\frac12R_{4-p}.
\end{aligned}
\end{equation}

We claim that, for every \(t>0\),
\begin{equation}
\label{eq:hardy-pointwise}
(p-1)(p-2)t^2
-\frac{(p-1)^2}{2}t^p
+\frac12t^{4-p}
\leq
(p-1)(p-2)+\frac12t^{2-p}.
\end{equation}
Indeed, if \(0<t\leq1\), then $t^{4-p}-t^{2-p}\le 0,$
and therefore
\[
\begin{aligned}
&(p-1)(p-2)t^2
-\frac{(p-1)^2}{2}t^p
+\frac12t^{4-p}
-\frac12t^{2-p}\le
(p-1)(p-2).
\end{aligned}
\]
Suppose now that \(t\geq1\).
Set \(s=t^2\geq1\) and define
\[
F(s)
:=
\left(
(p-1)(p-2)+\frac12
\right)s
-
\frac{(p-1)^2}{2}s^{p/2}.
\]
Then $F'(s)\le 0,$ and thus $F(s)
\leq
F(1)
=
\frac{(p-2)^2}{2}
\leq
(p-1)(p-2),$
which proves \eqref{eq:hardy-pointwise}.

Multiplying \eqref{eq:hardy-pointwise} by \(g^2\) and integrating with
respect to \(\mu\), we obtain
\[
(p-1)(p-2)R_2
-
\frac{(p-1)^2}{2}R_p
+
\frac12R_{4-p}
\leq
(p-1)(p-2)M+\frac12I.
\]
Combining this with \eqref{eq:hardy-special-parameter} gives
\[
I
\leq
(p-1)(2p-3)M
+
2(p-1)^2E
+
\frac12I.
\]
Therefore, $I
\leq
2(p-1)(2p-3)M
+
4(p-1)^2E.$
Since $2(p-1)(2p-3)\leq4(p-1)^2,$
we also have
\[
I
\leq
4(p-1)^2(M+E).
\]

Restoring the original scale yields
\[
\int_\Omega r^{2-p}g^2\,d\mu
\leq
\lambda^{-\frac{p-2}{p}}
\left[
2(p-1)(2p-3)
\int_\Omega g^2\,d\mu
+
4(p-1)^2\lambda^{-1}
\int_\Omega r^{p-2}|\nabla g|^2\,d\mu
\right],
\]
and hence the stated estimate with $C_p=4(p-1)^2,\, p>2.$
\end{proof}

We shall also use the following elementary diameter estimate.

\begin{lemma}
\label{lem:first-eigenvalue-diameter-lower-bound}
For every bounded convex domain $\Omega$,
\[
\lambda_{1,p}(\Omega,0)
\geq
(p-1)\left(\frac{\pi_p}{D}\right)^p.
\]
\end{lemma}

\begin{proof}
Fix $e\in\mathbb S^{N-1}$. Since $\Omega$ is convex, for almost every
$y\in e^\perp$ the section
\[
I_y:=\{t\in\mathbb R:y+te\in\Omega\}
\]
is an interval of length at most $D$. For
$\varphi\in C_c^\infty(\Omega)$, the one-dimensional Dirichlet
$p$-Poincar\'e inequality gives
\[
\int_{I_y}
|\partial_e\varphi(y+te)|^p\,dt
\geq
(p-1)\left(\frac{\pi_p}{D}\right)^p
\int_{I_y}|\varphi(y+te)|^p\,dt.
\]
Integrating with respect to $y$ and using
$|\partial_e\varphi|\leq|\nabla\varphi|$ proves the claim by density.
\end{proof}

We next extract an $L^1$ consequence of the sharp weighted Poincar\'e
inequality of \cite{ferone-nitsch-trombetti}.

\begin{lemma}
\label{lem:weighted-L1-diameter}
Let $\nu$ be a log-concave probability measure supported on a bounded
convex set of diameter $D$. Then, for every bounded Lipschitz function
$h$,
\begin{equation}
\label{eq:weighted-L1-diameter}
\inf_{a\in\mathbb R}
\int|h-a|\,d\nu
\leq
\frac D2\int|\nabla h|\,d\nu.
\end{equation}
\end{lemma}

\begin{proof}
For every $q>1$, \cite[Theorem~1.1]{ferone-nitsch-trombetti} gives
\[
\inf_{a\in\mathbb R}
\|h-a\|_{L^q(\nu)}
\leq
\frac{D}{\widehat\pi_q}
\|\nabla h\|_{L^q(\nu)},
\]
where
\[
\widehat\pi_q
=
2\int_0^\infty
\frac{ds}{1+\frac{s^q}{q-1}}
=
(q-1)^{1/q}\pi_q.
\]
Since $\widehat\pi_q\to2$ as $q\downarrow1$, letting $q\downarrow1$
gives \eqref{eq:weighted-L1-diameter}. The minimizing constants may be
chosen between $\operatorname*{ess\,inf}h$ and
$\operatorname*{ess\,sup}h$, so the passage to the limit on the
left-hand side is immediate.
\end{proof}

Recall that, for $V\equiv0$, we write $\Lambda_p(\Omega):=\Lambda_p(\Omega,0)$
and
\[C_{p,N}^{\mathrm{opt}}
:=
\sup_{\Omega\in\mathcal C_N}
\frac{
\lambda_{1,p}(\Omega,0)^{\frac{p-2}{p}}
}{
D^2\Lambda_p(\Omega)
}.\]

\subsection{The degenerate weighted Poincar\'e inequality and the enhanced fundamental gap}
\label{subsec:sharp-weighted-constant-enhanced-gap}

We first establish the strict lower bound needed in dimensions $N\geq2$.

\begin{lemma}
\label{lem:optimal-weighted-constant-strict-lower-bound}
Let \(N\geq2\) and \(p>2\). Set \(J:=(-1,1)\), and let
\(\varphi>0\) be the first Dirichlet \(p\)-eigenfunction on \(J\),
normalized by \(\int_J\varphi^p\,dz=1\). Then
\[
C_{p,N}^{\mathrm{opt}}
\geq
\frac{\lambda_{1,p}(J,0)^{\frac{p-2}{p}}}{3\pi^2A_p}
>
\frac1{3\pi^2},\qquad A_p:=\int_J|\varphi'|^{p-2}\varphi^2\,dz.
\]
\end{lemma}

\begin{proof}
Set \(\alpha:=(p-2)/p\) and
\(\lambda_J:=\lambda_{1,p}(J,0)\). For \(\delta>0\), let
\[
I:=\left(-\frac12,\frac12\right),
\qquad
D_\delta:=I\times(-\delta,\delta)^{N-2},
\qquad
Q_\delta:=D_\delta\times J,
\]
where \(D_\delta:=I\) when \(N=2\). The positive first Dirichlet eigenfunction of the Laplacian on
\(D_\delta\), normalized in \(L^2(D_\delta)\), is, by separation of
variables,
\[
G_\delta(x)
=
\sqrt2\cos(\pi x_1)
\prod_{k=2}^{N-1}
\delta^{-1/2}
\cos\left(\frac{\pi x_k}{2\delta}\right),
\]
with the empty product interpreted as \(1\). Fix \(\delta>0\). For \(\varepsilon>0\), set
\[
\Omega_\varepsilon:=D_\delta\times(-\varepsilon,\varepsilon),
\qquad
d_\varepsilon:=\operatorname{diam}(\Omega_\varepsilon).
\]
Then \(\Omega_\varepsilon\in\mathcal C_N\) and $d_\varepsilon^2
=
1+4(N-2)\delta^2+4\varepsilon^2.$
Let \(u_\varepsilon>0\) be the first Dirichlet \(p\)-eigenfunction
on \(\Omega_\varepsilon\), normalized by
\(\int_{\Omega_\varepsilon}u_\varepsilon^p\,dx\,dy=1\), and define
\[
v_\varepsilon(x,z)
:=
\varepsilon^{1/p}u_\varepsilon(x,\varepsilon z),
\qquad
\nabla_\varepsilon v
:=
(\varepsilon\nabla_xv,\partial_zv).
\]
By the proof of \cite[Theorem~1.5]{chen-two-point}, we obtain
\[
v_\varepsilon
\longrightarrow
G_\delta^{2/p}\varphi
\quad\text{strongly in }L^p(Q_\delta),
\qquad
\int_{Q_\delta}|\nabla_\varepsilon v_\varepsilon|^p\,dx\,dz
=\varepsilon^p\lambda_{1,p}(\Omega_\varepsilon,0)
\longrightarrow\lambda_J.
\]

We claim that
\[
\nabla_\varepsilon v_\varepsilon
\longrightarrow
Z_\delta:=(0,G_\delta^{2/p}\varphi')
\quad\text{strongly in }L^p(Q_\delta;\mathbb R^N).
\]
Indeed, the preceding energy bound gives weak compactness in \(L^p\).
For every \(\psi\in C_c^\infty(Q_\delta)\),
\[
\int_{Q_\delta}
\varepsilon\partial_{x_i}v_\varepsilon\,\psi\,dx\,dz
=
-\varepsilon\int_{Q_\delta}
v_\varepsilon\partial_{x_i}\psi\,dx\,dz
\longrightarrow0,
\]
and
\[
\int_{Q_\delta}\partial_zv_\varepsilon\,\psi\,dx\,dz
=
-\int_{Q_\delta}v_\varepsilon\partial_z\psi\,dx\,dz
\longrightarrow
\int_{Q_\delta}G_\delta^{2/p}\varphi'\psi\,dx\,dz.
\]
Thus every weakly convergent subsequence has limit \(Z_\delta\).
Since
\[
\int_{Q_\delta}|Z_\delta|^p\,dx\,dz
=
\int_{D_\delta}G_\delta^2\,dx
\int_J|\varphi'|^p\,dz
=
\lambda_J,
\]
the convergence of the norms and the uniform convexity of \(L^p\)
prove the claim. Consequently, by H\"older's inequality and the
continuity of the power maps,
\[
v_\varepsilon^p
\longrightarrow G_\delta^2\varphi^p,
\qquad
v_\varepsilon^2|\nabla_\varepsilon v_\varepsilon|^{p-2}
\longrightarrow
G_\delta^2\varphi^2|\varphi'|^{p-2}
\quad\text{in }L^1(Q_\delta).
\]

Take \(g(x,y):=\sin(\pi x_1)\). The domain \(\Omega_\varepsilon\)
is invariant under reflection across \(x_1=0\). By positivity and
simplicity of the first eigenfunction, \(u_\varepsilon\) is even
in \(x_1\), and hence
\[
g\in C^{0,1}(\overline{\Omega_\varepsilon}),
\qquad
\int_{\Omega_\varepsilon}g\,u_\varepsilon^p\,dx\,dy=0.
\]
Set
\[
M_\varepsilon
:=
\int_{\Omega_\varepsilon}g^2u_\varepsilon^p\,dx\,dy,
\qquad
E_\varepsilon
:=
\int_{\Omega_\varepsilon}
u_\varepsilon^2|\nabla u_\varepsilon|^{p-2}
|\nabla g|^2\,dx\,dy.
\]
The preceding \(L^1\)-convergences give
\[
\begin{aligned}
M_\varepsilon
&=
\int_{Q_\delta}g^2v_\varepsilon^p\,dx\,dz\longrightarrow
\int_{D_\delta}g^2G_\delta^2\,dx
=
2\int_{-1/2}^{1/2}
\sin^2(\pi x)\cos^2(\pi x)\,dx
=
\frac14
\end{aligned}
\]
and
\[
\begin{aligned}
\varepsilon^{p-2}E_\varepsilon
&=
\int_{Q_\delta}
v_\varepsilon^2|\nabla_\varepsilon v_\varepsilon|^{p-2}
|\nabla_xg|^2\,dx\,dz\longrightarrow
A_p\int_{D_\delta}G_\delta^2|\nabla_xg|^2\,dx\\
&=
2\pi^2A_p\int_{-1/2}^{1/2}\cos^4(\pi x)\,dx
=
\frac{3\pi^2A_p}{4}.
\end{aligned}
\]

Using \(g\) in the Rayleigh quotient defining
\(\Lambda_p(\Omega_\varepsilon)\), we obtain $\Lambda_p(\Omega_\varepsilon)
\leq \frac{E_\varepsilon}{M_\varepsilon}.$
Therefore, by the definition of \(C_{p,N}^{\mathrm{opt}}\),
\[
\begin{aligned}
C_{p,N}^{\mathrm{opt}}
&\geq
\frac{\lambda_{1,p}(\Omega_\varepsilon,0)^\alpha}
{d_\varepsilon^2\Lambda_p(\Omega_\varepsilon)}\geq
\frac{
\bigl(\varepsilon^p\lambda_{1,p}(\Omega_\varepsilon,0)\bigr)^\alpha
M_\varepsilon
}{
d_\varepsilon^2\varepsilon^{p-2}E_\varepsilon
}.
\end{aligned}
\]
Letting first \(\varepsilon\downarrow0\), with \(\delta\) fixed, yields
\[
C_{p,N}^{\mathrm{opt}}
\geq
\frac{\lambda_J^\alpha}
{3\pi^2A_p\bigl(1+4(N-2)\delta^2\bigr)}.
\]
Letting \(\delta\downarrow0\), we obtain $C_{p,N}^{\mathrm{opt}}
\geq
\frac{\lambda_J^\alpha}{3\pi^2A_p}.$
Finally, H\"older's inequality gives
\[
A_p
\leq
\left(\int_J|\varphi'|^p\,dz\right)^{\frac{p-2}{p}}
\left(\int_J\varphi^p\,dz\right)^{\frac2p}
=
\lambda_J^\alpha.
\]
The inequality is strict. Indeed, equality would imply $|\varphi'|^p=c\varphi^p$ a.e. in $J$
for some \(c>0\). By continuity this identity would hold at \(z=0\),
which contradicts $\varphi'(0)=0,\,
\varphi(0)>0.$
Hence $\lambda_J^\alpha>A_p,$
and therefore
\[
C_{p,N}^{\mathrm{opt}}
\geq
\frac{\lambda_J^\alpha}{3\pi^2A_p}
>
\frac1{3\pi^2}.
\]
This completes the proof.
\end{proof}

We can now prove the degenerate weighted Poincar\'e inequality.

\begin{proof}[\textbf{Proof of Theorem~\ref{thm:ground-state-weighted-spectral-gap}.}]
We first prove the uniform estimate. Choose a median \(m\in\mathbb R\) of \(g\), so that $\mu(\{g>m\})\leq\frac12,\,
\mu(\{g<m\})\leq\frac12,$
and set
\[
G:=g-m,
\qquad
G_+:=\max\{G,0\},
\qquad
G_-:=\max\{-G,0\}.
\]
Then $\mu(\{G_\pm^2=0\})\geq\frac12,$
hence \(0\) is a median of \(G_\pm^2\). Since a median minimizes the
\(L^1(\mu)\)-distance to constants, Lemma~\ref{lem:weighted-L1-diameter}
applied to \(G_\pm^2\) gives
\[
\int_\Omega G_\pm^2\,d\mu
\leq
D\int_\Omega G_\pm|\nabla G_\pm|\,d\mu.
\]
Adding the two inequalities yields $\int_\Omega G^2\,d\mu
\leq
D\int_\Omega |G||\nabla G|\,d\mu.$
Set
\[
M:=\int_\Omega G^2\,d\mu,
\qquad
E:=\int_\Omega r^{p-2}|\nabla G|^2\,d\mu,
\qquad
I:=\int_\Omega r^{2-p}G^2\,d\mu.
\]
By Cauchy--Schwarz,
\[
M
\leq
D E^{1/2}I^{1/2},
\qquad\text{and hence}\qquad
M^2\leq D^2EI.
\]
For \(p>2\), Proposition~\ref{prop:ground-state-hardy} gives
\[
M^2
\leq
4(p-1)^2
D^2\lambda^{-\frac{p-2}{p}}
E\left(M+\lambda^{-1}E\right).
\]

Let $A:=D^2\lambda^{-\frac{p-2}{p}}E.$
By Lemma~\ref{lem:first-eigenvalue-diameter-lower-bound}, $D^2\lambda^{2/p}
\geq
(p-1)^{2/p}\pi_p^2
\geq4,$
and hence
\[
\lambda^{-1}E
=
\frac{A}{D^2\lambda^{2/p}}
\leq
\frac{A}{4}.
\]
It follows that $M^2
\leq
4(p-1)^2AM
+
(p-1)^2A^2.$
If \(A=0\), then \(M=0\). Otherwise, dividing by \(A^2\) and setting $y:=\frac{M}{A},$
we obtain $y^2
\leq
4(p-1)^2y+(p-1)^2.$
Thus
\[
\begin{aligned}
y
&\leq
2(p-1)^2
+
\sqrt{4(p-1)^4+(p-1)^2}
<
5(p-1)^2.
\end{aligned}
\]
Consequently, $M
\leq
5(p-1)^2
D^2\lambda^{-\frac{p-2}{p}}
E.$
Therefore
\[
\operatorname{Var}_{\mu}(g)
\leq
\int_\Omega |g-m|^2\,d\mu
=
M
\leq
5(p-1)^2D^2\lambda^{-\frac{p-2}{p}}
\int_\Omega r^{p-2}|\nabla g|^2\,d\mu.
\]
In particular, for \(\int_\Omega g\,d\mu=0\),
\[
\Lambda_p(\Omega)
\geq
\frac{1}{5(p-1)^2}
D^{-2}\lambda^{\frac{p-2}{p}},
\]
and therefore $C_{p,N}^{\mathrm{opt}}
\leq
5(p-1)^2.$
By the definition of \(C_{p,N}^{\mathrm{opt}}\), $\Lambda_p(\Omega)
\geq
\frac{1}{C_{p,N}^{\mathrm{opt}}}
D^{-2}\lambda^{\frac{p-2}{p}},$
which is equivalent, after replacing \(g\) by
\(g-\int_\Omega g\,d\mu\), to
\[
\operatorname{Var}_{\mu}(g)
\leq
C_{p,N}^{\mathrm{opt}}D^2
\lambda^{-\frac{p-2}{p}}
\int_\Omega r^{p-2}|\nabla g|^2\,d\mu.
\]
\medskip
\noindent
\textbf{(ii)}
Let \(p=2\). For \(w=u_1g\), the ground-state identity gives
\[
\int_\Omega u_1^2|\nabla g|^2\,dx
=
\int_\Omega |\nabla w|^2\,dx
-
\lambda_{1,2}(\Omega,0)\int_\Omega w^2\,dx.
\]
Moreover, $\int_\Omega g\,d\mu=0
\Longleftrightarrow
\int_\Omega wu_1\,dx=0.$
Hence, by the variational characterization of the second Dirichlet
eigenvalue and a standard density argument,
\[
\begin{aligned}
\Lambda_2(\Omega)
&=
\inf_{\substack{w\in H_0^1(\Omega)\setminus\{0\}\\
\int_\Omega wu_1\,dx=0}}
\left(
\frac{\displaystyle\int_\Omega|\nabla w|^2\,dx}
{\displaystyle\int_\Omega w^2\,dx}
-\lambda_{1,2}(\Omega,0)
\right)=
\lambda_{2,2}(\Omega,0)-\lambda_{1,2}(\Omega,0).
\end{aligned}
\]
By \cite[Theorem 1]{AmatoBucurFragala2024} and \cite{Lavine1994}, we obtain that $\Gamma_2(\Omega,0)
\geq
\frac{3\pi^2}{D^2}.$
Thus \(C_{2,N}^{\mathrm{opt}}\leq1/(3\pi^2)\). Equality follows from the
one-dimensional interval, for which $\Gamma_2(I_D,0)=\frac{3\pi^2}{D^2}.$

\medskip
\noindent
\textbf{(iii)}
Let \(N=1\) and \(\Omega=I_D=(-D/2,D/2)\). Write \(u=u_1\) and set
\[
q:=u^2|u'|^{p-2},
\qquad
\rho:=u^p,
\qquad
b:=u'.
\]
Since $r^{p-2}\,d\mu
=
q\,dx,$
the quantity \(\Lambda_p(I_D)\) is precisely the first nonzero
eigenvalue associated with the weighted Neumann form $\phi\mapsto\int_{I_D}q|\phi'|^2\,dx$
in \(L^2(I_D,\rho\,dx)\). On each of the two intervals where \(u'\neq0\), the eigenvalue equation gives
\[
-(p-1)|u'|^{p-2}u''
=
\lambda_{1,p}(I_D,0)u^{p-1},
\]
and hence
\[
qb'
=
u^2|u'|^{p-2}u''
=
-\frac{\lambda_{1,p}(I_D,0)}{p-1}u^{p+1}.
\]
By \eqref{eq:first eigen}, we obtain
\[
-(qb')'
=
\frac{p+1}{p-1}\lambda_{1,p}(I_D,0)\rho b
=
(p+1)\left(\frac{\pi_p}{D}\right)^p\rho b.
\]
Moreover,
\[
qb'=0\quad\text{at }x=\pm\frac D2,
\qquad
\int_{I_D}\rho b\,dx=0.
\]
The function \(b=u'\) belongs to the natural form domain of $\mathcal Q[\phi]:=\int_{I_D}q|\phi'|^2\,dx.$
Since \(u\) is strictly increasing on \((-D/2,0)\) and strictly
decreasing on \((0,D/2)\), \(b\) has exactly one zero in \(I_D\).
Hence the oscillation theorem for singular Sturm--Liouville problems with separated boundary
conditions \cite[Theorem~5.3]{NiessenZettl1992}  shows that \(b\)
corresponds to the first nonzero eigenvalue $\lambda_{N,1}(q,\rho)$ of the associated weighted
Neumann problem, namely, $\lambda_{N,1}(q,\rho)
=
(p+1)\left(\frac{\pi_p}{D}\right)^p.$

It remains to identify this value with \(\Lambda_p(I_D)\), whose
admissible class consists of Lipschitz functions. This is immediate for
\(p=2\). Let \(p>2\) and set \(\alpha:=1/(p-1)\). Near \(x=0\),
\[
|b(x)|\asymp |x|^\alpha,
\qquad
q(x)\asymp |x|^{1-\alpha},
\qquad
|b'(x)|\asymp |x|^{\alpha-1}.
\]
For \(\delta>0\), let \(b_\delta=b\) on
\(I_D\setminus(-\delta,\delta)\) and extend it linearly and oddly on
\((-\delta,\delta)\). Then $b_\delta\in C^{0,1}(\overline{I_D}),\,\int_{I_D}\rho b_\delta\,dx=0,$
and
\[
b_\delta\to b
\quad\text{in }L^2(I_D,\rho\,dx),
\qquad
\int_{I_D}q|b_\delta'|^2\,dx
\longrightarrow
\int_{I_D}q|b'|^2\,dx,
\]
since
\[
\int_{-\delta}^{\delta}q|b'|^2\,dx
+
\int_{-\delta}^{\delta}q|b_\delta'|^2\,dx
\lesssim \delta^\alpha\to0.
\]
Thus $\Lambda_p(I_D)
\leq
\lambda_{N,1}(q,\rho).$
The reverse inequality follows because
\(C^{0,1}(\overline{I_D})\) is contained in the natural form domain.
Consequently, $\Lambda_p(I_D)
=
(p+1)\left(\frac{\pi_p}{D}\right)^p.$
By the definition of \(C_{p,1}^{\mathrm{opt}}\),
\[
\begin{aligned}
C_{p,1}^{\mathrm{opt}}
&\geq
\frac{\lambda_{1,p}(I_D,0)^{\frac{p-2}{p}}}
{D^2\Lambda_p(I_D)}
=
\frac{(p-1)^{\frac{p-2}{p}}}
{(p+1)\pi_p^2}=:\ell(p).
\end{aligned}
\]
 
Since
\[
\frac{d}{dp}\log \ell(p)
=
\frac{2}{p^2}
\left(
\log(p-1)+p-\frac{p}{p^2-1}
-\pi\cot\frac{\pi}{p}
\right)>0,
\qquad p>2,
\]
where we used
\(
x\cot x\leq 1-\frac{x^2}{3}
\)
for \(0<x\leq\pi/2\), the function \(\ell\) is strictly increasing on
\([2,\infty)\). Hence
\[
C_{p,1}^{\mathrm{opt}}
\geq
\ell(p)
>
\ell(2)
=
\frac{1}{3\pi^2}
=
C_{2,1}^{\mathrm{opt}},
\qquad p>2.
\]
This completes the proof.
\end{proof}

We finish with the enhanced gap estimate.

\begin{proof}[\textbf{Proof of Corollary~\ref{cor:enhanced-fundamental-gap}.}]
By Theorem~\ref{thm:dimension-free-higher-dimensional-gap}(ii), $\Gamma_p(\Omega,0)
\geq
\frac{2}{p}\Lambda_p(\Omega).$
Theorem~\ref{thm:ground-state-weighted-spectral-gap} gives
\[
\Lambda_p(\Omega)
\geq
\frac{1}{C_{p,N}^{\mathrm{opt}}}
\lambda_{1,p}(\Omega,0)^{\frac{p-2}{p}}D^{-2}.
\]
Therefore
\[
\Gamma_p(\Omega,0)
\geq
\frac{2}{p\,C_{p,N}^{\mathrm{opt}}}
\lambda_{1,p}(\Omega,0)^{\frac{p-2}{p}}D^{-2},
\]
which is \eqref{eq:enhanced-fundamental-gap}. By Lemma~\ref{lem:first-eigenvalue-diameter-lower-bound},
\[
\lambda_{1,p}(\Omega,0)
\geq
(p-1)\left(\frac{\pi_p}{D}\right)^p.
\]
Hence
\[
\Gamma_p(\Omega,0)
\geq
\frac{2}{p\,C_{p,N}^{\mathrm{opt}}}
(p-1)^{\frac{p-2}{p}}
\pi_p^{p-2}D^{-p}.
\]
Taking the infimum over $\Omega\in\mathcal C_N$ after multiplying by
$D^p$ yields
\[
\mathcal G_{p,N}
\geq
\frac{2}{p\,C_{p,N}^{\mathrm{opt}}}
(p-1)^{\frac{p-2}{p}}
\pi_p^{p-2}.
\]

Finally, let $p>2$ and let $\Omega_\varepsilon$ be the collapsing family
of Theorem~\ref{thm:smooth-quadratic-threshold-collapse}. From the
transverse Dirichlet Poincar\'e inequality used in the proof of that
result, $\lambda_{1,p}(\Omega_\varepsilon,0)
\gtrsim
\varepsilon^{-p},$
while $\operatorname{diam}(\Omega_\varepsilon)\to D>0.$
Consequently,
\[
\Gamma_p(\Omega_\varepsilon,0)
\gtrsim
\lambda_{1,p}(\Omega_\varepsilon,0)^{\frac{p-2}{p}}
\gtrsim
\varepsilon^{2-p}
\longrightarrow+\infty
\qquad\text{as }\varepsilon\downarrow0,
\]
which completes the proof.
\end{proof}

\section{Appendix: Regularity and Regularization}
\label{app:logconca-technical}

In this appendix, we collect the regularity, boundary nondegeneracy,
and regularization results used in Section~\ref{sec:logconca}. We work
under the assumptions and notation introduced at the beginning of that
section.

\subsection{Regularity of the First Eigenfunction}

We begin with the global \(C^{1,\beta}\)-regularity of the positive
first eigenfunction and its higher interior regularity away from the
critical set.

\begin{proposition}
\label{prop:first-eigenfunction-regularity}
Let \(u\) be the positive first eigenfunction. Then there exist
\(\beta\in(0,1)\) and \(C>0\) such that $u\in C^{1,\beta}(\overline\Omega)$ and $\|u\|_{C^{1,\beta}(\overline\Omega)}\leq C.$
Moreover, setting
\begin{equation}\label{eq:zeroset}
    \mathcal Z_u:=\{x\in\Omega:\nabla u(x)=0\},
\end{equation}
for every $0<\gamma<\min\{\alpha_V,\beta\},$
one has $u\in C_{\mathrm{loc}}^{2,\gamma}
\bigl(\Omega\setminus\mathcal Z_u\bigr).$
In particular, the equation
\[
-\Delta_pu+V(x)u^{p-1}
=
\lambda_{1,p}(\Omega,V)u^{p-1}
\]
holds pointwise in \(\Omega\setminus\mathcal Z_u\).
\end{proposition}

\begin{proof}
The global \(C^{1,\beta}\)-regularity follows from
\cite[Theorem~1]{otani-teshima}, while the corresponding estimate is
provided by Lieberman's global boundary regularity theorem
\cite[Theorem~1]{lieberman-boundary-regularity}.

Fix \(U\Subset\Omega\setminus\mathcal Z_u\), and choose an open set
\(U_1\) such that $\overline U\subset U_1\Subset\Omega\setminus\mathcal Z_u.$
By continuity of \(\nabla u\), there exist constants
\(\delta_U,M_U>0\) such that
\[
0<\delta_U\leq|\nabla u|\leq M_U
\qquad\text{in }U_1.
\]

Set $h:=
\bigl(\lambda_{1,p}(\Omega,V)-V\bigr)u^{p-1}.$
Since \(u\in C^{1,\beta}(\overline\Omega)\) and
\(V\in C^{0,\alpha_V}(\overline\Omega)\), one has
\(h\in C(\overline{U_1})\). The weak eigenvalue equation is therefore $-\Delta_pu=h$ in $U_1.$
By the equivalence between weak and viscosity solutions for the
nonhomogeneous \(p\)-Laplace equation
\cite{julin-juutinen}, \(u\) is also a viscosity solution of this
equation. Since \(|\nabla u|>0\) in \(U_1\), the equation can be divided by
\(|\nabla u|^{p-2}\) and rewritten, in the viscosity sense, as
\[
-a_{ij}(x)u_{ij}=f(x)
\qquad\text{in }U_1,
\]
where
\[
a_{ij}(x)
=
\delta_{ij}
+
(p-2)\frac{u_i(x)u_j(x)}{|\nabla u(x)|^2},
\qquad
f(x)
=
\bigl(\lambda_{1,p}(\Omega,V)-V(x)\bigr)
u(x)^{p-1}|\nabla u(x)|^{2-p}.
\]

Writing $e:=\frac{\nabla u}{|\nabla u|},$
we have $a_{ij}
=
\delta_{ij}+(p-2)e_i e_j.$
Hence
\[
\min\{1,p-1\}|\xi|^2
\leq
a_{ij}(x)\xi_i\xi_j
\leq
\max\{1,p-1\}|\xi|^2
\]
for every \(x\in U_1\) and \(\xi\in\mathbb R^N\), so
\(A=(a_{ij})\) is uniformly elliptic. Since
\(\overline{U_1}\Subset\Omega\setminus\mathcal Z_u\), $\inf_{\overline{U_1}}|\nabla u|>0,$
and therefore
\[
a_{ij}\in C^{0,\beta}(\overline{U_1}),
\qquad
f\in C^{0,\theta}(\overline{U_1}),
\qquad
\theta:=\min\{\alpha_V,\beta\}.
\]

Define
\[
F(M,x):=a_{ij}(x)M_{ij},
\qquad
M=(M_{ij})\in\mathbb S^N,
\]
where
\[
\mathbb S^N
:=
\{M\in\mathbb R^{N\times N}:M^T=M\}
\]
denotes the space of real symmetric \(N\times N\) matrices.
Then \(u\) is a viscosity solution of
\[
F(D^2u,x)=-f(x)
\qquad\text{in }U_1.
\]
The operator \(F\) is uniformly elliptic, its dependence on \(x\) is
\(C^{0,\beta}\), and each frozen operator \(F(\,\cdot\,,x_0)\) is a
constant-coefficient linear uniformly elliptic operator. After the
standard normalization at \(x_0\), the assumptions of
\cite[Theorem~8.1]{caffarelli-cabre} are satisfied. Hence, for every
\(U\Subset U_1\),
\[
u\in C^{2,\gamma}(\overline U)
\qquad
\text{for every }
0<\gamma<\min\{\alpha_V,\beta\}.
\]
Since \(U\Subset\Omega\setminus\mathcal Z_u\) is arbitrary, the
eigenvalue equation holds classically in
\(\Omega\setminus\mathcal Z_u\).

Alternatively, the same conclusion follows from a difference-quotient
argument. Since \(|\nabla u|\) is bounded away from zero in \(U_1\),
the linearized coefficients are uniformly elliptic and belong to
\(C^{0,\beta}\). Applying the interior Schauder estimate
\cite[Theorem~8.32]{gilbarg-trudinger} to the difference quotients of
\(u\), and then letting the increment tend to zero, yields $u\in C_{\mathrm{loc}}^{2,\gamma}(U_1)$ for every $0<\gamma<\min\{\alpha_V,\beta\}.$
\end{proof}

We next establish the uniform boundary nondegeneracy and the resulting
higher regularity in a collar neighborhood of \(\partial\Omega\).

\begin{proposition}
\label{prop:first-eigenfunction-boundary-regularity}
Let \(u\) be the positive first
eigenfunction, and let \(\nu\) denote the outward unit normal to
\(\partial\Omega\). Then there exist constants
\(\eta_0,\delta_0,\eta>0\) and \(\gamma\in(0,1)\) such that
\[
-\frac{\partial u}{\partial\nu}(y)\geq\eta_0
\quad\text{for every }y\in\partial\Omega,
\qquad
|\nabla u(x)|\geq\eta
\quad\text{for every }x\in\mathcal N_{\delta_0},
\]
where
\[
\mathcal N_\delta
:=
\bigl\{
x\in\overline\Omega:
\operatorname{dist}(x,\partial\Omega)<\delta
\bigr\}.
\]
Moreover, $u\in C^{2,\gamma}
\left(
\overline{\mathcal N_{\delta_0}\cap\Omega}
\right).$
\end{proposition}

\begin{proof}
By the Hopf boundary point lemma; see
\cite[Lemma~4]{otani-teshima} or
\cite[Theorem~2.2]{esposito-farina-montoro-sciunzi}, we obtain
\[
-\frac{\partial u}{\partial\nu}(y)>0
\qquad
\text{for every }y\in\partial\Omega.
\]
The map \(y\mapsto-\partial_\nu u(y)\) is continuous on the compact set
\(\partial\Omega\); hence
\[
\eta_0
:=
\min_{y\in\partial\Omega}
\left(
-\frac{\partial u}{\partial\nu}(y)
\right)>0.
\]
In particular, \(|\nabla u|\geq\eta_0\) on \(\partial\Omega\).
The continuity of \(\nabla u\) and the compactness of
\(\partial\Omega\) therefore yield constants \(\delta_*>0\) and
\(\eta_*>0\) such that
\[
|\nabla u(x)|\geq\eta_*
\qquad\text{for every }x\in\mathcal N_{\delta_*}.
\]

Choose \(0<\delta_0<\delta_1<\delta_*\) and set
\(\mathcal G_1:=\mathcal N_{\delta_1}\cap\Omega\). Since
\(\mathcal G_1\subset\Omega\setminus\mathcal Z_u\),
Proposition~\ref{prop:first-eigenfunction-regularity}, applied locally,
gives \(u\in C^2(\mathcal G_1)\), and the eigenvalue equation holds
pointwise there. Thus
\[
-a_{ij}(x)u_{ij}=f(x)
\qquad\text{in }\mathcal G_1,
\]
where
\[
a_{ij}
:=
\delta_{ij}
+
(p-2)\frac{u_i u_j}{|\nabla u|^2},
\qquad
f
:=
(\lambda_1-V)u^{p-1}|\nabla u|^{2-p}.
\]
Writing \(e:=\nabla u/|\nabla u|\), we have
\(a_{ij}=\delta_{ij}+(p-2)e_i e_j\), and hence
\[
\min\{1,p-1\}|\xi|^2
\leq
a_{ij}(x)\xi_i\xi_j
\leq
\max\{1,p-1\}|\xi|^2
\]
for every \(x\in\mathcal G_1\) and \(\xi\in\mathbb R^N\). Thus
\((a_{ij})\) is uniformly elliptic.

Since \(u\in C^{1,\beta}(\overline\Omega)\) and
\(|\nabla u|\geq\eta_*>0\) on \(\overline{\mathcal G_1}\), we have
\[
\frac{\nabla u}{|\nabla u|},\ |\nabla u|^{2-p}
\in C^{0,\beta}(\overline{\mathcal G_1}),
\qquad
a_{ij}\in C^{0,\beta}(\overline{\mathcal G_1}).
\]
Moreover, with \(\vartheta_p:=\min\{1,p-1\}\), $u^{p-1}\in C^{0,\vartheta_p}(\overline\Omega).$
Since \(V\in C^{0,\alpha_V}(\overline\Omega)\), it follows that
\[
f\in C^{0,\sigma}(\overline{\mathcal G_1}),
\qquad
\sigma:=\min\{\alpha_V,\beta,\vartheta_p\}.
\]
Choose $0<\gamma<\min\{\alpha_\Omega,\alpha_V,\beta,\vartheta_p\}.$
Then \(a_{ij},f\in C^{0,\gamma}(\overline{\mathcal G_1})\) and
\(\partial\Omega\in C^{2,\gamma}\). The local
interior and boundary regularity results
\cite[Lemmas~6.16 and Lemma~6.18]{gilbarg-trudinger}, together with a finite
covering argument, yield $u\in C^{2,\gamma}
\left(
\overline{\mathcal N_{\delta_0}\cap\Omega}
\right).$
\end{proof}

\subsection{Regularity and convergence of the regularized eigenfunctions}

For \(\varepsilon>0\), define the regularized energy
\[
\mathcal E_{\varepsilon,V}(w)
:=
\int_\Omega
\bigl(\varepsilon w^2+|\nabla w|^2\bigr)^{p/2}\,dx
+
\int_\Omega V(x)|w|^p\,dx
\]
and the associated first variational eigenvalue
\[
\lambda_\varepsilon
:=
\inf_{w\in\mathcal K}\mathcal E_{\varepsilon,V}(w),
\qquad
\mathcal K=
\left\{
w\in W_0^{1,p}(\Omega):
\|w\|_{L^p(\Omega)}=1
\right\}.
\]

We first construct the positive first eigenpair for the regularized
energy and establish its convergence to the original first eigenpair as
\(\varepsilon\downarrow0\).

\begin{proposition}
\label{prop:regularized-first-eigenpair}
Let \(u\in\mathcal K\) be the positive first eigenfunction of the problem \eqref{eq:weighted-p-laplacian-eigenvalue}. For every \(\varepsilon>0\), there exists a nonnegative minimizer
\(u_\varepsilon\in\mathcal K\) such that $\mathcal E_{\varepsilon,V}(u_\varepsilon)=\lambda_\varepsilon,$ and \(u_\varepsilon\) satisfies
\begin{equation}\label{eq:regularized-eigenvalue-weak}
\begin{aligned}
&\int_\Omega
\bigl(\varepsilon u_\varepsilon^2+|\nabla u_\varepsilon|^2\bigr)^{\frac{p-2}{2}}
\bigl(
\nabla u_\varepsilon\cdot\nabla\varphi
+\varepsilon u_\varepsilon\varphi
\bigr)\,dx
+\int_\Omega Vu_\varepsilon^{p-1}\varphi\,dx
=
\lambda_\varepsilon
\int_\Omega u_\varepsilon^{p-1}\varphi\,dx
\end{aligned}
\end{equation}
for every \(\varphi\in W_0^{1,p}(\Omega)\). Furthermore, we obtain $\lambda_{1,p}(\Omega,V) \leq \lambda_\varepsilon$,
and as $\varepsilon \downarrow 0$,
\[
\lambda_\varepsilon \rightarrow \lambda_{1,p}(\Omega,V) 
\quad \text{and} \quad 
u_\varepsilon \rightarrow u \quad \text{strongly in } W_0^{1,p}(\Omega).
\]
\end{proposition}

\begin{proof}
The existence of a minimizer follows from the direct method of the
calculus of variations, using the convexity of $(s,\xi)\longmapsto
\bigl(\varepsilon s^2+|\xi|^2\bigr)^{p/2}$
and the compact embedding
\(W_0^{1,p}(\Omega)\hookrightarrow L^p(\Omega)\). Since
\(\mathcal E_{\varepsilon,V}(|w|)
=\mathcal E_{\varepsilon,V}(w)\), the minimizer may be chosen
nonnegative. The Lagrange multiplier theorem then gives
\eqref{eq:regularized-eigenvalue-weak} with a multiplier
\(\mu_\varepsilon\); testing with \(u_\varepsilon\) and using
\(\|u_\varepsilon\|_{L^p(\Omega)}=1\) yields $\mu_\varepsilon
=
\mathcal E_{\varepsilon,V}(u_\varepsilon)
=
\lambda_\varepsilon.$

Since $\mathcal E_{\varepsilon,V}(w)\geq\mathcal E_V(w)$ for every $w\in\mathcal K,$
we have \(\lambda_\varepsilon\geq\lambda_{1,p}(\Omega,V)\). On the other hand,
\[
\lambda_\varepsilon
\leq
\mathcal E_{\varepsilon,V}(u)
\rightarrow
\mathcal E_V(u)
=
\lambda_{1,p}(\Omega,V)
\]
by dominated convergence. Hence $\lambda_\varepsilon\rightarrow\lambda_{1,p}(\Omega,V).$
Moreover,
\[
\lambda_{1,p}(\Omega,V)
\leq
\mathcal E_V(u_\varepsilon)
\leq
\mathcal E_{\varepsilon,V}(u_\varepsilon)
=
\lambda_\varepsilon,
\]
so that $\mathcal E_V(u_\varepsilon)\rightarrow\lambda_{1,p}(\Omega,V),$ thus,
the family \(\{u_\varepsilon\}_{0<\varepsilon\leq1}\) is bounded in
\(W_0^{1,p}(\Omega)\). Given any sequence
\(\varepsilon_j\downarrow0\), after passing to a subsequence,
\[
u_{\varepsilon_j}\rightharpoonup u_*
\quad\text{in }W_0^{1,p}(\Omega),
\qquad
u_{\varepsilon_j}\to u_*
\quad\text{in }L^p(\Omega).
\]
Thus \(u_*\in\mathcal K\) and \(u_*\geq0\). By weak lower
semicontinuity, $\mathcal E_V(u_*)\le\lambda_{1,p}(\Omega,V).$
Hence \(u_*\) is a nonnegative normalized first eigenfunction, and the
simplicity of \(\lambda_1\) implies \(u_*=u\).

Finally,
\[
\int_\Omega|\nabla u_{\varepsilon_j}|^p\,dx
=
\mathcal E_V(u_{\varepsilon_j})
-
\int_\Omega Vu_{\varepsilon_j}^p\,dx
\longrightarrow
\int_\Omega|\nabla u|^p\,dx.
\]
Together with the weak convergence of the gradients and the uniform
convexity of \(L^p(\Omega;\mathbb R^N)\), this gives $u_{\varepsilon_j}\rightarrow u$ strongly in  $W_0^{1,p}(\Omega).$
Since the limit is unique, the convergence holds for the whole family.
\end{proof}

Below, we use the lifting
\[
U_\varepsilon(x,t)
:=
e^{\sqrt{\varepsilon}\,t}u_\varepsilon(x)
\]
to prove the boundary nondegeneracy and
\(C^{1,\theta_\varepsilon}\)-regularity of \(u_\varepsilon\). These
properties remove the degeneracy of the original equation and allow us
to further obtain the global \(C^{2,\beta_\varepsilon}\)-regularity.

\begin{proposition}
\label{prop:regularized-eigenfunction-positivity-regularity}
Let \(u_\varepsilon\in\mathcal K\) be the nonnegative minimizer given by
Proposition~\ref{prop:regularized-first-eigenpair}, and let \(\nu\)
denote the outward unit normal to \(\partial\Omega\). Then
\[
u_\varepsilon>0 \quad\text{in }\Omega,
\qquad
-\frac{\partial u_\varepsilon}{\partial\nu}>0
\quad\text{on }\partial\Omega.
\]
Moreover, there exists \(\beta_\varepsilon\in(0,1)\) such that $u_\varepsilon\in C^{2,\beta_\varepsilon}(\overline\Omega).$
\end{proposition}

\begin{proof}
For simplicity, write $u:=u_\varepsilon,\,\lambda:=\lambda_\varepsilon,\,S:=S_\varepsilon(u):=
\bigl(\varepsilon u^2+|\nabla u|^2\bigr)^{1/2}.$
Then
\begin{equation}\label{eq:regularized-equation-regularity}
-\operatorname{div}\bigl(S^{p-2}\nabla u\bigr)
+\varepsilon S^{p-2}u
=
(\lambda-V)u^{p-1}
\end{equation}
weakly in \(\Omega\).

We first prove \(u\in L^\infty(\Omega)\). Suppose \(1<p\leq N\), and
choose \(\chi>1\) such that
\[
\|z\|_{L^{p\chi}(\Omega)}
\leq C_S\|\nabla z\|_{L^p(\Omega)}
\qquad
\text{for every }z\in W_0^{1,p}(\Omega);
\]
one may take \(\chi=N/(N-p)\) if \(p<N\), while any fixed
\(\chi>1\) is admissible if \(p=N\). For \(r\geq p\), set
\(u_m:=\min\{u,m\}\) and test \eqref{eq:regularized-equation-regularity}
with \(u\,u_m^{r-p}\). Since
\[
\nabla u\cdot\nabla\bigl(u\,u_m^{r-p}\bigr)
+\varepsilon u^2u_m^{r-p}
\geq
u_m^{r-p}\bigl(|\nabla u|^2+\varepsilon u^2\bigr),
\]
we obtain
\[
\int_\Omega u_m^{r-p}|\nabla u|^p\,dx
\leq
Q_\varepsilon\int_\Omega u^r\,dx,
\qquad
Q_\varepsilon:=
1+\|(\lambda_\varepsilon-V)_+\|_{L^\infty(\Omega)}.
\]
Applying the Sobolev inequality to
\(z_m:=u\,u_m^{(r-p)/p}\), for which $|\nabla z_m|
\leq
\frac{r}{p}u_m^{(r-p)/p}|\nabla u|,$
and then letting \(m\to\infty\), gives
\[
\|u\|_{L^{\chi r}(\Omega)}
\leq
\left[
C_S^pQ_\varepsilon
\left(\frac{r}{p}\right)^p
\right]^{1/r}
\|u\|_{L^r(\Omega)}.
\]
Iteration with \(r_j=p\chi^j\) yields \(u\in L^\infty(\Omega)\); see
also \cite[Chapter~IV, Theorem~7.1]
{ladyzhenskaya-uraltseva}. If \(p>N\), testing with \(u\) gives
\[
\int_\Omega|\nabla u|^p\,dx
\leq
|\lambda_\varepsilon|+\|V\|_{L^\infty(\Omega)},
\]
and the Morrey embedding yields the same conclusion.

Set
\[
A_\varepsilon(z,\xi)
:=
\bigl(\varepsilon z^2+|\xi|^2\bigr)^{\frac{p-2}{2}}\xi,
\qquad
B_\varepsilon(x,z,\xi)
:=
(\lambda_\varepsilon-V(x))|z|^{p-2}z
-\varepsilon
\bigl(\varepsilon z^2+|\xi|^2\bigr)^{\frac{p-2}{2}}z.
\]
Then \eqref{eq:regularized-equation-regularity} takes the form
\[
\operatorname{div}A_\varepsilon(u,\nabla u)
+
B_\varepsilon(x,u,\nabla u)=0.
\]
For suitable constants depending only on
\(p,\varepsilon,\lambda_\varepsilon\), and
\(\|V\|_{L^\infty(\Omega)}\), one has
\[
|A_\varepsilon(z,\xi)|
\leq
C_\varepsilon\bigl(|\xi|^{p-1}+|z|^{p-1}\bigr),
\qquad
A_\varepsilon(z,\xi)\cdot\xi
\geq
|\xi|^p-C_\varepsilon|z|^p,
\]
and
\[
|B_\varepsilon(x,z,\xi)|
\leq
C_\varepsilon\bigl(|\xi|^{p-1}+|z|^{p-1}\bigr).
\]
Hence the structural assumptions \emph{(1.2)--(1.3)} of
\cite{trudinger-harnack} are satisfied with exponent \(p\) and
\(a_3=a_4=b_3=0\). Since \(u\in L^\infty(\Omega)\), the Harnack
inequality \cite[Theorem~1.1]{trudinger-harnack} applies. As
\(u\geq0\), \(u\not\equiv0\), and \(\Omega\) is connected, we conclude
that $u>0$ in $\Omega.$

To obtain the global \(C^{1,\theta_\varepsilon}\)-regularity, set
\[
a:=\sqrt{\varepsilon},
\qquad
U(x,t):=e^{at}u(x),
\qquad
\rho(t):=e^{-pat}.
\]
Then
\[
\nabla_{x,t}U=e^{at}(\nabla u,au),
\qquad
|\nabla_{x,t}U|
=
e^{at}\bigl(\varepsilon u^2+|\nabla u|^2\bigr)^{1/2},
\]
and a direct computation in the weak sense gives
\[
\operatorname{div}_{x,t}
\!\left(
\rho(t)|\nabla_{x,t}U|^{p-2}\nabla_{x,t}U
\right)
+
\rho(t)(\lambda-V(x))|U|^{p-2}U
=0
\]
in \(\Omega\times\mathbb R\), with
\(U=0\) on \(\partial\Omega\times\mathbb R\).

Fix \(T>0\). On \(\Omega\times(-T,T)\), the function \(\rho\) is smooth
and bounded above and below by positive constants. Hence
\[
\mathcal A((x,t),z,P):=\rho(t)|P|^{p-2}P,
\qquad
\mathcal B((x,t),z,P)
:=
\rho(t)(\lambda-V(x))|z|^{p-2}z
\]
satisfy the structural assumptions \emph{(0.3a)--(0.3d)} of
\cite[Theorem~1]{lieberman-boundary-regularity} with
\(m=p-2\) and \(\kappa=0\). Applying the corresponding local interior
and boundary estimates away from the top and bottom faces yields $U\in C^{1,\theta_\varepsilon}$
in a neighborhood of \(\overline\Omega\times\{0\}\), for some
\(\theta_\varepsilon\in(0,1)\). Restricting to \(t=0\), we obtain $u\in C^{1,\theta_\varepsilon}(\overline\Omega).$ Moreover, the exponent and the estimate may be chosen uniformly for
all sufficiently small \(\varepsilon\). Indeed, since
\(\{\lambda_\varepsilon\}\) and
\(\{\|u_\varepsilon\|_{L^\infty(\Omega)}\}\) are uniformly bounded and,
for fixed \(T>0\),
\[
0<c_T\leq \rho_\varepsilon(t)\leq C_T,
\qquad
|\rho_\varepsilon'(t)|\leq C_T
\qquad |t|\leq T,
\]
with constants independent of sufficiently small \(\varepsilon\), all
the structural constants in
\cite[Theorem~1]{lieberman-boundary-regularity} can be chosen uniformly
in \(\varepsilon\). Consequently, there exist
\(\theta\in(0,1)\) and \(C>0\), independent of sufficiently small
\(\varepsilon\), such that
\[
\|U_\varepsilon\|_{C^{1,\theta}
(\overline\Omega\times[-T/2,T/2])}
\leq C.
\]
Restricting to \(t=0\), we obtain $\|u_\varepsilon\|_{C^{1,\theta}(\overline\Omega)}
\leq C.$

We next prove the boundary nondegeneracy. Choose $c_0>\|(V-\lambda)_+\|_{L^\infty(\Omega)}.$
The lifted equation gives
\[
\operatorname{div}_{x,t}\!\left(
\rho|\nabla_{x,t}U|^{p-2}\nabla_{x,t}U
\right)
-c_0\rho U^{p-1}
\leq0
\qquad\text{in }\,\mathcal C_T:=\Omega\times(-T,T).
\]
This is of the form considered in
\cite[Theorem~5.5.1]{pucci-serrin-maximum-principle}, with
\[
a_{ij}(X,z)=\rho(t)\delta_{ij},
\qquad
\mathcal A(s)=s^{p-2},
\qquad
B(X,z,\xi)=-c_0\rho(t)z^{p-1}.
\]
Indeed,
\[
\lim_{s\downarrow0}
\frac{s\mathcal A'(s)}{\mathcal A(s)}
=p-2>-1.
\]
If \(p\neq2\), we choose \(T>0\) sufficiently small that
\[
\sqrt{\frac{\max_{[-T,T]}\rho}
{\min_{[-T,T]}\rho}}
=e^{paT}<\phi(p-2),\quad \phi(c):=\frac{2+c+2\sqrt{1+c}}{|c|},c\ne0,
\]
whereas for \(p=2\) no additional condition is required.

Setting $\rho_T:=\max_{[-T,T]}\rho,f_0(s):=c_0\rho_Ts^{p-1},$
we have $B(X,z,\xi)\geq-f_0(z),$
so that condition \emph{(B1)} holds with \(\kappa=0\). Moreover,
\[
H(t)=\frac{p-1}{p}t^p,
\qquad
F_0(s)=\frac{c_0\rho_T}{p}s^p,
\]
and hence $\int_{0^+}\frac{ds}{H^{-1}(F_0(s))}
=+\infty.$
Since \(U>0\) in \(\mathcal C_T\), \(U(y,0)=0\), and
\(\mathcal C_T\) satisfies the interior sphere condition at
\((y,0)\), the boundary point lemma yields $\partial_{(\nu(y),0)}U(y,0)<0.$
Finally,
\[
\partial_{(\nu(y),0)}U(y,0)
=\partial_\nu u(y),
\]
and therefore $-\partial_\nu u(y)>0$ for every $y\in\partial\Omega.$

By the boundary point lemma, $|\nabla u|>0$ on $\partial\Omega.$
Since \(u>0\) in \(\Omega\), it follows that
\[
S_\varepsilon(u)
:=
\bigl(\varepsilon u^2+|\nabla u|^2\bigr)^{1/2}>0
\qquad\text{on }\overline\Omega.
\]
As \(u\in C^1(\overline\Omega)\) and \(\overline\Omega\) is compact,
there exist constants \(c_\varepsilon,C_\varepsilon>0\) such that
\[
c_\varepsilon
\leq S_\varepsilon(u)
\leq C_\varepsilon
\qquad\text{on }\overline\Omega.
\]

Set
\[
a:=\sqrt{\varepsilon},
\qquad
U(x,t):=e^{at}u(x),
\qquad
\rho(t):=e^{-pat}.
\]
For every fixed \(\tau>0\),
\[
|\nabla_{x,t}U(x,t)|
=
e^{at}S_\varepsilon(u(x))
\geq e^{-a\tau}c_\varepsilon>0
\qquad\text{in }
\mathcal Q_\tau:=\Omega\times(-\tau,\tau).
\]
The lifted equation can be written as
\[
-\Delta_pU=H_\varepsilon(X)
\qquad\text{weakly in }\mathcal Q_\tau,
\]
where
\[
H_\varepsilon(X)
:=
(\lambda-V(x))U^{p-1}
+
\nabla_X\log\rho(t)\cdot
|\nabla_XU|^{p-2}\nabla_XU.
\]
Since \(U\in C^{1,\theta_\varepsilon}\) and
\(|\nabla_XU|\) is bounded away from zero, \(H_\varepsilon\) is
continuous. The equivalence between weak and viscosity solutions \cite{julin-juutinen}, applied to the continuous right-hand
side \(H_\varepsilon(X)\), therefore shows that \(U\) is a viscosity
solution of this equation.

Let \(\varphi\in C^2(\Omega)\) touch \(u\) from above or below at
\(x_0\in\Omega\), and define $\Phi(x,t):=e^{at}\varphi(x).$
Then \(\Phi\) touches \(U\) at \((x_0,0)\). Evaluating the viscosity
equation at this point shows that \(u\) is a viscosity solution of
\[
-a_{ij}(x)u_{ij}=f_\varepsilon(x)
\qquad\text{in }\Omega,
\]
where, with
\[
S:=\bigl(\varepsilon u^2+|\nabla u|^2\bigr)^{1/2},\quad
a_{ij}
=
\delta_{ij}
+
(p-2)\frac{u_i u_j}{S^2}
\]
and
\[
f_\varepsilon
=
(\lambda-V)u^{p-1}S^{2-p}
+
(p-2)\varepsilon u\frac{|\nabla u|^2}{S^2}
-\varepsilon u.
\]
Since \(S\geq c_\varepsilon>0\),
\[
\min\{1,p-1\}|\xi|^2
\leq
a_{ij}(x)\xi_i\xi_j
\leq
\max\{1,p-1\}|\xi|^2,
\]
and hence \((a_{ij})\) is uniformly elliptic. Moreover, setting
\[
\vartheta_p:=\min\{1,p-1\},
\qquad
\sigma_\varepsilon
:=
\min\{\alpha_V,\theta_\varepsilon,\vartheta_p\},
\]
we have $a_{ij}\in C^{0,\theta_\varepsilon}(\overline\Omega)$ and $f_\varepsilon\in C^{0,\sigma_\varepsilon}(\overline\Omega).$

Fix $0<\beta_\varepsilon
<
\min\{\alpha_\Omega,\sigma_\varepsilon\}.$
By \cite[Theorem~6.14]{gilbarg-trudinger}, the Dirichlet problem
\[
\begin{cases}
a_{ij}(x)w_{ij}=-f_\varepsilon(x),
&x\in\Omega,\\
w=0,
&x\in\partial\Omega,
\end{cases}
\]
has a unique solution $w\in C^{2,\beta_\varepsilon}(\overline\Omega).$

It remains to identify \(u\) with \(w\). Choose
\(R>\sup_{\overline\Omega}x_1\) and set $\psi(x):=e^R-e^{x_1}>0.$
Writing \(L:=-a_{ij}D_{ij}\), we have $L\psi=a_{11}e^{x_1}>0$ in $\overline\Omega.$
Hence, for every \(\delta>0\),
\[
L(w+\delta\psi)>f_\varepsilon,
\qquad
L(w-\delta\psi)<f_\varepsilon.
\]
Since \(u=w=0\) on \(\partial\Omega\),
\[
w-\delta\psi<u<w+\delta\psi
\qquad\text{on }\partial\Omega.
\]
If \(u-(w+\delta\psi)\) had a positive maximum at some
\(x_0\in\Omega\), then, after adding a constant to
\(w+\delta\psi\), we would obtain a \(C^2\) test function touching
\(u\) from above at \(x_0\). The viscosity subsolution inequality would
give
\[
L(w+\delta\psi)(x_0)\leq f_\varepsilon(x_0),
\]
a contradiction. Thus \(u\leq w+\delta\psi\). The viscosity
supersolution inequality similarly gives $w-\delta\psi\leq u.$
Letting \(\delta\downarrow0\), we conclude that \(u=w\). Therefore $u_\varepsilon\in C^{2,\beta_\varepsilon}(\overline\Omega).$
\end{proof}

As a consequence of the uniform regularity estimates above, we obtain
the following convergence result, which also gives higher-order
convergence away from the critical set of the limiting eigenfunction.

\begin{corollary}
\label{cor:regularized-uniform-convergence}
Let \(u_\varepsilon\in\mathcal K\) and \(u\in\mathcal K\) be as in
Proposition~\ref{prop:regularized-first-eigenpair}. Then there exist
\(\varepsilon_0>0\), \(\beta\in(0,1)\), and \(C>0\), independent of
\(\varepsilon\in(0,\varepsilon_0]\), such that $\|u_\varepsilon\|_{C^{1,\beta}(\overline\Omega)}
\leq C.$
Moreover, for every \(0<\gamma<\beta\),
\[
u_\varepsilon\longrightarrow u
\qquad\text{in }C^{1,\gamma}(\overline\Omega)
\quad\text{as }\varepsilon\downarrow0.
\]

For every compact set $K\Subset\Omega\setminus\mathcal Z_u,$
there exist \(\varepsilon_K>0\), \(\theta_K\in(0,1)\), and
\(C_K,c_K>0\), independent of
\(\varepsilon\in(0,\varepsilon_K]\), such that $|\nabla u_\varepsilon|\geq c_K$ on $K,$
and $\|u_\varepsilon\|_{C^{2,\theta_K}(K)}
\leq C_K.$
Consequently, for every \(0<\gamma<\theta_K\),
\[
u_\varepsilon\longrightarrow u
\qquad\text{in }C^{2,\gamma}(K).
\]
\end{corollary}

\begin{proof}
The uniform \(C^{1,\beta}\)-estimate follows from the uniform version
of the regularity argument in
Proposition~\ref{prop:regularized-eigenfunction-positivity-regularity}.
Together with $u_\varepsilon\rightarrow u$ strongly in $W_0^{1,p}(\Omega),$
the compact embedding
\[
C^{1,\beta}(\overline\Omega)
\hookrightarrow
C^{1,\gamma}(\overline\Omega),
\qquad 0<\gamma<\beta,
\]
gives $u_\varepsilon\rightarrow u$ in $C^{1,\gamma}(\overline\Omega).$ Now let \(K\Subset\Omega\setminus\mathcal Z_u\). Since $\min_K|\nabla u|>0,$
the \(C^1\)-convergence implies, after decreasing
\(\varepsilon_K\) if necessary, $|\nabla u_\varepsilon|\geq c_K>0$ on $K.$

Choose $K\Subset K'\Subset\Omega\setminus\mathcal Z_u.$
On \(K'\), the equations for \(u_\varepsilon\) are uniformly
nondegenerate and can be written in nondivergence form $-a_\varepsilon^{ij}(x)(u_\varepsilon)_{ij}
=
f_\varepsilon(x).$
The \(C^{1,\gamma}\)-convergence, the boundedness of
\(\{\lambda_\varepsilon\}\), and the uniform lower bound for
\(|\nabla u_\varepsilon|\) give, for some
\(\theta_K\in(0,1)\),
\[
\|a_\varepsilon^{ij}\|_{C^{0,\theta_K}(K')}
+
\|f_\varepsilon\|_{C^{0,\theta_K}(K')}
\leq C_K,
\]
with \(C_K\) independent of sufficiently small \(\varepsilon\).
The interior Schauder estimate
\cite[Theorem~6.2]{gilbarg-trudinger} therefore yields $\|u_\varepsilon\|_{C^{2,\theta_K}(K)}
\leq C_K.$
Hence, for every \(0<\gamma<\theta_K\), compactness of
\[
C^{2,\theta_K}(K)\hookrightarrow C^{2,\gamma}(K)
\]
shows that every sequence \(\varepsilon_j\downarrow0\) admits a
subsequence converging in \(C^{2,\gamma}(K)\). Its limit must be \(u\)
by \(C^1\)-convergence. Thus the whole family
converges: $u_\varepsilon\rightarrow u$ in $C^{2,\gamma}(K).$
\end{proof}

\section*{Acknowledgements}
AI tools were used to assist with certain calculations and aspects of the writing. All mathematical ideas, verification of the results, and final responsibility for the content rest entirely with the human author.

\bibliographystyle{amsplain}
\bibliography{references}

\end{document}

%% file: preamble.tex
\usepackage[T1]{fontenc}
\usepackage{microtype}
\usepackage{silence}
\usepackage{amsmath,amssymb,mathtools}
\usepackage{mathrsfs}
\usepackage{bm}

\usepackage{graphicx}
\usepackage{booktabs}

\usepackage{enumitem}

\usepackage[dvipsnames]{xcolor}
\definecolor{refblue}{RGB}{0,0,255}

\usepackage[
  colorlinks=true,
  linkcolor=refblue,
  citecolor=refblue,
  urlcolor=refblue,
  pdfborder={0 0 0}
]{hyperref}

\usepackage[nameinlink,noabbrev]{cleveref}

\numberwithin{equation}{section}

\theoremstyle{plain}
\newtheorem{theorem}{Theorem}[section]
\newtheorem{lemma}[theorem]{Lemma}
\newtheorem{proposition}[theorem]{Proposition}
\newtheorem{corollary}[theorem]{Corollary}

\theoremstyle{definition}

\theoremstyle{remark}
\newtheorem{remark}[theorem]{Remark}

\crefname{theorem}{Theorem}{Theorems}
\crefname{lemma}{Lemma}{Lemmas}
\crefname{proposition}{Proposition}{Propositions}
\crefname{corollary}{Corollary}{Corollaries}
\crefname{claim}{Claim}{Claims}
\crefname{definition}{Definition}{Definitions}
\crefname{assumption}{Assumption}{Assumptions}
\crefname{example}{Example}{Examples}
\crefname{remark}{Remark}{Remarks}
\crefname{equation}{equation}{equations}
\crefname{section}{Section}{Sections}